\documentclass[11pt]{article}
\usepackage{amssymb,amsmath,accents}
\usepackage{amscd}
\usepackage{amsfonts,amsthm,mathrsfs}
\usepackage{setspace} 
\usepackage{graphics}
\usepackage{latexsym}
\usepackage{color}
\usepackage{authblk}

\newtheorem{theorem}{Theorem}
\newtheorem{lemma}[theorem]{Lemma}
\newtheorem{proposition}[theorem]{Proposition}
\newtheorem{corollary}[theorem]{Corollary}

\theoremstyle{definition}

\newtheorem{remark}[theorem]{Remark}

\title{\textbf{The Helmholtz decomposition of a $BMO$ type vector field in a slightly perturbed half space}}

\author[1]{Yoshikazu Giga\thanks{labgiga@ms.u-tokyo.ac.jp}}
\author[1,2]{Zhongyang Gu\thanks{zgu@ms.u-tokyo.ac.jp}}  
\affil[1]{Graduate School of Mathematical Sciences, The University of Tokyo, 3-8-1 Komaba, Meguro-ku, 153-8914, Tokyo, Japan}
\affil[2]{Graduate School of Economics, Hitotsubashi University, 2-1 Naka, Kunitachi, 186-8601, Tokyo, Japan}

\begin{document}
\date{}
%\vspace{0.7cm}

\maketitle
%%%%%%%%%%%%%%
\begin{abstract}
We introduce a space of $L^2$ vector fields with bounded mean oscillation whose ``normal'' component to the boundary is well-controlled.
We establish its Helmholtz decomposition in the case when the domain is a perturbed $C^3$ half space in $\mathbf{R}^n$ $(n \geq 3)$ with small perturbation.
\end{abstract}
%
%\keywords{$BMO$, normal trace, zero extension, Jones' extension}
\begin{center}
\end{center}

%%%%%%%%%%%%%%%
%
% 原稿1-1 / 8
\section{Introduction} % Section 1
\label{intro}
This is a continuation of our paper \cite{GG22b}.
It is well known that the Helmholtz decomposition plays a key role in analyzing the Stokes and the Navier-Stokes equations \cite{L}.
Such decomposition is well studied for $L^p$ space with $1<p<\infty$.
It is a topological direct sum decomposition
\[
L^p(\Omega)^n = L^p_\sigma(\Omega) \oplus G^p(\Omega)
\]
of the $L^p$ vector fields in a domain $\Omega \subseteq \mathbf{R}^n$.
% 原稿 （2022/11/6) 2/6
Here, $L^p_\sigma(\Omega)$ denotes the $L^p$-closure of the space of all smooth div-free vector fields that are compactly supported and $G^p$ denotes the space of all $L^p$ gradient fields.
If $p=2$, such decomposition holds for any domain $\Omega$.
It is an orthogonal decomposition and often called Weyl's decomposition.
For $1<p<\infty$, the decomposition still holds for various domain including the whole space $\mathbf{R}^n$, the half space $\mathbf{R}^n_+$, a perturbed half space, a bounded smooth domain \cite{FM} and an exterior smooth domain.
However, there are smooth unbounded domains which do not admit $L^p$-Helmholtz decomposition;
see e.g.\ a nice book of Galdi \cite{Gal}.

If $p=\infty$, such a decomposition does not hold even when $\Omega=\mathbf{R}^n$ since the projection operator is a kind of the Riesz operator which is unbounded in $L^\infty$, though it is bounded in $L^p$ ($1<p<\infty$).
We replace $L^\infty$ by $BMO$ space.
It turns out that it is convenient to consider a subspace $vBMO$ of $BMO$ to have the Helmholtz decomposition, at least for a half space \cite{GG20} and a bounded domain \cite{GG22b}.
Our goal is to extend such a result to a perturbed half space.
Unfortunately, it seems that a direct extension is difficult because the behavior at space infinity is not well controlled.
Thus we consider the $L^2$ intersection of this space.
For $L^p$ space, Farwig, Kozono, and Sohr \cite{FKS1} established the Helmholtz decomposition of $L^p\cap L^2$ ($p\geq 2$) and $L^p+L^2$ ($1<p<2$) for arbitrary uniformly $C^2$ smooth domain in $\mathbf{R}^3$.
(Later, it is extended to arbitrary dimension \cite{FKS2}.)
Although we consider $vBMO\cap L^2$ in a slightly perturbed half space in the present paper, our results extend to any uniformly $C^3$ domain.
This will be discussed in a separate forthcoming paper.

% 原稿1-2 / 8
Let us recall the $BMO$ space of vector fields introduced in \cite{GG22a}, \cite{GG22b}.
For $\mu\in(0,\infty]$, we recall the $BMO$ seminorm.
For a locally integrable function $f$ in $\Omega$, i.e., $f\in L^1_{loc}(\Omega)$, we set
\[
[f]_{BMO^\mu(\Omega)} := \sup \left\{ \frac{1}{\left|B_r(x)\right|} \int_{B_r(x)} \left| f(y) - f_{B_r(x)} \right| \, dy \biggm| B_r(x) \subset \Omega,\ r < \mu \right\};
\]
here $f_B$ denotes the average over $B$ and $B_r(x)$ denotes the open ball of radius $r$ centered at $x$ and $\lvert B \rvert$ denotes the Lebesgue measure of $B$.
For $\nu\in(0,\infty]$, we also use a seminorm
\[
[f]_{b^\nu(\Gamma)} := \sup \left\{ r^{-n} \int_{\Omega\cap B_r(x)} \left| f(y) \right| \, dy \biggm| x \in \Gamma,\ 0<r<\nu \right\},
\]
where $\Gamma:=\partial\Omega$ denotes the boundary of $\Omega$.
The space
 \[
BMO^{\mu,\nu}_b(\Omega) := \left\{ f \in L^1_{loc}(\Omega) \,\middle\vert\, [f]_{BMO^\mu(\Omega)} + [f]_{b^\nu(\Gamma)}< \infty \right\}
\]
is essentially introduced in \cite{BG} and well studied in \cite{BGST}.
The Stokes semigroup in such spaces was studied \cite{BG}, \cite{BGS} and it is useful to prove that the analyticity of the Stokes semigroup still holds in some unbounded domains which do not admit $L^p$-Helmholtz decomposition \cite{BGMST}.

Our space $vBMO$ requires a control only on the normal component.
Let $d_\Gamma$ denote the distance function from $\Gamma$.
We set
\begin{align*}
	&[v]_{vBMO^{\mu,\nu}(\Omega)} := [v]_{BMO^\mu(\Omega)} + [\nabla d_\Gamma\cdot v]_{b^\nu(\Gamma)}, \\
	&vBMO^{\mu,\nu}(\Omega) := \left\{ v \in L^1_{loc}(\Omega)^n \,\middle\vert\, [v]_{vBMO^{\mu,\nu}(\Omega)} < \infty \right\},
\end{align*}
where $\cdot$ denotes the standard inner product.
This is a seminorm.
If $\Omega=\mathbf{R}^n_+$, this is not a norm unless $n=1$, $\nu=\infty$.
However, if $\Gamma$ has a fully curved part in the sense of \cite[Definition 7]{GG22a}, then $[\cdot]_{vBMO^{\mu,\nu}(\Omega)}$ becomes a norm \cite[Lemma 8]{GG22a}.
In particular, when $\Omega$ is a bounded $C^2$ domain, this is a norm.
In this paper, we consider the case where $\Omega$ is a perturbed $C^k$ half space
\[
\mathbf{R}_h^n := \left\{ x = (x',x_n) \in \mathbf{R}^n \,\middle\vert\, x_n > h(x') \right\},
\]
where $h\not\equiv 0$ is in $C^k_c(\mathbf{R}^{n-1})$, i.e., $h$ is a compactly supported $C^k$ function in $\mathbf{R}^{n-1}$;
here $x'=(x_1,\ldots,x_{n-1})$ for $x\in\mathbf{R}^n$.
A perturbed $C^k$ ($k\geq2$) half space $\mathbf{R}^n_h$ is said to be type $(K)$ if
\[
\sup_{x' \in \mathbf{R}^{n-1}} \left\lvert \nabla'^2 h (x') \right\rvert < K
\]
where $\nabla' := (\partial_1, \partial_2,\ldots, \partial_{n-1})$.
We note that the perturbed $C^k$ half space has a fully curved part so that $[v]_{vBMO^{\mu,\nu}(\Omega)}$ is a norm.
By definition, there always exists $R_h>0$ such that the support $\operatorname{supp} h \subseteq \overline{B_{R_h}(0')}$.
We say that the perturbed $C^k$ half space $\mathbf{R}^n_h$ has small perturbation if 
\begin{align} \label{SC}
R_h^{\frac{2n-1}{2n}} < \frac{1}{2}, \quad
C_s(h)^{\frac{3n}{2} + 8} C_1(h) \left( C_{\ast,1}(h) + C_{\ast,2}(h) + R_h^{\frac{n}{2}} \right) < \frac{1}{2C^\ast(n)}
\end{align}
where $C^\ast(n)$ is a specific constant depending only on the space dimension $n$,
\begin{align*}
	&C_s(h) := 1 + \lVert h \rVert_{C^1(\mathbf{R}^{n-1})}, \quad 
	C_1(h) := 1 + R_h \left\lVert \nabla'^2 h \right\rVert_{L^\infty(\mathbf{R}^{n-1})}, \\
	&C_{\ast,1}(h) := C_1(h)^3 \left(1 + R_h^{\frac14}\right) \left( R_h^{\frac12} \lVert \nabla'^2 h \rVert_{L^\infty(\mathbf{R}^{n-1})} + R_h^{\frac52} \lVert \nabla'^2 h \rVert_{L^\infty(\mathbf{R}^{n-1})}^3 \right), \\
	&C_{\ast,2}(h) := \left( R_h + R_h^{\frac{1}{2h}} \right) \left\lVert \nabla'^2 h \right\rVert_{L^\infty(\mathbf{R}^{n-1})} + \left(R_h^{n-1} + 1\right) \lVert h \rVert_{C^1(\mathbf{R}^{n-1})}.
\end{align*}
To simplify the behavior near the infinity, we consider
\[
	vBMOL^2(\Omega) := vBMO^{\mu,\nu}(\Omega) \cap L^2(\Omega).
\]
Note that this space is independent of the choice of $\mu,\nu$.
The purpose of this paper is to establish the Helmholtz decomposition for the space $vBMOL^2\big( \mathbf{R}_h^n \big)$ in the case where $\mathbf{R}_h^n$ is a perturbed $C^3$ half space that has small perturbation with $n \geq 3$. Here is our main theorem.

\begin{theorem} \label{M}
Let $\mathbf{R}_h^n$ be a perturbed $C^3$ half space of type $(K)$ that has small perturbation with $n \geq 3$ and $\Gamma = \partial \mathbf{R}_h^n$.
Then for any $v \in vBMOL^2\big( \mathbf{R}_h^n \big)$, there exists a unique decomposition $v = v_0 + \nabla q$ with
\begin{align*}
v_0 \in vBMOL^2_\sigma\big( \mathbf{R}_h^n \big) &:= \left\{ f \in vBMOL^2\big( \mathbf{R}_h^n \big) \,\middle\vert\, \operatorname{div} f = 0\ \text{in}\ \mathbf{R}_h^n,\ f \cdot \mathbf{n} = 0\ \text{on}\ \Gamma \right\}, \\
\nabla q \in GvBMOL^2\big( \mathbf{R}_h^n \big) &:= \left\{ \nabla p \in vBMOL^2\big( \mathbf{R}_h^n \big) \,\middle\vert\, p \in L^2_{loc}\big( \mathbf{R}_h^n \big) \right\}
\end{align*}
satisfying the estimate
\[
\Vert v_0 \Vert_{vBMOL^2\big( \mathbf{R}_h^n \big)} + \Vert \nabla q \Vert_{vBMOL^2\big( \mathbf{R}_h^n \big)} \leq C(K,R_\ast,R_h) \Vert v \Vert_{vBMOL^2\big( \mathbf{R}_h^n \big)},
\]
where $C(K,R_\ast,R_h)$ is a constant that depends only on the constant $K$ which controls the second order derivative of $h$, the reach of the boundary $R_\ast$ and the size $R_h$ which characterizes the support of $h$.
In particular, the Helmholtz projection $P_{vBMOL^2}$, defined by $P_{vBMOL^2}(v) = v_0$, is a bounded linear map on $vBMOL^2\big( \mathbf{R}_h^n \big)$ with range $vBMOL^2_\sigma\big( \mathbf{R}_h^n \big)$ and kernel $GvBMOL^2\big( \mathbf{R}_h^n \big)$.
\end{theorem}

Roughly speaking, the reach $R_\ast$ represents the size of a small neighborhood of the boundary within which every point has a unique projection on the boundary. Here we would like to direct the readers to subsection \ref{sub:LOC} for its precise definition.

Our strategy to prove Theorem \ref{M} follows from the potential theoretic strategy we used to establish the Helmholtz decomposition in a bounded $C^3$ domain \cite{GG22b}.
Let $E$ represents the fundamental solution of $-\Delta$ in $\mathbf{R}^n$.
By \cite{Gu}, we see that as long as the boundary $\Gamma$ is uniformly $C^2$, the space $BMO^\infty(\Omega) \cap L^2(\Omega)$ allows the standard cut-off, 
i.e., we can decompose $v$ into two parts $v=v_1+ v_2$ with $v_2 = \varphi v$ and $v_1 = v - v_2$ with some $\varphi \in C^\infty(\mathbf{R}^n)$ that is supported within a small neighborhood of $\Gamma$.
Thus, the support of $v_2$ lies in a small neighborhood of $\Gamma$ whereas the support of $v_1$ is away from $\Gamma$.
For $v_1$, by extending $v_1$ as zero outside $\Omega$, we just set $q^1_1 = E * \operatorname{div} v_1$.
Then, the $L^\infty$ bound for $\nabla q^1_1$ is well controlled near $\Gamma$, which yields a bound for $b^\nu$ semi-norm.
To estimate $v_2$, we use a normal coordinate system near $\Gamma$ and reduce the problem to the half space. 
We extend $v_2$ to $\mathbf{R}^n$ so that the normal part $(\nabla d\cdot\overline{v}_2) \nabla d$ is odd and the tangential part $\overline{v_2} - (\nabla d \cdot \overline{v_2}) \nabla d$ is even in the direction of $\nabla d$ with respect to $\Gamma$.
In such type of coordinate system, the minus Laplacian can be transformed as
\[
	L = A - B + \text{lower order terms},\ 
	A = -\Delta_\eta, \ B = \sum_{1\leq i,j\leq n-1} \partial_{\eta_i} b_{ij} \partial_{\eta_j},
\]
where $\eta_n$ is the normal direction to the boundary so that $\{ \eta_n>0\}$ is the half space.
We then use a freezing coefficient method to construct volume potential $q_1^{\mathrm{tan}}$ and $q_1^{\mathrm{nor}}$, which corresponds to the contribution from the tangential part $\overline{v_2}^{\mathrm{tan}}$ and the normal part $\overline{v_2}^{\mathrm{nor}}$, respectively. Since the leading term of $\operatorname{div} \overline{v_2}^{\mathrm{nor}}$ in normal coordinate consists of the differential of $\eta_n$ only,
if we extend the coefficient $b_{ij}$ even in $\eta_n$, 
$q_1^{\mathrm{nor}}$ is constructed so that the leading term of $\nabla d \cdot \nabla q_1^{\mathrm{nor}}$ is odd in the direction of $\nabla d$.
On the other hand, as the leading term of $\operatorname{div} \overline{v_2}^{\mathrm{tan}}$ in normal coordinate consists of the differential of $\eta' = (\eta_1, ... , \eta_{n-1})$ only, the even extension of $b_{ij}$ in $\eta_n$ gives rise to $q_1^{\mathrm{tan}}$ so that the leading term of $\nabla d \cdot \nabla q_1^{\mathrm{tan}}$ is also odd in the direction of $\nabla d$.
Disregarding lower order terms and localization procedure, $q_1^{\mathrm{tan}}$ and $q_1^{\mathrm{nor}}$ are constructed as
\begin{align*}
q_1^{\mathrm{tan}} &= -L^{-1} \operatorname{div}\overline{v}^{\mathrm{tan}}_2 = -A^{-1} (I-BA^{-1})^{-1} \operatorname{div}\overline{v}^{\mathrm{tan}}_2, \\
q_1^{\mathrm{nor}} &= -L^{-1} \operatorname{div}\overline{v}^{\mathrm{nor}}_2 = -A^{-1} (I-BA^{-1})^{-1} \operatorname{div}\overline{v}^{\mathrm{nor}}_2.
\end{align*}
One is able to arrange $BA^{-1}$ small by working in a small neighborhood of a boundary point.
Then $(I-BA^{-1})^{-1}$ is given as the Neumann series $\sum^\infty_{m=0}(BA^{-1})^m$.
The $BMO$-$BMO$ estimates for $\nabla q_1^{\mathrm{tan}}$ and $\nabla q_1^{\mathrm{nor}}$ follow from \cite{FS}. 
Since the leading term of $\nabla d \cdot (\nabla q_1^{\mathrm{tan}} + \nabla q_1^{\mathrm{nor}})$ is odd in the direction of $\nabla d$ with respect to $\Gamma$, the $BMO$ bound implies $b^\nu$ bound.
The $L^2$ estimates for $\nabla q_1^{\mathrm{tan}}$ and $\nabla q_1^{\mathrm{nor}}$ hold as in the localization procedure, the partition of unity we consider for a small neighborhood of $\Gamma$ is locally finite.
As a result, setting $q_1=q^1_1+q_1^{\mathrm{tan}}+q_1^{\mathrm{nor}}$ would give us our desired volume potential corresponding to $\operatorname{div} v$. 
Some results needed for the construction of volume potential $q_1$ have already been established in \cite[Section 3]{GG22b} and \cite{Gu}, for these parts we omit their proofs and recall them directly.

\begin{theorem}[Construction of a suitable volume potential] \label{CSV}
Let $\Omega \subset \mathbf{R}^n$ be a uniformly $C^3$ domain of type $(\alpha,\beta,K)$ with $n \geq 2$. Let $R_\ast$ be the reach of the boundary $\Gamma = \partial \Omega$.
Then, there exists a bounded linear operator $v\longmapsto q_1$ from $vBMOL^2(\Omega)$ to $L^\infty(\Omega)$ such that
\[
- \Delta q_1 = \operatorname{div}v \quad\text{in}\quad \Omega
\]
and that there exists a constant $C=C(\alpha,\beta,K,R_\ast)>0$ satisfying
\[
\Vert \nabla q_1\Vert_{vBMOL^2(\Omega)} \leq C \Vert v\Vert_{vBMOL^2(\Omega)}.
\]
In particular, the operator $v\longmapsto\nabla q_1$ is a bounded linear operator in $vBMOL^2(\Omega)$.
\end{theorem}

Here $(\alpha,\beta,K)$ are parameters that characterize the regularity of $\Gamma$. We would like to direct the readers to subsection \ref{sub:LOC} for their precise definitions.
Although the construction of the suitable volume potential works for arbitrary uniformly $C^3$ domain in $\mathbf{R}^n$ with $n \geq 2$, for the rest of the theory we need to focus back on perturbed half space in $\mathbf{R}^n$ with $n \geq 3$.
For $v \in vBMOL^2\big( \mathbf{R}_h^n \big)$, we observe that $w=v-\nabla q_1$ is divergence free in $\mathbf{R}_h^n$.
Unfortunately, this $w$ may not fulfill the trace condition $w \cdot \mathbf{n}=0$ on the boundary $\Gamma = \partial \mathbf{R}_h^n$.
We construct another potential $q_2$ by solving the Neumann problem
\begin{align*}
\Delta q_2 &= 0 \quad \quad \quad \text{in}\quad \mathbf{R}_h^n, \\
\frac{\partial q_2}{\partial n} &= w\cdot\mathbf{n} \quad \; \text{on} \quad \Gamma.
\end{align*}
We then set $q=q_1+q_2$.
Since $\partial q_2/\partial\mathbf{n}=\nabla q_2\cdot \mathbf{n}$, $v_0=v-\nabla q$ gives the Helmholtz decomposition.
To complete the proof of Theorem \ref{M}, it suffices to control $\Vert \nabla q_2\Vert_{vBMOL^2\big( \mathbf{R}_h^n \big)}$ by $\Vert v\Vert_{vBMOL^2\big( \mathbf{R}_h^n \big)}$.
\begin{lemma}[Estimate of the normal trace] \label{ET}
Let $\mathbf{R}_h^n$ be a perturbed $C^{2+\kappa}$ half space of type $(K)$ with $\kappa \in (0,1)$, $n \geq 3$ and $\Gamma = \partial \mathbf{R}_h^n$.
Then, there is a constant $C = C(K,R_\ast,R_h)>0$ such that
\[
\Vert w \cdot \mathbf{n} \Vert_{L^\infty(\Gamma) \cap \dot{H}^{-\frac{1}{2}}(\Gamma)} \leq C \Vert w \Vert_{vBMOL^2\big( \mathbf{R}_h^n \big)}
\]
for all $w \in vBMOL^2\big( \mathbf{R}_h^n \big)$ with $\operatorname{div} w = 0$ in $\mathbf{R}_h^n$.
\end{lemma}

Here $\dot{H}^{-\frac{1}{2}}(\Gamma)$ is a Hilbert space that is isomorphic to $\dot{H}^{-\frac{1}{2}}(\mathbf{R}^{n-1})$, which turns out to be the dual space of the homogeneous fractional Sobolev space $\dot{H}^{\frac{1}{2}}(\mathbf{R}^{n-1})$. Here, the homogeneous Sobolev space of order $s \in \mathbf{R}$ is defined as
\begin{multline*}
	\dot{H}^s(\mathbf{R}^{n-1})
	:=\Bigg\{ f \in \mathcal{S}'(\mathbf{R}^{n-1}) \, \Bigg\vert \, \widehat{f} \in L^1_{loc}(\mathbf{R}^{n-1}) \quad \text{and} \\ 
	\Vert f \Vert_{\dot{H}^s(\mathbf{R}^{n-1})}
	:= \left( \int_{\mathbf{R}^{n-1}} \vert \xi' \vert^{2s} \big\vert \widehat{f}(\xi') \big\vert^2 \, d\xi' \right)^{\frac{1}{2}} < \infty \Bigg\}
\end{multline*}
where $\mathcal{S}'(\mathbf{R}^{n-1})$ denotes the space of Schwartz's tempered distributions and $\widehat{f}$ denotes the Fourier transform of $f$, see e.g. \cite[Section 1.3]{BCD}.
Unfortunately, the space $\dot{H}^{\frac{1}{2}}(\mathbf{R}^{n-1})$ is complete if and only if $n \geq 3$. 
Thus, we assume that $n \geq 3$ when the $\dot{H}^{-\frac{1}{2}}$ norm appears.
The basic idea to establish Lemma \ref{ET} is as follows.
The $L^\infty$ estimate of $w \cdot \mathbf{n}$ follows directly from the trace theorem established in \cite[Theorem 22]{GG22a}. 
For the $\dot{H}^{-\frac{1}{2}}$ estimate for $w \cdot \mathbf{n}$, we split the boundary into the straight part and the curved part. Since we have the $L^\infty$ estimate for $w \cdot \mathbf{n}$ and the curved part is compact, the contribution in the $\dot{H}^{-\frac{1}{2}}$ estimate for $w \cdot \mathbf{n}$ that comes from the curved part can be estimated by the $L^\infty$ norm of $w \cdot \mathbf{n}$.
For the contribution in the $\dot{H}^{-\frac{1}{2}}$ estimate of $w \cdot \mathbf{n}$ that comes from the straight part, we invoke the $\dot{H}^{-\frac{1}{2}}$ estimate of $w \cdot \mathbf{n}$ in the case of the half space.
We finally need the estimate for the Neumann problem.

\begin{lemma}[Estimate for the Neumann problem] % Lemma 2
\label{EN}
Let $\mathbf{R}_h^n$ be a perturbed $C^2$ half space of type $(K)$ that has small perturbation with $n \geq 3$. Let $R_\ast$ be the reach of $\Gamma = \partial \mathbf{R}_h^n$.
For any $g \in L^\infty(\Gamma) \cap \dot{H}^{-\frac{1}{2}}(\Gamma)$, there exists a unique (up to constant) solution $u \in L_{loc}^2\big( \mathbf{R}_h^n \big)$ to the Neumann problem
\begin{align}
&\begin{aligned} \label{1NP} % (3)
\Delta u &= 0 \quad\text{in}\quad \mathbf{R}_h^n \\
\frac{\partial u}{\partial \mathbf{n}} &= g \quad\text{on}\quad \Gamma
\end{aligned}
\end{align}
such that the operator $g \longmapsto u$ is linear. Moreover, there exists a constant $C = C(K,R_\ast,R_h)>0$ such that
\[
\Vert \nabla u \Vert_{vBMOL^2\big( \mathbf{R}_h^n \big)} \leq C \Vert g\Vert_{L^\infty(\Gamma) \cap \dot{H}^{-\frac{1}{2}}(\Gamma)}.
\]
\end{lemma}

To establish Lemma \ref{EN}, the basic strategy is the same as in \cite{GG22b}, we firstly show that
\begin{equation*}
\left\Vert \nabla E \ast \big( \delta_{\Gamma} \otimes g \big) \right\Vert_{vBMO^{\infty,\nu}\big( \mathbf{R}_h^n \big)}
\end{equation*}
can be controlled for any $g \in L^\infty(\Gamma) \cap \dot{H}^{-\frac{1}{2}}(\Gamma)$ where 
\[
E \ast (\delta_\Gamma \otimes g)(x) := \int_\Gamma E(x-y) g(y) \, d\mathcal{H}^{n-1}(y), \quad x \in \mathbf{R}_h^n
\]
represents the single layer potential for $g$.
Since the boundary $\Gamma = \partial \mathbf{R}_h^n$ is curved and not compact, we appeal a perturbation argument.
For $g \in L^\infty(\Gamma)$, we decompose $g$ into the curved part $g_1$ and the straight part $g_2$ by setting $g_1\big( y',h(y') \big) := 1'_{B_{2R_h}(0')}(y') g\big( y',h(y') \big)$ for $y' \in \mathbf{R}^{n-1}$ and $g_2 := g - g_1$, where $1'_{B_{2R_h}(0')}$ represents the characteristic function for the open ball $B_{2R_h}(0')$ in $\mathbf{R}^{n-1}$. Since $g_2$ vanishes in the curved part of $\Gamma$, we can define $g_2^H \in L^\infty(\mathbf{R}^{n-1}) \cap \dot{H}^{-\frac{1}{2}}(\mathbf{R}^{n-1})$ by setting $g_2^H(y',0) = g_2\big( y',h(y') \big)$ for any $y' \in \mathbf{R}^{n-1}$. Note that
\[
E \ast \big( \delta_\Gamma \otimes g_2 \big) (x) = E \ast \big( \delta_{\partial \mathbf{R}_+^n} \otimes g_2^H \big) (x)
\]
for any $x \in \mathbf{R}^n$.
By setting $g_2^{\mathbf{R}^n}(y',y_n) := g_2^H(y',0)$ for any $(y',y_n) \in \mathbf{R}^n$, we can deduce the $BMO$ estimate of $\nabla E \ast \big( \delta_\Gamma \otimes g_2 \big)$ by applying the $L^\infty-BMO$ estimate for singular integral operator \cite[Theorem 4.2.7]{GraM} to $\nabla \partial_{x_n} E \ast 1_{\mathbf{R}_+^n} g_2^{\mathbf{R}^n}$. 
Since $g_1\big( \cdot', h(\cdot') \big)$ is compactly supported in $\mathbf{R}^{n-1}$, we may extend $g_1$ to some $g_{1,\mathrm{c}}^\mathrm{e} \in L^\infty(\mathbf{R}^n)$ such that $g_{1,\mathrm{c}}^\mathrm{e}$ is compactly supported and $\nabla d \cdot \nabla g_{1,\mathrm{c}}^\mathrm{e} = 0$ within a small neighborhood of $\Gamma$. Then, we can rewrite $\nabla E \ast \big( \delta_\Gamma \otimes g_1 \big)$ as
\begin{align} \label{ReGSg1}
\nabla E \ast \big( \delta_\Gamma \otimes g_1 \big) = \nabla\operatorname{div} \big(E*( g_{1,\mathrm{c}}^\mathrm{e} 1_{\mathbf{R}_h^n} \nabla d) \big) -\nabla E* \big( 1_{\mathbf{R}_h^n} g_1^\mathrm{e} f_{\theta,\rho_0/4} \big)
\end{align}
with some compactly supported continuous function $f_{\theta,\rho_0/4}$ (see Proof of Lemma \ref{NM} (i)) that is independent of $g$.
The $BMO$ estimate for the first term on the right hand side of \eqref{ReGSg1} can be controlled using the $L^\infty-BMO$ estimate for singular integration operator. The second term on the right hand side of \eqref{ReGSg1} can be controlled by the $L^\infty$ norm as $\nabla E (x)$ is a locally integrable kernel and $f_{\theta,\rho_0/4}$ has compact support. We thus obtain the $BMO$ estimate for $\nabla E \ast \big( \delta_\Gamma \otimes g \big)$ for $g \in L^\infty(\Gamma)$.

For the $b^\nu$ estimate of the normal component of $\nabla E \ast \big( \delta_\Gamma \otimes g \big)$ with $g \in L^\infty(\Gamma) \cap \dot{H}^{-\frac{1}{2}}(\Gamma)$, we also decompose $g$ into the curved part $g_1$ and the straight part $g_2$. 
It is possible to show that
\begin{align} \label{abvNGE}
\sup_{x \in \Gamma^{\mathbf{R}^n}_{\rho_0}} \int_\Gamma \left\lvert \frac{\partial E}{\partial\mathbf{n}_y} (x-y) \right\rvert \, d\mathcal{H}^{n-1}(y) < \infty
\end{align}
where $\Gamma^{\mathbf{R}^n}_{\rho_0} := \left\{ x \in \Omega \,\middle\vert\,d(x) < \rho_0 \right\}$ denotes the $\rho_0$-neighborhood of $\Gamma$ in $\Omega$. 
Since $g_1\big( \cdot',h(\cdot') \big)$ is compactly supported in $\mathbf{R}^{n-1}$, estimate \eqref{abvNGE} allows us to show that
\[
\left\lvert \nabla d (x) \cdot \nabla \big( E*(\delta_\Gamma\otimes g_1) \big)(x) \right\rvert \leq C \Vert g \Vert_{L^\infty(\Gamma)}
\]
for any $x \in \Gamma^{\mathbf{R}^n}_{\rho_0}$ with some constant $C=C(K,R_h,R_\ast)>0$. 
On the other hand, for $x$ close to the curved part of $\Gamma$, $\nabla d(x)$ is not necessarily $(0,...,0,1)$, in this case
\[
\left\lvert \nabla d(x) \cdot \nabla \big( E \ast (\delta_\Gamma \otimes g_2) \big) (x) \right\rvert 
\]
would contain contributions from $\nabla' \big( E \ast (\delta_\Gamma \otimes g_2) \big)$, which cannot be estimated by the $L^\infty$ bound of $g_2$ (see Proposition \ref{RfHdmh}). As a result, we introduce the $\dot{H}^{-\frac{1}{2}}$ bound of $g_2$.
Since $g_2\big( \cdot',h(\cdot') \big)$ is supported in $B_{2R_h}(0')^\mathrm{c}$, $\nabla d(x) \cdot \nabla E (x - \cdot')$ can be viewed as an element of $\dot{H}^{\frac{1}{2}}(\Gamma)$ for any $x$ close to the curved part of $\Gamma$. Hence, by the $\dot{H}^{-\frac{1}{2}}-\dot{H}^{\frac{1}{2}}$ duality we can deduce that
\[
\left\lvert \nabla d(x) \cdot \nabla \big( E \ast (\delta_\Gamma \otimes g_2) \big) (x) \right\rvert \leq C \Vert g \Vert_{\dot{H}^{-\frac{1}{2}}(\Gamma)}
\]
with some constant $C=C(K,R_h,R_\ast)>0$.
We thus obtain an $L^\infty$ estimate for the normal component of $\nabla E \ast \big( \delta_\Gamma \otimes g \big)$ within a small neighborhood of $\Gamma$. The $b^\nu$ estimate naturally follows.

For $g \in L^\infty(\Gamma) \cap \dot{H}^{-\frac{1}{2}}(\Gamma)$, we prove that the trace of 
\[
\big( Qg \big) (x) = \int_\Gamma \frac{\partial E}{\partial \mathbf{n}_x} (x-y) g(y) \, d\mathcal{H}^{n-1}(y), \quad x \in \Gamma^{\mathbf{R}^n}_{\rho_0}
\]
is of the form
\[
\gamma \big( Qg \big) \big( x',h(x') \big) = \frac{1}{2} g\big( x',h(x') \big) - \big( Sg \big) \big( x',h(x') \big),
\]
where $S : L^\infty(\Gamma) \cap \dot{H}^{-\frac{1}{2}}(\Gamma) \to L^\infty(\Gamma) \cap \dot{H}^{-\frac{1}{2}}(\Gamma)$ is a bounded linear operator satisfying
\begin{align*}
\Vert S \Vert_{L^\infty(\Gamma) \cap \dot{H}^{-\frac{1}{2}}(\Gamma) \to L^\infty(\Gamma) \cap \dot{H}^{-\frac{1}{2}}(\Gamma)} \leq C^\ast(n) C_s(h)^{\frac{3n}{2} + 8} C_1(h) \big( C_{\ast,1}(h) + C_{\ast,2}(h) + R_h^{\frac{n}{2}} \big)
\end{align*}
with $C^\ast(n)$ denoting a specific fixed constant which depends on dimension $n$ only. 
Therefore, if $\mathbf{R}_h^n$ is a perturbed $C^2$ half space that has small perturbation with $n \geq 3$, the inverse of $I - 2S$ is well-defined as a bounded linear map from $L^\infty(\Gamma) \cap \dot{H}^{-\frac{1}{2}}(\Gamma)$ to $L^\infty(\Gamma) \cap \dot{H}^{-\frac{1}{2}}(\Gamma)$ by the Neumann series
\[
(I - 2S)^{-1} = \sum_{i=0}^\infty (2S)^i.
\]
The solution to the Neumann problem (\ref{1NP}) is formally given by 
\[
u(x) = E \ast \Big( \delta_\Gamma \otimes \big( 2(I - 2S)^{-1} g \big) \Big) (x),  \quad x \in \mathbf{R}_h^n.
\]

We finally need the $L^2$ estimate for $\nabla u$ in $\mathbf{R}_h^n$.
In the case of a half space, for $g \in \dot{H}^{-\frac{1}{2}}(\mathbf{R}^{n-1})$, the single layer potential $E \ast \big( \delta_{\partial \mathbf{R}_+^n} \otimes g \big)$ is exactly half of the solution $u$ to the Neumann problem. By directly considering the partial Fourier transform of $E(x',x_n)$ with respect to $x'$, we could easily deduce that
\[
\Vert \nabla u \Vert_{L^2(\mathbf{R}_+^n)} = 2 \big\Vert \nabla E \ast \big( \delta_{\partial \mathbf{R}_+^n} \otimes g \big) \big\Vert_{L^2(\mathbf{R}_+^n)} = C(n) \Vert g \Vert_{\dot{H}^{-\frac{1}{2}}(\mathbf{R}^{n-1})}.
\]
In the case that $\mathbf{R}_h^n$ is a perturbed $C^2$ half space with $n \geq 3$, for $g \in L^\infty(\Gamma) \cap \dot{H}^{-\frac{1}{2}}(\Gamma)$, we still decompose $g$ into the curved part $g_1$ and the straight part $g_2$. 
Since $L^{\frac{2n-2}{n}}(\Gamma)$ is continuously embedded in $\dot{H}^{-\frac{1}{2}}(\Gamma)$ and the curved part $g_1$ has compact support in $\Gamma$, $g \in L^\infty(\Gamma)$ would imply that both $g_1,g_2 \in L^\infty(\Gamma) \cap \dot{H}^{-\frac{1}{2}}(\Gamma)$.
Since for any $x \in \mathbf{R}^n$ we have that
\[
\nabla E \ast \big( \delta_\Gamma \otimes g_2 \big) (x) = \nabla E \ast \big( \delta_{\partial \mathbf{R}_+^n} \otimes g_2^H \big) (x)
\]
and for any $x = (x',x_n) \in \mathbf{R}_+^n$ we have that
\[
\left\lvert \nabla E \ast \big( \delta_{\partial \mathbf{R}_+^n} \otimes g_2^H \big) (x',-x_n) \right\rvert = \left\lvert \nabla E \ast \big( \delta_{\partial \mathbf{R}_+^n} \otimes g_2^H \big) (x',x_n) \right\rvert,
\]
the $L^2$ estimate of $\nabla E \ast \big( \delta_\Gamma \otimes g_2 \big)$ in $\mathbf{R}^n$ follows from the $L^2$ estimate of $\nabla E \ast \big( \delta_{\partial \mathbf{R}_+^n} \otimes g_2^H \big)$ in $\mathbf{R}_+^n$, i.e., we have that
\begin{align*}
\big\Vert \nabla E \ast \big( \delta_\Gamma \otimes g_2 \big) \big\Vert_{L^2(\mathbf{R}^n)} &= 2 \big\Vert \nabla E \ast \big(\delta_{\partial \mathbf{R}_+^n} \otimes g_2^H \big) \big\Vert_{L^2(\mathbf{R}_+^n)} \leq C(K,R_h) \Vert g \Vert_{L^\infty(\Gamma) \cap \dot{H}^{-\frac{1}{2}}(\Gamma)}.
\end{align*}
For the curved part $g_1$, we recall the argument which establishes the $vBMO^{\infty,\nu}$ norm for $\nabla E \ast \big( \delta_\Gamma \otimes g \big)$. We extend $g_1$ to $g_{1,\mathrm{c}}^\mathrm{e} \in L^\infty(\mathbf{R}^n)$ and consider equation (\ref{ReGSg1}).
Since $\nabla \operatorname{div} E$ is $L^p$ integrable for any $1<p<\infty$, see e.g. \cite[Theorem 5.2.7 and Theorem 5.2.10]{Gra}, the $L^2$ norm of the first term on the right hand side of (\ref{ReGSg1}) can be estimated by the $L^\infty$ norm of $g$. Whereas $\nabla E(x)$ is an integration kernel that is dominated by a constant multiple of $\lvert x\rvert^{-(n-1)}$, by the famous Hardy-Littlewood-Sobolev inequality \cite[Theorem 1.7]{BCD}, we can also control the $L^2$ norm of the second term on the right hand side of (\ref{ReGSg1}) by the $L^\infty$ norm of $g$.
Combine with the $L^2$ estimate for the contribution from the straight part $g_2$, we finally obtain our desired $L^2$ estimate
\begin{align*} 
\Vert \nabla u \Vert_{L^2\big( \mathbf{R}_h^n \big)} \leq C(K,R_h,R_\ast) \Vert (I - 2S)^{-1} g \Vert_{L^\infty(\Gamma) \cap \dot{H}^{-\frac{1}{2}}(\Gamma)} \leq C(K,R_h,R_\ast) \Vert g \Vert_{L^\infty(\Gamma) \cap \dot{H}^{-\frac{1}{2}}(\Gamma)}
\end{align*}
for $g \in L^\infty(\Gamma) \cap \dot{H}^{-\frac{1}{2}}(\Gamma)$. This completes the proof of Lemma \ref{EN}.

When taking the normal trace and solving the Neumann problem, the reason why we need to require the dimension $n$ to be greater than or equal to $3$ is because when $n \geq 3$, we indeed have the fact that $\dot{H}^{\frac{1}{2}}(\mathbf{R}^{n-1})$ is continuously embedded in $L^{\frac{2n-2}{n-2}}(\mathbf{R}^{n-1})$ and $\dot{H}^{-\frac{1}{2}}(\mathbf{R}^{n-1})$ is the dual space of $\dot{H}^{\frac{1}{2}}(\mathbf{R}^{n-1})$, which further implies that $L^{\frac{2n-2}{n}}(\mathbf{R}^{n-1})$ is continuously embedded in $\dot{H}^{-\frac{1}{2}}(\mathbf{R}^{n-1})$. Based on these facts, we can show that in the case of any perturbed $C^2$ half space $\mathbf{R}_h^n$ with boundary $\Gamma = \partial \mathbf{R}_h^n$, $\dot{H}^{\frac{1}{2}}(\mathbf{R}^{n-1})$ is isomorphic to $\dot{H}^{\frac{1}{2}}(\Gamma)$ and $\dot{H}^{-\frac{1}{2}}(\mathbf{R}^{n-1})$ is isomorphic to $\dot{H}^{-\frac{1}{2}}(\Gamma)$. More importantly, we can estimate the $\dot{H}^{-\frac{1}{2}}$ norm of the trace operator $S$ by its $L^{\frac{2n-2}{n}}$ norm and do cut-offs to boundary data $g \in L^\infty(\Gamma) \cap \dot{H}^{-\frac{1}{2}}(\Gamma)$ to  decompose it into the curved part $g_1$ and the straight part $g_2$. In the case where $n=2$, the space $\dot{H}^{\frac{1}{2}}(\mathbf{R})$ is no longer complete and $\dot{H}^{-\frac{1}{2}}(\mathbf{R})$ is not necessarily the dual space of $\dot{H}^{\frac{1}{2}}(\mathbf{R})$. 
The completion of $\dot{H}^{\frac{1}{2}}(\mathbf{R})$ cannot be embedded in the space of Schwartz's tempered distributions.
Moreover, as a limit case, $\dot{H}^{\frac{1}{2}}(\mathbf{R})$ is not continuously embedded in $L^\infty(\mathbf{R})$. 
As a result, at the present we lack of tools to estimate the $\dot{H}^{-\frac{1}{2}}$ norm of the trace operator $S$ and we cannot do cut-offs to decompose a boundary data into the curved part and the straight part any more. This is why we focus on the case where the dimension $n \geq 3$ in this paper.
The problem of $\dot{H}^{\frac{1}{2}}(\mathbf{R})$ is similar to $\dot{H}^1(\mathbf{R}^2)$. The space $\dot{H}^1(\mathbf{R}^2)$ is not complete. Its completion should be the quotient space $\{ u \in L_{loc}^1(\mathbf{R}^2) \, \vert \, \nabla u \in L^2(\mathbf{R}^2)^n \} / \mathbf{R}$ as discussed in \cite{Gal} since $\dot{H}^1(\mathbf{R}^2)$ includes all smooth compactly supported functions. We have to study an appropriate dual space as in \cite{GKL}.

This paper is organized as follows.
In Section \ref{sec:VPC}, we recall results from \cite{Gu} to localize the problem and results from \cite{GG22b} to construct a suitable volume potential corresponding to $\operatorname{div}v$. 
Theorem \ref{CSV} is proved in this section.
In Section \ref{sec:NT}, we take the normal trace in $\dot{H}^{-\frac{1}{2}}$ sense by considering isomorphisms between $\dot{H}^{-\frac{1}{2}}(\Gamma)$ and $\dot{H}^{-\frac{1}{2}}(\mathbf{R}^{n-1})$. 
In Section \ref{sec:EBI}, we establish estimates for the trace operator $S$ of $Qg$. We show that $S$ is bounded from $L^\infty \cap \dot{H}^{-\frac{1}{2}}$ to $L^\infty \cap \dot{H}^{-\frac{1}{2}}$. 
In Section \ref{sec:NPB}, we solve Neumann problem (\ref{1NP}) by considering the single layer potential with $2(I-2S)^{-1} g$ and establish the $vBMOL^2$ estimate for its gradient in $\mathbf{R}_h^n$.
%%%%%%

%%%%%%
\section{Volume potential construction in a uniformly $C^3$ domain} % Section 2
\label{sec:VPC}

For $v \in vBMOL^2(\Omega)$, we shall construct a suitable potential $q_1$ so that $v\longmapsto\nabla q_1$ is a bounded linear operator in $vBMOL^2(\Omega)$ as stated in Theorem \ref{CSV}.
The construction in the case where $\Omega$ is a uniformly $C^3$ domain basically follows from the theory in \cite{GG22b}, where $\Omega$ is a bounded $C^3$ domain. 
 
\subsection{Localization tools} % Subsection 2.1
\label{sub:LOC}
Let us recall some uniform estimates established in \cite{Gu}.
Let $\Omega$ be a uniformly $C^k$ domain in $\mathbf{R}^n$ with $k \in \mathbf{N}$ and $n \geq 2$.
Let $\Gamma = \partial \Omega$ denotes the boundary of domain $\Omega$.
There exists $\alpha, \beta>0$ such that for each $z_0 \in \Gamma$, up to translation and rotation, there exists a function $h_{z_0} \in C^k\big( B_\alpha(0') \big)$, where $B_\alpha(0')$ denotes the open ball in $\mathbf{R}^{n-1}$ of radius $\alpha$ with center $0'$, that satisfies the following properties:
\begin{enumerate}
\item[(i)] 
\[
K := \sup_{0 \leq s \leq k} \left\lVert  (\nabla')^s h_{z_0} \right\rVert_{L^\infty\big( B_\alpha(0') \big)} < \infty; \quad \nabla' h_{z_0}(0')=0, \quad h_{z_0}(0')=0,
\] 
\item[(i\hspace{-1pt}i)] $\Omega \cap U_{\alpha,\beta, h_{z_0}}(z_0)=\left\{ (x',x_n) \in \mathbf{R}^n \,\middle\vert\, h_{z_0}(x') < x_n < h_{z_0}(x')+ \beta,\ \lvert x' \rvert < \alpha \right\}$ where
\[
U_{\alpha,\beta, h_{z_0}}(z_0) := \left\{ (x',x_n) \in \mathbf{R}^n \,\middle\vert\,  h_{z_0}(x') - \beta < x_n < h_{z_0}(x') +\beta,\ \lvert x' \rvert< \alpha \right\},
\]
\item[(i\hspace{-1pt}i\hspace{-1pt}i)] $\Gamma \cap U_{\alpha,\beta, h_{z_0}}(z_0)=\left\{ (x',x_n)\in\mathbf{R}^n \,\middle\vert\,  x_n = h_{z_0}(x'),\ \lvert x' \rvert < \alpha \right\}$.
\end{enumerate}
We say that $\Omega$ is of type $(\alpha,\beta,K)$.

Let $d$ denote the signed distance function from $\Gamma$ which is defined by
\begin{equation} \label{SDF}
	d(x) := \left \{
\begin{array}{r}
	\displaystyle \inf_{y\in\Gamma} \lvert x-y \rvert \quad\text{for}\quad x\in\Omega, \\
	\displaystyle -\inf_{y\in\Gamma}\lvert x-y \rvert \quad\text{for}\quad x\notin\Omega 
\end{array}
	\right.
\end{equation}
so that $d(x)=d_\Gamma(x)$ for $x\in\Omega$.
For a uniformly $C^k$ domain $\Omega$, there is $R^\Omega>0$ such that for $x \in \Omega$ with $\left\lvert d(x)\right\rvert <R^\Omega$, there is a unique point $\pi x \in \Gamma$ such that $\lvert x-\pi x\rvert=\left\lvert d(x)\right\rvert $.
The supremum of such $R^\Omega$ is called the reach of $\Gamma$ in $\Omega$, we denote this supremum by $R_\ast^\Omega$.
Let $\Omega^{\mathrm{c}}$ be the complement of $\Omega$ in $\mathbf{R}^n$. Similarly, there is $R^{\Omega^{\mathrm{c}}}>0$ such that for $x \in \Omega^{\mathrm{c}}$ with $\left\lvert d(x)\right\rvert < R^{\Omega^{\mathrm{c}}}$, we can also find a unique point $\pi x \in \Gamma$ such that $\lvert x-\pi x\rvert=\left\lvert d(x)\right\rvert $.
The supremum of such $R^{\Omega^{\mathrm{c}}}$ is called the reach of $\Gamma$ in $\Omega^{\mathrm{c}}$, we denote this supremum by $R_\ast^{\Omega^{\mathrm{c}}}$. We then define
\[
R_* := \min \left( R^\Omega_*, R^{\Omega^{\mathrm{c}}}_* \right),
\]
which we call it the reach of $\Gamma$.
Moreover, $d$ is $C^k$ in the $\rho$-neighborhood of $\Gamma$ for any $\rho \in (0,R_\ast)$, i.e., $d \in C^k\left(\Gamma^{\mathbf{R}^n}_\rho\right)$ for any $\rho \in (0,R_\ast)$ with
\[
\Gamma^{\mathbf{R}^n}_\rho := \left\{ x \in \mathbf{R}^n \,\middle\vert\, \left\lvert  d(x) \right\rvert  < \rho \right\};
\]
see e.g. \cite[Chap.\ 14, Appendix]{GT}, \cite[Section 4.4]{KP}.

There exists $0 < \rho_0 < \min \big( \alpha, \beta, \frac{R_*}{2}, \frac{1}{2n (K+1)} \big)$ such that for every $z_0 \in \Gamma$,
\[ 
U_\rho(z_0) := \left\{ x \in \mathbf{R}^n \,\middle\vert\,  (\pi x)' \in \operatorname{int} B_\rho(0'),\ \left\lvert  d(x) \right\rvert  < \rho \right\}
\]
is contained in the coordinate chart $U_{\alpha,\beta, h_{z_0}}(z_0)$ for any $\rho \leq\rho_0$.
For $z_0 \in \Gamma$, the normal coordinate change $F_{z_0}: V_{\rho_0} := B_{\rho_0}(0') \times (-\rho_0,\rho_0) \to U_{\rho_0}(z_0)$ is defined by
\begin{eqnarray} \label{NCC}
x = F_{z_0}(\eta) =
\left\{
\begin{array}{lcl}
\eta' + \eta_n (\nabla_x' d) \big( \eta', h_{z_0}(\eta') \big); \\
h_{z_0}(\eta') + \eta_n (\partial_{x_n} d) \big( \eta', h_{z_0}(\eta') \big),
\end{array}
\right.
\end{eqnarray}
or shortly
\[
x = \pi x  - d(x) \mathbf{n} (\pi x), \quad \mathbf{n}(\pi x) = -\nabla d(\pi x).
\]
Note that that for any $z_0 \in \Gamma$, $F_{z_0}$ is indeed a $C^1$-diffeomorphism between $V_{\rho_0}$ and $U_{\rho_0}(z_0)$.
For any $\varepsilon \in (0,1)$, there exists a constant $c_\Omega^\varepsilon := C(\varepsilon, K, \rho_0)>0$ and $c_\Omega^\varepsilon < \rho_0$ such that for any $\rho \in (0, c_\Omega^\varepsilon]$ and $z_0 \in \Gamma$, the estimates
\begin{align} \label{UCN}
\Vert \nabla F_{z_0} - I \Vert_{L^\infty(V_\rho)} < \varepsilon, \quad \Vert \nabla F_{z_0}^{-1} - I \Vert_{L^\infty(U_\rho(z_0))} < \varepsilon
\end{align}
hold simultaneously, see \cite[Proposition 3]{Gu}. 

For $\rho \in (0, \rho_0/2)$, there exist a countable family of points in $\Gamma$, say $\mathcal{P}_\Gamma := \{ x_i \in \Gamma \, \mid \, i \in \mathbf{N} \}$, such that
\[
\Gamma^{\mathbf{R}^n}_{\rho} = \underset{x_i \in S}{\bigcup} \, U_{\rho}(x_i).
\]
Moreover, there exists a natural number $N_\ast=C(n)$ such that for any $x_i \in \mathcal{P}_\Gamma$, there exist at most $N_\ast$ points in $\mathcal{P}_\Gamma$, say $\{ x_{i_1}, ... , x_{i_{N_\ast}} \} \subset \mathcal{P}_\Gamma$, with
\[
U_{\rho}(x_i) \cap U_{\rho}(x_{i_l}) \neq \emptyset
\]
for each $1 \leq l \leq N_\ast$, see e.g. \cite[Proposition 5]{Gu}.
Based on this open cover, a partition of unity for $\Gamma^{\mathbf{R}^n}_{\rho}$ can be constructed.
There exist $\varphi_i \in C^1\big( U_\rho(x_i) \big)$ for each $x_i \in \mathcal{P}_\Gamma$ and a constant $C(N_\ast,n,\rho)>0$ such that 
\begin{equation} \label{PUNG}
\begin{split}
& 0 \leq \varphi_i \leq 1 \; \; \, \text{for any} \; \; \, i \in \mathbf{N}, \\
& \operatorname{supp} \big( \varphi_i \circ F_{x_i}\big) (\cdot', \eta_n) \subset \overline{B_\rho(0')} \; \; \, \text{for any} \; \; \, \eta_n \in (-\rho, \rho) \; \; \, \text{and} \; \; \, i \in \mathbf{N}, \\
& \displaystyle\sum_{i=1}^\infty \, \varphi_i(x) \equiv 1 \; \; \, \text{for any} \; \; \, x \in \Gamma^{\mathbf{R}^n}_{\rho}, \quad \sup_{i \in \mathbf{N}} \, \Vert \nabla \varphi_i \Vert_{L^\infty(\Gamma^{\mathbf{R}^n}_{\rho})} \leq C(N_\ast,n,\rho);
\end{split}
\end{equation}
see e.g. \cite[Proposition 6]{Gu}.
%%%

%%%
\subsection{Cut-off and extension} % Subsection 2.2
\label{sub:CUT}
In this subsection, we assume that $\Omega \subset \mathbf{R}^n$ is a uniformly $C^2$ domain of type $(\alpha, \beta, K)$ with $n \geq 2$.
Let $\rho \in (0, \rho_0/2)$.
For a function $f$ defined in $\Gamma^{\mathbf{R}^n}_\rho \cap \overline{\Omega}$, we define $f_{\mathrm{even}}$ to be the even extension of $f$ to $\Gamma^{\mathbf{R}^n}_\rho$ with respect to the boundary $\Gamma$, i.e.,
\[
f_{\mathrm{even}} \big( \pi x + d(x)\mathbf{n}(\pi x) \big) := f \big( \pi x - d(x)\mathbf{n}(\pi x) \big) \quad \text{for} \quad x \in \Gamma^{\mathbf{R}^n}_\rho \backslash \overline{\Omega}
\]
and $f_{\mathrm{odd}}$ to be the odd extension of $f$ to $\Gamma^{\mathbf{R}^n}_\rho$ with respect to the boundary $\Gamma$, i.e.,
\[
f_{\mathrm{odd}} \big( \pi x + d(x)\mathbf{n}(\pi x) \big) := -f \big( \pi x - d(x)\mathbf{n}(\pi x) \big) \quad \text{for} \quad x \in \Gamma^{\mathbf{R}^n}_\rho \backslash \overline{\Omega}.
\]

For $x \in \Gamma^{\mathbf{R}^n}_\rho$, we further define that
\[
P(x) := \nabla d(\pi x) \otimes \nabla d(\pi x) = \mathbf{n}(\pi x) \otimes \mathbf{n}(\pi x), \quad Q(x) := I - P(x).
\]
It is not hard to see that $P(x)$ represents the normal projection to the direction $\nabla d$ whereas $Q(x)$ represents the tangential projection to the direction $\nabla d$.
For $v\in vBMOL^2(\Omega)$ with $\operatorname{supp} v \subseteq \Gamma^{\mathbf{R}^n}_\rho \cap \Omega$, we define $\overline{v}$ to be its extension of the form
\begin{align} \label{SE}
	\overline{v}(x) := (P v_{\mathrm{odd}}) (x) + (Q v_{\mathrm{even}}) (x)
\end{align}
for $x \in \Gamma^{\mathbf{R}^n}_\rho$. Since $\operatorname{supp} \overline{v} \subset \Gamma^{\mathbf{R}^n}_\rho$, $\overline{v}$ can be viewed as being defined in $\mathbf{R}^n$ with $\overline{v}(x) = 0$ for any $x \in \mathbf{R}^n \setminus \Gamma^{\mathbf{R}^n}_\rho$.

\begin{proposition}[\cite{Gu}] \label{2E}
Let $\Omega \subset \mathbf{R}^n$ be a uniformly $C^2$ domain of type $(\alpha,\beta,K)$ with $n \geq 2$. Let $\varepsilon \in (0,1)$.
For any $\rho \in (0, c_\Omega^\varepsilon/2]$, there exists a constant $C=C(\alpha,\beta,K,\rho)>0$ such that estimates
\begin{align*}
\Vert v_\mathrm{even} \Vert_{BMOL^2(\mathbf{R}^n)} &\leq C \Vert v\Vert_{vBMOL^2(\Omega)}, \\
\Vert \nabla d \cdot v_\mathrm{odd}\Vert_{BMOL^2\left( \mathbf{R}^n \right)} + [\nabla d \cdot v_\mathrm{odd}]_{b^\nu(\Gamma)} &\leq C \Vert v\Vert_{vBMOL^2(\Omega)}
\end{align*}
hold for all $v \in vBMOL^2(\Omega)$ with $\operatorname{supp} v \subset \Gamma^{\mathbf{R}^n}_\rho$ and $\nu>0$.
\end{proposition}

Actually, we could achieve more than Proposition \ref{2E}. We consider a cut off function $\theta \in C_{\mathrm{c}}^\infty(\mathbf{R})$ such that $0 \leq \theta \leq 1$, $\theta(t) = 1$ for any $0 < \lvert t \rvert < 1/2$ and $\theta(t) = 0$ for any $\lvert t \rvert> 3/4$. 
Suppose that $\varepsilon \in (0,1)$ and $\rho \in (0, c_\Omega^\varepsilon/2]$. We set $\theta_\rho := \theta(d(x)/\rho)$. 
An easy check tells us that $\theta_\rho \in C_\mathrm{c}^2(\mathbf{R}^n)$ with $\operatorname{supp} \theta_\rho \subseteq \overline{\Gamma^{\mathbf{R}^n}_\rho}$ and $\theta_\rho(x) = 1$ for $x \in \Gamma^{\mathbf{R}^n}_{\rho/2}$. 
Within this paper, for any subset $D \subset \mathbf{R}^n$ we denote $r_D$ to be the restriction operator in $D$.
For $v \in vBMOL^2(\Omega)$, we let $v_2 := (r_\Omega \theta_\rho) v$ and $v_1 := v - v_2$. 
We extend $v_2$ to $\overline{v_2}$ in the same way as \eqref{SE} in which the normal component of $\overline{v_2}$ is odd with respect to $\Gamma$ and the tangential component of $\overline{v_2}$ is even with respect to $\Gamma$, i.e., we set
\[
\overline{v_2} := P (v_2)_\mathrm{odd} + Q (v_2)_\mathrm{even}.
\]
By \cite[Theorem 1]{Gu}, we see that $\overline{v} := \overline{v_2} + v_1$ is a linear extension of $v$ to $\mathbf{R}^n$ which satisfies
\[
\Vert \overline{v} \Vert_{BMOL^2(\mathbf{R}^n)} + [\nabla d \cdot \overline{v}]_{b^\infty(\Gamma)} \leq C \Vert v \Vert_{vBMOL^2(\Omega)}
\]
with some $C=C(\alpha,\beta,K,\rho)>0$.
In general, multiplication by a smooth function is not bounded in $BMO^\infty(\Omega)$.
However, since we have a bounded linear extension from $vBMOL^2(\Omega)$ to $BMOL^2(\mathbf{R}^n)$, such multiplication is bounded in $vBMOL^2(\Omega)$. 
Since $\rho_0$ depends on $\alpha,\beta,K$ and the reach $R_\ast$, by fixing an arbitrary $\varepsilon \in (0,1)$ and an arbitrary $\rho \in (0,c_\Omega^\varepsilon/2]$, we can deduce the following multiplication rule.

\begin{proposition}[\cite{Gu}] \label{2M}
Let $\Omega \subset \mathbf{R}^n$ be a uniformly $C^2$ domain of type $(\alpha,\beta,K)$ with $n \geq 2$.
Let $\varphi \in C^\gamma(\Omega)$ with $\gamma\in(0,1)$.
For each $v \in vBMOL^2(\Omega)$, the function $\varphi v \in vBMOL^2(\Omega)$ satisfies
\[
	\Vert \varphi v \Vert_{vBMOL^2(\Omega)}
	\leq C\Vert \varphi\Vert_{C^\gamma(\Omega)}
	\Vert v\Vert_{vBMOL^2(\Omega)}
\]
with some constant $C=C(\alpha,\beta,K,R_\ast)>0$ where $R_\ast$ represents the reach of $\Gamma$.
\end{proposition}
%%%

%%%
\subsection{Decomposition of volume potential corresponding to $v$} % Subsection 2.3
\label{sub:VOL}
In this subsection, we assume that $\Omega \subset \mathbf{R}^n$ is a uniformly $C^3$ domain of type $(\alpha,\beta,K)$ with $n \geq 2$.
Let us recall some results that have already been established in \cite{GG22b}.
There exists a constant $\rho_\ast = C(\rho_0,K)>0$ such that all theories in this subsection hold for every $\rho \in (0, \rho_\ast]$.
Let $\rho \in (0, \rho_\ast]$ and $\theta_\rho$ be the cut-off function defined in subsection \ref{sub:CUT}. Since now we assume that $\Gamma$ is uniformly $C^3$, in this case $\theta_\rho \in C^3_\mathrm{c}(\mathbf{R}^n)$ with $\operatorname{supp} \theta_\rho \subset \Gamma^{\mathbf{R}^n}_\rho$ and $\theta_\rho = 1$ for any $x \in \Gamma^{\mathbf{R}^n}_{\rho/2}$.
Still, for $v \in vBMOL^2(\Omega)$ we set $v_2 := (r_\Omega \theta_\rho) v$ and $v_1 := v - v_2$.
By Proposition \ref{2M}, we see that $v_1, v_2 \in vBMOL^2(\Omega)$ satisfying
\[
\Vert v_1 \Vert_{vBMOL^2(\Omega)} + \Vert v_2 \Vert_{vBMOL^2(\Omega)} \leq C \Vert v \Vert_{vBMOL^2(\Omega)}
\]
with some constant $C=C(\alpha,\beta,K,\rho)>0$.

To construct the mapping $v\mapsto q_1$ in Theorem \ref{CSV}, we localize $v_2$ by using the partition of the unity $\{\varphi_i\}^\infty_{i=1}$ associated with the covering $\{ U_{\rho,i} \}^\infty_{i=1}$ of $\Gamma^{\mathbf{R}^n}_\rho$. Here for each $i \in \mathbf{N}$, $U_{\rho,i}$ denotes $U_\rho(x_i)$ with $x_i \in \mathcal{P}_\Gamma$. The corresponding volume potential to $v_1$ can be estimated directly.

\begin{proposition} \label{2V}
There exists a constant $C(\rho)>0$ such that
\begin{align*}
\Vert \nabla q^1_1 \Vert_{BMOL^2(\mathbf{R}^n)} & \leq C(\rho) \Vert v\Vert_{vBMOL^2(\Omega)}, \\
\Vert \nabla q^1_1 \Vert_{L^\infty(\Gamma_{\rho/4}^{\mathbf{R}^n})} &\leq C(\rho) \Vert v\Vert_{vBMOL^2(\Omega)}
\end{align*}
for $q^1_1= E*\operatorname{div} \, v_1$ and $v \in vBMOL^2(\Omega)$.
In particular, 
\[
\left[ \nabla q^1_1 \right]_{b^\nu(\Gamma)} \leq C(\rho) \Vert v \Vert_{vBMOL^2(\Omega)}
\]
for any $\nu < \rho/4$.
\end{proposition}
%
% 原稿2-7 / 10
%
\begin{proof}
By the $BMO$-$BMO$ estimate \cite{FS} and Proposition \ref{2M}, we have the estimate
\[
\left[ \nabla q^1_1 \right]_{BMO(\mathbf{R}^n)} \leq C [ v_1 ]_{BMO(\mathbf{R}^n)} \leq C(\rho) \Vert v \Vert_{vBMOL^2(\Omega)}.
\]
Consider $x \in \Gamma_{\rho/4}^{\mathrm{R}^n}$. Since $\nabla q_1^1$ is harmonic in $\Gamma_{\rho/2}^{\mathrm{R}^n}$ and $B_{\frac{\rho}{4}}(x) \subset \Gamma_{\rho/2}^{\mathrm{R}^n}$, the mean value property for harmonic functions implies that
\[
\nabla q_1^1 (x) = \frac{C_n}{\rho^n} \int_{B_{\frac{\rho}{4}}(x)} \nabla q_1^1 (y) \, dy,
\]
i.e., we can estimate $\lvert \nabla q_1^1 (x) \rvert$ by $C(\rho) \Vert \nabla q_1^1 \Vert_{L^2(\mathbf{R}^n)}$. Since the convolution with $\nabla^2 E$ is bounded in $L^p$ for any $1<p<\infty$, see e.g.\ \cite[Theorem 5.2.7 and Theorem 5.2.10]{Gra}, we have that
\begin{align*}
\Vert \nabla q_1^1 \Vert_{L^2(\mathbf{R}^n)} \leq C \Vert v_1 \Vert_{L^2(\mathbf{R}^n)} \leq C \Vert v \Vert_{L^2(\mathbf{R}^n)}.
\end{align*}
Therefore, the estimate
\[
\lvert \nabla q^1_1 (x) \rvert \leq C(\rho) \Vert v\Vert_{vBMOL^2(\Omega)}
\]
holds for any $x \in \Gamma_{\rho/4}^{\mathrm{R}^n}$.
\end{proof}

For $i \in \mathbf{N}$, we extend $(r_\Omega \varphi_i) v_2$ as in Proposition \ref{2E} to get $\overline{(r_\Omega \varphi_i) v_2}$ and set
\[
\overline{v_2} := \sum^\infty_{i=1} \overline{(r_\Omega \varphi_i) v_2}.
\]
Indeed, this extension is independent of the choice of $\{ \varphi_i \}_{i=1}^\infty$ as long as $\{ \varphi_i \}_{i=1}^\infty$ is a partition of unity for $\Gamma^{\mathbf{R}^n}_\rho$.
We next set
\[
\overline{v_2}^{\tan} := Q \, \overline{v_2} = \sum_{i=1}^\infty Q \, \big( (r_\Omega \varphi_i)_\mathrm{even} (v_2)_{\mathrm{even}} \big).
\]
For $i \in \mathbf{N}$, we have that $\varphi_i \in C^2(U_{\rho,i})$ as in this case $\Gamma$ is of regularity uniformly $C^3$. For simplicity of notation, we denote $(r_\Omega \varphi_i)_\mathrm{even}(v_2)_{\mathrm{even}}$ by $v_{2,i}$. 
By Proposition \ref{2E} and \ref{2M}, we can easily deduce that for any $i \in \mathbf{N}$, $v_{2,i} \in BMOL^2(\mathbf{R}^n)$ with $\operatorname{supp} v_{2,i} \subset U_{\rho,i}$ satisfying the estimate
\begin{align} \label{BLEve}
\Vert v_{2,i} \Vert_{BMOL^2(\mathbf{R}^n)} \leq C(\rho) \Vert v_{2,i} \Vert_{BMOL^2\big( \Gamma^{\mathbf{R}^n}_\rho \big)}  \leq C(\alpha,\beta,K,\rho) \Vert v \Vert_{vBMOL^2(\Omega)}.
\end{align}
We further denote $Q \, v_{2,i}$ by $w_i^{\mathrm{tan}}$. Now, we are ready to construct the suitable potential corresponding to 
\[
\overline{v_2}^{\mathrm{tan}} := \sum_{i=1}^\infty w_i^\mathrm{tan} = \sum_{i=1}^\infty Q \, v_{2,i}.
\]

\begin{proposition}[\cite{GG22b}] \label{VPT}
For every $i \in \mathbf{N}$, there exist bounded linear operators $v \longmapsto p_{i,1}^{\mathrm{tan}}$ and $v \longmapsto p_{i,2}^{\mathrm{tan}}$ from $vBMOL^2(\Omega)$ to $L^\infty(\mathbf{R}^n)$ such that
\[
- \Delta p_i^{\mathrm{tan}} = \operatorname{div} w_i^{\mathrm{tan}} \quad \text{in} \quad U_{2 \rho,i} \cap \Omega
\]
with $p_i^{\mathrm{tan}} := p_{i,1}^{\mathrm{tan}} + p_{i,2}^{\mathrm{tan}}$, $\operatorname{supp} p_{i,1}^{\mathrm{tan}} \subset U_{2 \rho,i}$.
Moreover, there exists a constant $C=C(K,\rho)>0$ such that
\begin{align*}
\left\lVert \nabla p_{i,1}^{\mathrm{tan}}\right\rVert_{BMOL^2(\mathbf{R}^n)} 
&\leq C \left\lVert v_{2,i} \right\rVert_{BMOL^2(\mathbf{R}^n)}, \\
\left\lVert \nabla p_{i,2}^{\mathrm{tan}} \right\rVert_{L^\infty(\mathbf{R}^n)} 
&\leq C \left\lVert v_{2,i} \right\rVert_{L^p(\mathbf{R}^n)}, \\
\sup_{x\in\Gamma, \, r<\rho} r^{-n} \int_{B_r(x)} \left\lvert\nabla d \cdot \nabla p_i^{\mathrm{tan}} \right\rvert \, dy &\leq C \left\lVert v_{2,i}\right\rVert_{BMOL^2(\mathbf{R}^n)}
\end{align*}
with some $p>n$.
\end{proposition}

Having the estimate for the volume potential near the boundary regarding its tangential component, we are left to handle the contribution from $\overline{v}^{\mathrm{nor}}_2:=\overline{v}_2-\overline{v}^{\tan}_2$.
We note that $\overline{v}_2^\mathrm{nor}$ admits decomposition
\[
\overline{v}^{\mathrm{nor}}_2 = \sum^\infty_{i=1} P \, \big( (r_\Omega \varphi_i)_\mathrm{even} (v_2)_{\mathrm{odd}} \big).
\]
For simplicity of notations, for every $i \in \mathbf{N}$ we denote $\nabla d \cdot \big( (r_\Omega \varphi_i)_\mathrm{even} (v_2)_{\mathrm{odd}} \big)$ by $f_{2,i}$.
In this case, for any $i \in \mathbf{N}$ we have that $f_{2,i} \in BMOL^2(\mathbf{R}^n)$ with $\operatorname{supp} f_{2,i} \subset U_{\rho,i}$ satisfying the estimate
\[
\Vert f_{2,i} \Vert_{BMOL^2(\mathbf{R}^n)} \leq C(\rho) \Vert f_{2,i} \Vert_{BMOL^2\big( \Gamma^{\mathbf{R}^n}_\rho \big)}  \leq C(\alpha,\beta,K,\rho) \Vert v \Vert_{vBMOL^2(\Omega)}.
\]
We further denote $f_{2,i} \nabla d$ by $w_i^{\mathrm{nor}}$. 
With a similar idea of proof, we can establish the suitable potential corresponding to $\overline{v}^{\mathrm{nor}}_2$.

\begin{proposition}[\cite{GG22b}] \label{VPN}
For every $i \in \mathbf{N}$, there exist bounded linear operators $v\longmapsto p_{i,1}^{\mathrm{nor}}$ and $v\longmapsto p_{i,2}^{\mathrm{nor}}$ from $vBMOL^2(\Omega)$ to $L^\infty(\mathbf{R}^n)$ such that
\[
- \Delta p_i^{\mathrm{nor}} = \operatorname{div} w_i^{\mathrm{nor}} \quad \text{in} \quad U_{2 \rho,i} \cap \Omega
\]
with $p_i^{\mathrm{nor}} := p_{i,1}^{\mathrm{nor}} + p_{i,2}^{\mathrm{nor}}$, $\operatorname{supp} p_{i,1}^{\mathrm{nor}} \subset U_{2 \rho,i}$.
Moreover, there exists a constant $C=C(K,\rho)>0$ such that
\begin{align*}
\Vert \nabla p_{i,1}^{\mathrm{nor}}\Vert_{BMOL^2(\mathbf{R}^n)} &\leq C \Vert f_{2,i}\Vert_{BMOL^2(\mathbf{R}^n)}, \\
\Vert \nabla p_{i,2}^{\mathrm{nor}}\Vert_{L^\infty(\mathbf{R}^n)} &\leq C \Vert f_{2,i}\Vert_{L^p(\mathbf{R}^n)}, \\
\sup_{x\in\Gamma, \, r<\rho} r^{-n} \int_{B_r(x)} \lvert\nabla d \cdot \nabla p_i^{\mathrm{nor}}\rvert \, dy &\leq C \Vert f_{2,i}\Vert_{BMOL^2(\mathbf{R}^n)}
\end{align*}
with some $p>n$.
\end{proposition}
\begin{remark}
Specifically speaking, Proposition \ref{VPT} is indeed \cite[Proposition 4]{GG22b} whose proof is in \cite[Section 3.4]{GG22b}, Proposition \ref{VPN} is indeed \cite[Proposition 5]{GG22b} whose proof is in \cite[Section 3.5]{GG22b}. For Proposition \ref{VPT}, in the local normal coordinate system at $x_i \in \mathcal{P}_\Gamma$, $p_{i,1}^\mathrm{tan}$ is constructed as 
\begin{align*}
F_{x_i}^{-1}\big( p_{i,1}^\mathrm{tan} \big) &= - \theta A^{-1} (I - BA^{-1})^{-1} \nabla'_{\eta} F_{x_i}^{-1}\big( v_{2,i} \big), \\
A &= -\Delta_\eta, \quad B = \sum_{1\leq i,j\leq n-1} \partial_{\eta_i} (b_{ij} \theta) \partial_{\eta_j}
\end{align*}
with $\theta$ denoting some cut-off function in $V_{4\rho}$ and $F_{x_i} : V_{\rho_0} \to U_{\rho_0}(x_i)$ is the normal coordinate change in $U_{\rho_0}(x_i)$. Since $\partial_{\eta_k} A^{-1} \partial_{\eta_\ell}$ is bounded in $BMO$ \cite{FS} and in $L^p$ ($1<p<\infty$) for any $k,\ell=1,2,...,n$, see e.g. \cite[Theorem 5.2.7 and Theorem 5.2.10]{GraM}, the $BMO \cap L^2$ norm of $\nabla p_{i,1}^\mathrm{tan}$ is controlled by the $BMO \cap L^2$ norm of $v_{2,i}$. On the other hand, $p_{i,2}^\mathrm{tan}$ is constructed as the convolution of the Newton potential $E$ with some function of $v_{2,i}$, we can directly estimate the $L^\infty$ norm for $\nabla p_{i,2}^\mathrm{tan}$ by H$\ddot{\text{o}}$lder's inequality as $\nabla E$ is locally $L^p$-integrable for $p$ sufficiently close to $1$. Similarly, for Proposition \ref{VPN}, $p_{i,1}^\mathrm{nor}$ is constructed as 
\begin{align*}
F_{x_i}^{-1}\big( p_{i,1}^\mathrm{nor} \big) &= - \theta A^{-1} (I - BA^{-1})^{-1} \partial_{\eta_n} F_{x_i}^{-1}\big( f_{2,i} \big), \\
A &= -\Delta_\eta, \quad B = \sum_{1\leq i,j\leq n-1} \partial_{\eta_i} (b_{ij} \theta) \partial_{\eta_j}
\end{align*}
and $p_{i,2}^\mathrm{nor}$ is constructed as the convolution of the Newton potential $E$ with some function of $f_{2,i}$. The $BMO \cap L^2$ estimate for $\nabla p_{i,1}^\mathrm{nor}$ and the $L^\infty$ estimate for $\nabla p_{i,2}^\mathrm{nor}$ can be derived by exactly the same argument as in Proposition \ref{VPT}.
The reason why statements of Proposition \ref{VPT} and Proposition \ref{VPN} look different from \cite[Proposition 4]{GG22b} and \cite[Proposition 5]{GG22b} is because in the case that $\Omega$ is a bounded $C^3$ domain, the space $vBMO(\Omega)$ is continuously embedded in $L^1(\Omega)^n$. By the $L^1 - BMO$ interpolation inequality, the $L^p$ norm of $v_{2,i}$ and $f_{2,i}$ can be controlled by their $vBMO$ norm for any $1<p<\infty$, see e.g. \cite[Lemma 5]{BGST}, \cite[Theorem 2.2]{KW}. If $\Omega$ is a general uniformly $C^3$ domain, then it is not necessary that $vBMO^{\infty,\infty}(\Omega)$ is continuously embedded in $L^1(\Omega)^n$. This is why we state Proposition \ref{VPT} and Proposition \ref{VPN} in a different form from \cite{GG22b}.
\end{remark}
%%%

%%%
\subsection{Volume potential corresponding to $v$} %Subsection 2.4
\label{sub:PT2} 
We admit Proposition \ref{VPT} and \ref{VPN}. 
The corresponding volume potential to $v_2$ can be constructed by summing up the cut-off of $p_i^\mathrm{tan}$ and $p_i^\mathrm{nor}$ to $U_{2\rho,i}$ for all $i$.
We define $\mathcal{Q}_\Gamma := \left\{ U_{\rho,i} \,\middle\vert\, x_i \in \mathcal{P}_\Gamma \right\}$.
\begin{proof}[Proof of Theorem \ref{CSV}]
Let $i \in \mathbf{N}$. We firstly consider the contribution from the tangential part. 
Since $\Gamma$ is uniformly $C^3$, there exists a cut-off function $\theta_i \in C^2_c(U_{2\rho,i})$ such that $\theta_i=1$ in $U_{\rho,i}$, $0 \leq \theta_i \leq 1$ and moreover, the estimate
\[
\Vert \theta_i \Vert_{C^2(U_{2\rho,i})} \leq C(\rho)
\]
holds for some constant $C(\rho)>0$ independent of $i$.
We next set
\[
q^{\mathrm{tan}}_{1,i} := \theta_i p_i^{\mathrm{tan}} + E * \left(p_i^{\mathrm{tan}} \Delta \theta_i + 2\nabla\theta_i \cdot \nabla p_i^{\mathrm{tan}} \right).
\]
Note that Proposition \ref{VPT} ensures that
\begin{align*}
- \Delta q^{\mathrm{tan}}_{1,i} = - \Delta(\theta_i p_i^{\mathrm{tan}}) + p_i^{\mathrm{tan}} \Delta\theta_i + 2\nabla\theta_i \cdot \nabla p_i^{\mathrm{tan}} = \theta_i \operatorname{div} w_i^{\mathrm{tan}} = \operatorname{div} w_i^{\mathrm{tan}}
\end{align*}
in $\Omega$ as $\operatorname{supp}w_i^{\mathrm{tan}} \subset U_{\rho,i}$. We define that
\[
q_1^{\mathrm{tan}} := \sum_{i=1}^\infty q_{1,i}^{\mathrm{tan}}.
\]

Since $\operatorname{supp} p_{i,1}^{\mathrm{tan}} \subset U_{2 \rho,i}$ for all $i$, by Proposition \ref{VPT} we have that
\begin{align*}
\sum_{i=1}^\infty \Vert \nabla (\theta_i p_{i,1}^{\mathrm{tan}}) \Vert_{L^2(\mathbf{R}^n)} 
&\leq \sum_{i=1}^\infty C(\rho) \Big( \big\Vert p_{i,1}^\mathrm{tan} \big\Vert_{L^\infty(U_{2\rho,i})} + \big\Vert \nabla p_{i,1}^\mathrm{tan} \big\Vert_{L^2(U_{2\rho,i})} \Big) \\
&\leq C(K,\rho) \sum_{i=1}^\infty \Vert v_{2,i} \Vert_{L^2(U_{\rho,i})}.
\end{align*}
Since our partition of unity for $\Gamma^{\mathbf{R}^n}_\rho$ is locally finite, see subsection \ref{sub:LOC}, we can deduce that
\[
\sum_{i=1}^\infty \Vert v_{2,i} \Vert_{L^2(U_{\rho,i})} \leq N_\ast \Vert (v_2)_\mathrm{even} \Vert_{L^2\big( \Gamma^{\mathbf{R}^n}_\rho \big)} \leq 8 N_\ast \Vert v_2 \Vert_{L^2(\Omega)} \leq 8 N_\ast \Vert v \Vert_{L^2(\Omega)},
\]
where $N_\ast$ is the constant which characterizes the local finiteness of $\mathcal{Q}_\Gamma$ in the sense that any element of $\mathcal{Q}_\Gamma$ can intersect for at most $N_\ast$ other elements of $\mathcal{Q}_\Gamma$.
Suppose that $B$ is a ball of radius $r(B) < \rho$. If $B$ does not intersect $\Gamma^{\mathbf{R}^n}_{2 \rho}$, then 
\[
\frac{1}{\lvert B \rvert} \int_B \left\lvert \nabla (\theta_i p_{i,1}^{\mathrm{tan}}) - \big( \nabla ( \theta_i p_{i,1}^{\mathrm{tan}}) \big)_B \right\rvert \, dy = 0
\]
for each $i \in \mathbf{N}$. If $B$ intersects $\Gamma^{\mathbf{R}^n}_{2 \rho}$, then $B$ intersects at most $N_\ast$ neighborhoods of $\{ U_{2 \rho}(x_i) \bigm| x_i \in \mathcal{P}_\Gamma \}$, see \cite[Lemma 11]{Gu}. Hence in this case, we have that 
\begin{align*}
&\frac{1}{\lvert B \rvert} \int_B \left\lvert \, \sum_{i=1}^\infty \nabla (\theta_i p_{i,1}^{\mathrm{tan}}) - \bigg( \sum_{i=1}^\infty \nabla (\theta_i p_{i,1}^{\mathrm{tan}}) \bigg)_B \right\rvert \, dy \leq \sum_{\ell=1}^{N_\ast} [ \nabla (\theta_{i_\ell} p_{i_\ell,1}^{\mathrm{tan}} )]_{BMO(\mathbf{R}^n)} \\
&\ \ \leq \sum_{\ell=1}^{N_\ast} C(\rho) \Big( \big\Vert p_{i_\ell,1}^\mathrm{tan} \big\Vert_{L^\infty(U_{2\rho,i_\ell})} + \big\Vert \nabla p_{i_\ell,1}^\mathrm{tan} \big\Vert_{BMOL^2(\mathbf{R}^n)} \Big) \leq C(K,\rho) \sum_{\ell=1}^{N_\ast} \left\lVert v_{2,i_\ell} \right\rVert_{BMOL^2(\mathbf{R}^n)}.
\end{align*}
Hence, by estimate \eqref{BLEve} we deduce that
\[
\left\lVert \sum_{i=1}^\infty \nabla (\theta_i p_{i,1}^{\mathrm{tan}}) \right\rVert_{BMOL^2(\mathbf{R}^n)} \leq C(\alpha,\beta,K,\rho) \Vert v \Vert_{vBMOL^2(\Omega)}.
\]

Note that $\operatorname{supp} \theta_i p_{i,2}^{\mathrm{tan}} \subset U_{2 \rho,i}$ for any $i \in \mathbf{N}$ and for every $x \in \Gamma^{\mathbf{R}^n}_{2\rho}$, $x$ belongs to at most $N_\ast$ elements of $\mathcal{Q}_\Gamma$. By Proposition \ref{VPT} again we can deduce that
\begin{align*}
\bigg[ \sum_{i=1}^\infty \nabla (\theta_i p_{i,2}^{\mathrm{tan}}) \bigg]_{BMO(\mathbf{R}^n)} 
\leq 2 \bigg\Vert \sum_{i=1}^\infty \nabla (\theta_i p_{i,2}^{\mathrm{tan}}) \bigg\Vert_{L^\infty(U_{2\rho,i})} \leq C(K,\rho) N_\ast \sup_{i \in \mathbf{N}} \Vert v_{2,i} \Vert_{BMOL^2(\mathbf{R}^n)}.
\end{align*}
In addition, as
\[
\lVert \nabla (\theta_i p_{i,2}^{\mathrm{tan}}) \rVert_{L^2(\mathbf{R}^n)} \leq \lvert U_{2 \rho,i} \rvert^{\frac{1}{2}} \lVert \nabla (\theta_i p_{i,2}^{\mathrm{tan}}) \rVert_{L^\infty(\mathbf{R}^n)} \leq C(\rho) \Vert v_{2,i} \Vert_{L^p(\mathbf{R}^n)}
\]
with some $p>n$, by the $BMO-L^2$ interpolation inequality we have that
\begin{align*}
&\bigg\lVert \sum_{i=1}^\infty \nabla (\theta_i p_{i,2}^{\mathrm{tan}}) \bigg\rVert_{L^2(\mathbf{R}^n)} \leq \sum_{i=1}^\infty \lVert \nabla (\theta_i p_{i,2}^{\mathrm{tan}}) \rVert_{L^2(\mathbf{R}^n)} \leq C(\rho) \sum_{i=1}^\infty \lVert v_{2,i} \rVert_{L^p(\mathbf{R}^n)} \\
&\ \ \leq C(\rho) N_\ast \lVert (v_2)_\mathrm{even} \rVert_{L^p(\mathbf{R}^n)} \leq C(\rho) N_\ast \Vert (v_2)_\mathrm{even} \Vert_{BMOL^2(\mathbf{R}^n)}.
\end{align*}
Hence, by Proposition \ref{2E} we obtain that
\begin{align} \label{SBMOL}
\left\lVert \sum_{i=1}^\infty \nabla (\theta_i p_i^{\mathrm{tan}}) \right\rVert_{BMOL^2(\mathbf{R}^n)} \leq C(\alpha,\beta,K,\rho) \lVert v  \rVert_{vBMOL^2(\Omega)}.
\end{align}

Let $f_i^{\mathrm{tan}} = p_i^{\mathrm{tan}} \Delta \theta_i + 2 \nabla \theta_i \cdot \nabla p_i^{\mathrm{tan}}$. 
Since $\operatorname{supp} f_i^{\mathrm{tan}} \subset U_{2 \rho,i}$, we have that
\[
\Vert f_i^{\mathrm{tan}} \Vert_{L^1(U_{2 \rho,i})} \leq \vert U_{2 \rho,i} \vert^{\frac{1}{2}} \Vert f_i^{\mathrm{tan}} \Vert_{L^2(U_{2 \rho,i})}.
\]
By the same argument above which proves the estimate (\ref{SBMOL}), we can show that
\[
\bigg[ \sum_{i=1}^\infty f_i^{\mathrm{tan}} \bigg]_{BMO(\mathbf{R}^n)} + \sum_{i=1}^\infty \Vert f_i^{\mathrm{tan}} \Vert_{L^1(\mathbf{R}^n)} \leq C(\alpha,\beta,K,\rho) \Vert v \Vert_{vBMOL^2(\Omega)}.
\]
By the $BMO-L^1$ interpolation (cf. \cite[Lemma 5]{BGST}), we see that the estimate
\begin{equation} \label{ITBL} % ラベルの位置？
\begin{split}
\bigg\Vert \sum_{i=1}^\infty f_i^{\mathrm{tan}} \bigg\Vert_{L^s(\mathbf{R}^n)} 
&\leq C(n) \bigg\Vert \sum_{i=1}^\infty f_i^{\mathrm{tan}} \bigg\Vert_{L^1(\mathbf{R}^n)}^{\frac{1}{s}} \bigg[ \sum_{i=1}^\infty f_i^{\mathrm{tan}} \bigg]_{BMO(\mathbf{R}^n)}^{1 - \frac{1}{s}} \\
&\leq C(n) \bigg\Vert \sum_{i=1}^\infty f_i^{\mathrm{tan}} \bigg\Vert_{BMOL^1(\mathbf{R}^n)}
\end{split}
\end{equation}
holds for any $1<s<\infty$.
Since $\nabla E$ is in $L^{p'}\big( B_{6 \rho}(0) \big)$ for any $1<p'<n/(n-1)$, it follows that
\[
\sup_{x \in \mathbf{R}^n} \left\lvert \nabla E* \bigg( \sum_{i=1}^\infty f_i^{\mathrm{tan}} \bigg) (x) \right\rvert \leq C(\rho) \bigg\Vert \sum_{i=1}^\infty f_i^{\mathrm{tan}} \bigg\Vert_{L^p(\mathbf{R}^n)}
\]
for all $p>n$.
Thus, we have that
\[
\bigg\Vert \nabla E* \bigg( \sum_{i=1}^\infty f_i^{\mathrm{tan}} \bigg) \bigg\Vert_{L^\infty(\mathbf{R}^n)} \leq C(\alpha,\beta,K,\rho) \Vert v \Vert_{vBMOL^2(\Omega)}.
\]
Since the convolution kernel $\nabla E(x)$ is dominated by a constant multiple of $\lvert x \rvert^{-(n-1)}$, by the well-known Hardy-Littlewood-Sobolev inequality, see e.g. \cite[Theorem 1.7]{BCD}, we deduce that
\[
\bigg\Vert \nabla E \ast \bigg( \sum_{i=1}^\infty f_i^{\mathrm{tan}} \bigg) \bigg\Vert_{L^2(\mathbf{R}^n)} \leq C(n) \bigg\Vert \sum_{i=1}^\infty f_i^{\mathrm{tan}} \bigg\Vert_{L^{\frac{2n}{n+2}}(\mathbf{R}^n)}
\]
Hence by estimate (\ref{ITBL}), we see that
\[
\bigg\Vert \nabla E \ast \bigg( \sum_{i=1}^\infty f_i^{\mathrm{tan}} \bigg) \bigg\Vert_{BMOL^2(\mathbf{R}^n)} \leq C(\alpha,\beta,K,\rho) \Vert v \Vert_{vBMOL^2(\Omega)}.
\]
Combine with estimate (\ref{SBMOL}), we finally obtain that
\[
\Vert \nabla q_1^{\mathrm{tan}} \Vert_{BMOL^2(\mathbf{R}^n)} \leq C(\alpha,\beta,K,\rho) \Vert v \Vert_{vBMOL^2(\Omega)}.
\]
By Proposition \ref{VPT}, the control on the boundary with respect to $q_1^{\mathrm{tan}}$ is estimated by
\begin{align*}
\sup_{x\in\Gamma, \, r<\rho} \, r^{-n} \int_{B_r(x)} \vert \nabla d \cdot \nabla q_1^{\mathrm{tan}}\vert \, dy 
&\leq C(K,\rho) N_\ast \sup_{i \in \mathbf{N}} \Vert v_{2,i}\Vert_{vBMOL^2(\Omega)} \\
&\leq C(\alpha,\beta,K,\rho) \Vert v\Vert_{vBMOL^2(\Omega)}
\end{align*}
as the partition $\mathcal{Q}_\Gamma$ is a locally finite open cover of $\Gamma^{\mathbf{R}^n}_{\rho}$.

For the contribution coming from the normal component, we set in a similar way that
\[
q^{\mathrm{nor}}_{1,i} := \theta_i p_i^{\mathrm{nor}} + E * (p_i^{\mathrm{nor}} \Delta \theta_i + 2\nabla\theta_i \cdot \nabla p_i^{\mathrm{nor}} )
\]
and
\[
q_1^{\mathrm{nor}} := \sum_{i=1}^\infty q_{1,i}^{\mathrm{nor}}.
\]
By making use of Proposition \ref{VPN} and repeating the whole argument above that treats the case for $q_1^{\mathrm{tan}}$, we can prove that
\[
\Vert \nabla q_1^{\mathrm{nor}} \Vert_{BMOL^2(\mathbf{R}^n)} + [\nabla d \cdot \nabla q_1^\mathrm{nor}]_{b^\rho(\Gamma)} \leq C(\alpha,\beta,K,\rho) \Vert v \Vert_{vBMOL^2(\Omega)}
\]
in the same way.
Then the volume potential $q_1$ corresponding to $v$ can be constructed as
\[
q_1 := q_1^1 + q_1^{\mathrm{tan}} + q_1^{\mathrm{nor}}
\]
where $q_1^1$ is the volume potential defined in Proposition \ref{2V} corresponding to $v_1$.
By our construction, it can be easily seen that
\begin{align*}
	- \Delta q_1 &= - \Delta q^1_1 - \Delta q^{\mathrm{tan}}_1 - \Delta q^{\mathrm{nor}}_1 \\
	&= \operatorname{div}v_1 + \sum_{i=1}^\infty \operatorname{div} w_i^{\mathrm{tan}} + \sum^\infty_{i=1} \operatorname{div} w_i^{\mathrm{nor}} \\
	&= \operatorname{div} (v_1 + v_2) = \operatorname{div} v 
\end{align*}
in $\Omega$. We are done.
\end{proof}
%%%%%%

%%%%%%
\section{Normal trace for a $L^2$ vector field with bounded mean oscillation} % Section 3
\label{sec:NT}

Let $\mathbf{R}_h^n$ be a perturbed $C^{2+\kappa}$ half space of type $(K)$ with $\kappa \in (0,1)$ and $n \geq 3$. Let $R_\ast>0$ denote the reach of boundary $\Gamma = \partial \mathbf{R}_h^n$.
We have already shown that there exists a constant $C=C(\alpha,\beta,K,R_\ast)>0$ such that the estimate
\[
\Vert w \cdot \mathbf{n} \Vert_{L^\infty(\Gamma)} \leq C \Vert w \Vert_{vBMOL^2\big( \mathbf{R}_h^n \big)}
\]
holds for any $w \in vBMOL^2\big( \mathbf{R}_h^n \big)$, see \cite[Theorem 22]{GG22a}. In this section, we would like to further estimate the $\dot{H}^{-\frac{1}{2}}$ norm of the normal trace of $v$ on $\Gamma$ for $w \in vBMOL^2\big( \mathbf{R}_h^n \big)$.

For simplicity of notations, from now on we define that for any $\delta>0$,
\[
B'_\Gamma(\delta) := \left\{ x \in \Gamma \,\middle\vert\, \lvert x' \rvert < \delta \right\}, \quad B'_\Gamma(\delta)^{\mathrm{c}} := \left\{ x \in \Gamma \,\middle\vert\, \lvert x' \rvert \geq \delta \right\}
\]
and $\Lambda_\delta$ to be the surface area of $B'_\Gamma(\delta)$, i.e.,
\[
\Lambda_\delta := \int_{B'_\Gamma(\delta)} 1 \, d\mathcal{H}^{n-1}(y) = \int_{\lvert y' \rvert<\delta} \big( 1 + \left\lvert \nabla' h (y') \right\rvert^2 \big)^{\frac{1}{2}} \, dy'.
\]
Moreover, for $h \in C_\mathrm{c}^1(\mathbf{R}^{n-1})$, we define the constant $C_s(h) := 1 + \Vert h \Vert_{C^1(\mathbf{R}^{n-1})}$.
%%%

%%%
\subsection{Isomorphism between $\dot{H}^{\frac{1}{2}}(\mathbf{R}^{n-1})$ and $\dot{H}^{\frac{1}{2}}(\Gamma)$} % Subsection 3.1
\label{sub:IsoH}

We would like to firstly give a characterization to the homogeneous fractional Sobolev space $\dot{H}^{\frac{1}{2}}(\mathbf{R}^{n-1})$ before we define $\dot{H}^{\frac{1}{2}}(\Gamma)$.
Let us recall that if $f \in \dot{H}^{\frac{1}{2}}(\mathbf{R}^{n-1})$, then $f \in L_{loc}^2(\mathbf{R}^{n-1})$ and the Gagliardo seminorm
\[
[f]_{\dot{H}^{\frac{1}{2}}(\mathbf{R}^{n-1})}^2 := \int_{\mathbf{R}^{n-1}} \int_{\mathbf{R}^{n-1}} \frac{\left\lvert f(x') - f(y') \right\rvert^2}{\lvert x' - y' \rvert^n} \, dx' \, dy' < \infty.
\]
More precisely, if $f \in \dot{H}^{\frac{1}{2}}(\mathbf{R}^{n-1})$, then it holds that
\[
\Vert f \Vert_{\dot{H}^{\frac{1}{2}}(\mathbf{R}^{n-1})} = C(n) [f]_{\dot{H}^{\frac{1}{2}}(\mathbf{R}^{n-1})} 
\]
with some constant $C(n)$ that depends only on dimension $n$; see e.g. \cite[Proposition 1.37]{BCD}. 
However, the finiteness of the Gagliardo seminorm $[f]_{\dot{H}^{\frac{1}{2}}(\mathbf{R}^{n-1})}$ does not imply that $f \in \dot{H}^{\frac{1}{2}}(\mathbf{R}^{n-1})$. Constant functions in $\mathbf{R}^{n-1}$ are typical counterexamples.
Since we are considering the case where $n \geq 3$, we have a very important property that $\dot{H}^{\frac{1}{2}}(\mathbf{R}^{n-1})$ can be continuously embedded in $L^{\frac{2n-2}{n-2}}(\mathbf{R}^{n-1})$, i.e., there exists a constant $C(n)>0$ such that the estimate
\begin{align} \label{EmHdR}
\Vert f \Vert_{L^{\frac{2n-2}{n-2}}(\mathbf{R}^{n-1})} \leq C(n) \Vert f \Vert_{\dot{H}^{\frac{1}{2}}(\mathbf{R}^{n-1})}
\end{align}
holds for any $f \in \dot{H}^{\frac{1}{2}}(\mathbf{R}^{n-1})$, see e.g. \cite[Theorem 1.38]{BCD}.
As a result, we expect that $\dot{H}^{\frac{1}{2}}(\mathbf{R}^{n-1})$ can be identified with the space
\[
\dot{G}^{\frac{1}{2}}(\mathbf{R}^{n-1}) := \Big\{ f \in L^{\frac{2n-2}{n-2}}(\mathbf{R}^{n-1}) \; \Big\vert \; [f]_{\dot{H}^{\frac{1}{2}}(\mathbf{R}^{n-1})} < \infty \Big\}.
\]

Fortunately, this expectation is indeed true. The Gagliardo seminorm $[\cdot]_{\dot{H}^{\frac{1}{2}}(\mathbf{R}^{n-1})}$ turns out to be a norm on $\dot{G}^{\frac{1}{2}}(\mathbf{R}^{n-1})$. The space $\dot{G}^{\frac{1}{2}}(\mathbf{R}^{n-1})$ is complete with norm $[\cdot]_{\dot{H}^{\frac{1}{2}}(\mathbf{R}^{n-1})}$ and it contains $C_\mathrm{c}^\infty(\mathbf{R}^{n-1})$ as a dense subspace; see \cite[Theorem 3.1]{BGV}.
Since $C_\mathrm{c}^\infty(\mathbf{R}^{n-1}) \subset \dot{H}^{\frac{1}{2}}(\mathbf{R}^{n-1})$, we see that $\| f \|_{\dot{H}^{\frac{1}{2}}(\mathbf{R}^{n-1})}$ and $[f]_{\dot{H}^{\frac{1}{2}}(\mathbf{R}^{n-1})}$ are equivalent for any $f \in C_\mathrm{c}^\infty(\mathbf{R}^{n-1})$.
Hence, the space $\dot{G}^{\frac{1}{2}}(\mathbf{R}^{n-1})$ coincide with the completion of $C_\mathrm{c}^\infty(\mathbf{R}^{n-1})$ in norm $\| \cdot \|_{\dot{H}^{\frac{1}{2}}(\mathbf{R}^{n-1})}$, i.e., it holds that
\[
\dot{H}^{\frac{1}{2}}(\mathbf{R}^{n-1}) = \dot{G}^{\frac{1}{2}}(\mathbf{R}^{n-1}).
\]
The reason why the completion of $C_\mathrm{c}^\infty(\mathbf{R}^{n-1})$ in norm $\| \cdot \|_{\dot{H}^{\frac{1}{2}}(\mathbf{R}^{n-1})}$ is indeed $\dot{H}^{\frac{1}{2}}(\mathbf{R}^{n-1})$ is as follows.
For $\rho>0$, we can construct a cut-off function $\theta_{2 \rho} \in C_\mathrm{c}^\infty(\mathbf{R}^{n-1})$ such that $0 \leq \theta_{2 \rho} \leq 1$ in $\mathbf{R}^{n-1}$, $\theta_{2 \rho} = 1$ in $B_\rho(0')$ and $\theta_{2 \rho} = 0$ in $B_{2 \rho}(0')^\mathrm{c}$.
We note that for any $f \in \mathcal{S}(\mathbf{R}^{n-1})$ where $\mathcal{S}(\mathbf{R}^{n-1})$ denotes the Schwartz space, $\theta_{2 \rho} f$ converges to $f$ in norm $\| \cdot \|_{\dot{H}^{\frac{1}{2}}(\mathbf{R}^{n-1})}$ as $\rho \to \infty$ (This claim will be established within the proof of Proposition \ref{NTL2hs} which appears later in subsection \ref{sub:HmENT}). Therefore, by recalling that the space $\mathcal{S}_0(\mathbf{R}^{n-1})$ is dense in $\dot{H}^{\frac{1}{2}}(\mathbf{R}^{n-1})$ where $\mathcal{S}_0(\mathbf{R}^{n-1})$ denotes the subspace of $\mathcal{S}(\mathbf{R}^{n-1})$ whose Fourier transform vanishes near the origin, see e.g. \cite[Proposition 1.35]{BCD}, we can deduce that
\[
\dot{H}^{\frac{1}{2}}(\mathbf{R}^{n-1}) = \overline{C_\mathrm{c}^\infty(\mathbf{R}^{n-1})}^{\| \cdot \|_{\dot{H}^{\frac{1}{2}}(\mathbf{R}^{n-1})}}.
\]

With respect to a function $f$ defined on $\mathbf{R}^{n-1}$, we could define a corresponding function $f^h$ that is defined on $\Gamma$ by setting
\[
f^h\big( y',h(y') \big) := f(y'), \quad \text{for all} \quad y' \in \mathbf{R}^{n-1}.
\]
Let the mapping $f \mapsto f^h$ be denoted by $T_h$, i.e., $f^h = T_h(f)$. 
Since it is not that natural to consider the Fourier transform on a surface, we follow the characterization of $\dot{H}^{\frac{1}{2}}(\mathbf{R}^{n-1})$ above to define the homogeneous fractional Sobolev space $\dot{H}^{\frac{1}{2}}(\Gamma)$. We say that $f^h \in \dot{H}^{\frac{1}{2}}(\Gamma)$ if $f^h \in L^{\frac{2n-2}{n-2}}(\Gamma)$ satisfies
\[
[f^h]_{\dot{H}^{\frac{1}{2}}(\Gamma)}^2 := \int_\Gamma \int_\Gamma \frac{\left\lvert f^h(x) - f^h(y)\right\rvert^2}{\lvert x - y \rvert^n} \, d\mathcal{H}^{n-1}(x) \, d\mathcal{H}^{n-1}(y) < \infty.
\]
The space $\dot{H}^{\frac{1}{2}}(\Gamma)$ is a Banach space (actually Hilbert space) equipped with the norm $[\cdot]_{\dot{H}^{\frac{1}{2}}(\Gamma)}$. Without causing any ambiguity, we simply use the norm notation $\| \cdot \|_{\dot{H}^{\frac{1}{2}}(\Gamma)}$ to denote $[\cdot]_{\dot{H}^{\frac{1}{2}}(\Gamma)}$.
Finally, we would like to note that
\begin{align*}
\| f \|_{L^{\frac{2n-2}{n-2}}(\mathbf{R}^{n-1})} &\leq \left( \int_{\mathbf{R}^{n-1}} \lvert f(y') \rvert^{\frac{2n-2}{n-2}} \big( 1 + \left\lvert \nabla' h (y') \right\rvert^2 \big)^{\frac{1}{2}} \, dy' \right)^{\frac{n-2}{2n-2}} \\
&= \| f^h \|_{L^{\frac{2n-2}{n-2}}(\Gamma)} \leq C_s(h)^{\frac{n-2}{2n-2}} \Vert f \Vert_{L^{\frac{2n-2}{n-2}}(\mathbf{R}^{n-1})},
\end{align*}
i.e., $f \in L^{\frac{2n-2}{n-2}}(\mathbf{R}^{n-1})$ if and only if $f^h \in L^{\frac{2n-2}{n-2}}(\Gamma)$.

\begin{lemma} \label{IsoH}
The mapping $T_h : \dot{H}^{\frac{1}{2}}(\mathbf{R}^{n-1}) \to \dot{H}^{\frac{1}{2}}(\Gamma)$ is an isomorphism.
\end{lemma}
\begin{proof}
Let $f \in \dot{H}^{\frac{1}{2}}(\mathbf{R}^{n-1})$. We naturally have that
\begin{align*}
&\int_\Gamma \int_\Gamma \frac{\left\lvert f^h(x) - f^h(y) \right\rvert^2}{\lvert x-y \rvert^n} \, d\mathcal{H}^{n-1}(x) \, d\mathcal{H}^{n-1}(y) \\
&\ \ = \int_{\mathbf{R}^{n-1}} \int_{\mathbf{R}^{n-1}} \frac{\left\lvert f^h\big( x',h(x') \big) - f^h\big( y',h(y') \big) \right\rvert^2}{\left\lvert \big( x'-y', h(x') - h(y') \big) \right\rvert^n} \big( 1 + \left\lvert \nabla' h(x') \right\rvert^2 \big)^{\frac{1}{2}} \big( 1 + \left\lvert \nabla' h(y') \right\rvert^2 \big)^{\frac{1}{2}} \, dx' \, dy' \\
&\ \ \leq C_s(h)^2 \int_{\mathbf{R}^{n-1}} \int_{\mathbf{R}^{n-1}} \frac{\left\lvert f(x') - f(y') \right\rvert^2}{\lvert x'-y' \rvert^n} \, dx' \, dy'.
\end{align*}
Hence, the $\dot{H}^{\frac{1}{2}}(\Gamma)$ estimate for $f^h$ reads as
\begin{align} \label{ghgE}
\Vert f^h \Vert_{\dot{H}^{\frac{1}{2}}(\Gamma)} \leq C_s(h) \Vert f \Vert_{\dot{H}^{\frac{1}{2}}(\mathbf{R}^{n-1})}.
\end{align}

Let $f^h \in \dot{H}^{\frac{1}{2}}(\Gamma)$. With respect to $f^h$, we define that $f(y') := f^h\big( y',h(y') \big)$ for any $y' \in \mathbf{R}^{n-1}$. By the mean value theorem, we see that the estimate $\left\lvert h(x') - h(y')\right\rvert \leq \Vert \nabla' h \Vert_{L^\infty(\mathbf{R}^{n-1})} \lvert x'-y' \rvert$ holds for any $x',y' \in \mathbf{R}^{n-1}$. 
As a result, for any $x,y \in \Gamma$ we have that 
\[
\frac{1}{\lvert x' - y' \rvert} \leq \frac{C_s(h)}{\big\lvert \big( x'-y', h(x') - h(y') \big) \big\rvert} = \frac{C_s(h)}{\lvert x-y \rvert}.
\]
Hence, we deduce that
\begin{align*}
&\int_{\mathbf{R}^{n-1}} \int_{\mathbf{R}^{n-1}} \frac{\big\lvert f(x') - f(y') \big\rvert^2}{\lvert x'-y' \rvert^n} \, dx' \, dy' \\
&\leq C_s(h)^n \int_{\mathbf{R}^{n-1}} \int_{\mathbf{R}^{n-1}} \frac{\big\lvert f^h\big( x',h(x') \big) - f^h\big( y',h(y') \big) \big\rvert^2}{\big\lvert \big( x'-y', h(x') - h(y') \big) \big\rvert^n} \big( 1 + \big\lvert \nabla' h(x') \big\rvert^2 \big)^{\frac{1}{2}} \big( 1 + \big\lvert \nabla' h(y') \big\rvert^2 \big)^{\frac{1}{2}} \, dx' \, dy' \\
&= C_s(h)^n \int_\Gamma \int_\Gamma \frac{\left\lvert f^h(x) - f^h(y) \right\rvert^2}{\lvert x-y \rvert^n} \, d\mathcal{H}^{n-1}(x) \, d\mathcal{H}^{n-1}(y).
\end{align*}
Therefore, we obtain that
\begin{align} \label{gghE}
\Vert f \Vert_{\dot{H}^{\frac{1}{2}}(\mathbf{R}^{n-1})} \leq C(n) C_s(h)^{\frac{n}{2}} \Vert f^h \Vert_{\dot{H}^{\frac{1}{2}}(\Gamma)}.
\end{align}
\end{proof}

As an application of Lemma \ref{IsoH}, we have the following embedding result.

\begin{corollary}
$\dot{H}^{\frac{1}{2}}(\Gamma)$ is continuously embedded in $L^{\frac{2n-2}{n-2}}(\Gamma)$.
\end{corollary}
\begin{proof}
Let $f^h \in \dot{H}^{\frac{1}{2}}(\Gamma)$ and $f = T_h^{-1}(f^h)$. Since $\dot{H}^{\frac{1}{2}}(\mathbf{R}^{n-1})$ is continuously embedded in $L^{\frac{2n-2}{n-2}}(\mathbf{R}^{n-1})$, by estimate (\ref{gghE}) we see that
\begin{align*}
\Vert f^h \Vert_{L^{\frac{2n-2}{n-2}}(\Gamma)} &\leq C_s(h)^{\frac{n-2}{2n-2}} \Vert f \Vert_{L^{\frac{2n-2}{n-2}}(\mathbf{R}^{n-1})} \leq C(n) C_s(h)^{\frac{n-2}{2n-2}} \Vert f \Vert_{\dot{H}^{\frac{1}{2}}(\mathbf{R}^{n-1})} \\
&\leq C(n) C_s(h)^{\frac{n^2 - 2}{2n-2}} \Vert f^h \Vert_{\dot{H}^{\frac{1}{2}}(\Gamma)}.
\end{align*}
\end{proof}
%%%

%%%
\subsection{Isomorphism between $\dot{H}^{-\frac{1}{2}}(\mathbf{R}^{n-1})$ and $\dot{H}^{-\frac{1}{2}}(\Gamma)$} % Subsection 3.2
\label{sub:IsoHm}

Let us further recall that a tempered distribution $g$ is said to belong to the homogeneous Sobolev space $\dot{H}^{-\frac{1}{2}}(\mathbf{R}^{n-1})$ if $\widehat{g} \in L_{loc}^1(\mathbf{R}^{n-1})$ satisfies 
\[
\Vert g \Vert_{\dot{H}^{-\frac{1}{2}}(\mathbf{R}^{n-1})}^2 := \int_{\mathbf{R}^{n-1}} \lvert \xi' \rvert^{-1} \big\vert \widehat{g}(\xi) \big\vert^2 \, d\xi' < \infty.
\]
Since $\frac{1}{2} < \frac{n-1}{2}$ for all $n \geq 3$, $\dot{H}^{-\frac{1}{2}}(\mathbf{R}^{n-1})$ is characterized as the dual space of $\dot{H}^{\frac{1}{2}}(\mathbf{R}^{n-1})$, i.e., it holds that
\[
\Vert g \Vert_{\dot{H}^{-\frac{1}{2}}(\mathbf{R}^{n-1})} = \sup_{\Vert f \Vert_{\dot{H}^{\frac{1}{2}}(\mathbf{R}^{n-1})} \leq 1} \, \left\lvert  \int_{\mathbf{R}^{n-1}} g(y') f(y') \, dy' \right\rvert ,
\]
see e.g. \cite[Proposition 1.36]{BCD}.
In order to establish an isomorphism between the dual spaces of $\dot{H}^{\frac{1}{2}}(\mathbf{R}^{n-1})$ and $\dot{H}^{\frac{1}{2}}(\Gamma)$, we need the following multiplication rule for $\dot{H}^{\frac{1}{2}}(\mathbf{R}^{n-1})$.

\begin{proposition} \label{MRHhalf}
Let $\rho>0$ and $\zeta \in C^1(\mathbf{R}^{n-1})$ satisfies that $\zeta$ is identically a constant in $B_\rho(0')^\mathrm{c}$. Then for any $f \in \dot{H}^{\frac{1}{2}}(\mathbf{R}^{n-1})$, we have that $\zeta f \in \dot{H}^{\frac{1}{2}}(\mathbf{R}^{n-1})$ satisfying
\[
\Vert \zeta f \Vert_{\dot{H}^{\frac{1}{2}}(\mathbf{R}^{n-1})} \leq C(n) \big( \Vert \zeta \Vert_{L^\infty(\mathbf{R}^{n-1})} + \rho \Vert \nabla' \zeta \Vert_{L^\infty(\mathbf{R}^{n-1})} \big) \Vert f \Vert_{\dot{H}^{\frac{1}{2}}(\mathbf{R}^{n-1})}.
\]
\end{proposition}
\begin{proof}
Let $f \in \dot{H}^{\frac{1}{2}}(\mathbf{R}^{n-1})$.
For $x',y' \in \mathbf{R}^{n-1}$, by the triangle inequality we see that
\[
\left\lvert \zeta(x') f(x') - \zeta(y') f(y') \right\rvert^2 \leq 2 \Vert \zeta \Vert_{L^\infty(\mathbf{R}^{n-1})}^2 \left\lvert f(x') - f(y') \right\rvert^2 + 2 \left\lvert f(x') \right\rvert^2 \left\lvert \zeta(x') - \zeta(y') \right\rvert^2.
\]
Since $\zeta $ is identically a constant in $B_\rho(0')^\mathrm{c}$, it is sufficient to estimate
\begin{align*}
&\int_{B_\rho(0')} \int_{B_{2\rho}(0')^\mathrm{c}} \left\lvert f(x') \right\rvert^2 \frac{\left\lvert \zeta(x') - \zeta(y') \right\rvert^2}{\lvert x'-y' \rvert^n} \, dx' \, dy' + \int_{B_{2\rho}(0')^\mathrm{c}} \int_{B_\rho(0')} \left\lvert f(x') \right\rvert^2 \frac{\left\lvert  \zeta(x') - \zeta(y') \right\rvert^2}{\lvert x'-y' \rvert^n} \, dx' \, dy' \\
&\ \ + \int_{B_{2\rho}(0')} \int_{B_{2\rho}(0')} \left\lvert f(x')\right\rvert^2 \frac{\left\lvert \zeta(x') - \zeta(y') \right\rvert^2}{\lvert x'-y' \rvert^n} \, dx' \, dy' = I_1 + I_2 + I_3.
\end{align*}
For $\lvert x' \rvert\geq 2\rho$ and $\lvert y' \rvert<\rho$, we have that $\lvert x'-y' \rvert\geq \lvert x' \rvert  - \rho$. Hence, by H$\ddot{\text{o}}$lder's inequality and the embedding of $\dot{H}^{\frac{1}{2}}(\mathbf{R}^{n-1})$ in $L^{\frac{2n-2}{n-2}}(\mathbf{R}^{n-1})$, see estimate (\ref{EmHdR}), we deduce that
\begin{align*}
I_1 &\leq C(n) \rho^{n-1} \Vert \zeta \Vert_{L^\infty(\mathbf{R}^{n-1})}^2 \left( \int_{B_{2\rho}(0')^\mathrm{c}} \left\lvert f(x')\right\rvert^{\frac{2n-2}{n-2}} \, dx' \right)^{\frac{n-2}{n-1}} \left( \int_{B_{2\rho}(0')^\mathrm{c}} \frac{1}{( \lvert x' \rvert - \rho )^{n^2-n}} \, dx' \right)^{\frac{1}{n-1}} \\
&\leq C(n) \rho^{n-1} \Vert \zeta \Vert_{L^\infty(\mathbf{R}^{n-1})}^2 \Vert f \Vert_{L^{\frac{2n-2}{n-2}}(\mathbf{R}^{n-1})}^2 \left( \sum_{i=0}^{n-2} \int_\rho^\infty \frac{\rho^i}{r^{n^2-2n+2+i}} \, dr \right)^{\frac{1}{n-1}} \\
&\leq C(n) \Vert \zeta \Vert_{L^\infty(\mathbf{R}^{n-1})}^2 \Vert f \Vert_{\dot{H}^{\frac{1}{2}}(\mathbf{R}^{n-1})}^2.
\end{align*}
Similarly, for $\lvert y'\rvert \geq 2\rho$ and $\lvert x'\rvert<\rho$, we also have that $\lvert x'-y'\rvert\geq \lvert y'\rvert- \rho$.
By H$\ddot{\text{o}}$lder's inequality and estimate (\ref{EmHdR}) again, we see that
\begin{align*}
I_2 &\leq C(n) \rho \Vert \zeta \Vert_{L^\infty(\mathbf{R}^{n-1})}^2 \Vert f \Vert_{L^{\frac{2n-2}{n-2}}(\mathbf{R}^{n-1})}^2 \int_{B_{2\rho}(0')^\mathrm{c}} \frac{1}{(\lvert y'\rvert- \rho)^n} \, dy' \\
&\leq C(n) \rho \Vert \zeta \Vert_{L^\infty(\mathbf{R}^{n-1})}^2 \Vert f \Vert_{L^{\frac{2n-2}{n-2}}(\mathbf{R}^{n-1})}^2 \left( \sum_{i=0}^{n-2} \int_\rho^\infty \frac{\rho^{n-2-i}}{r^{n-i}} \, dr \right) \leq C(n) \Vert \zeta \Vert_{L^\infty(\mathbf{R}^{n-1})}^2 \Vert f \Vert_{\dot{H}^{\frac{1}{2}}(\mathbf{R}^{n-1})}^2.
\end{align*}
By the mean value theorem, it holds that $\left\lvert  \zeta(x') - \zeta(y') \right\rvert \leq \Vert \nabla' \zeta \Vert_{L^\infty(\mathbf{R}^{n-1})} \lvert x'-y' \rvert$. 
Therefore, by H$\ddot{\text{o}}$lder's inequality and estimate (\ref{EmHdR}) once more, we finally have that
\begin{align*}
I_3 &\leq \Vert \nabla' \zeta \Vert_{L^\infty(\mathbf{R}^{n-1})}^2 \int_{B_{2\rho}(0')} \left\lvert f(x') \right\rvert^2 \int_{B_{2\rho}(0')} \frac{1}{\lvert x'-y' \rvert^{n-2}} \, dy' \, dx' \\
&\leq C(n) \rho^2 \Vert \nabla' \zeta \Vert_{L^\infty(\mathbf{R}^{n-1})}^2 \Vert f \Vert_{L^{\frac{2n-2}{n-2}}(\mathbf{R}^{n-1})}^2 \leq C(n) \rho^2 \Vert \nabla' \zeta \Vert_{L^\infty(\mathbf{R}^{n-1})}^2 \Vert f \Vert_{\dot{H}^{\frac{1}{2}}(\mathbf{R}^{n-1})}^2.
\end{align*}
\end{proof}

In accordance with the duality relation between $\dot{H}^{-\frac{1}{2}}(\mathbf{R}^{n-1})$ and $\dot{H}^{\frac{1}{2}}(\mathbf{R}^{n-1})$, we define the homogeneous Sobolev space $\dot{H}^{-\frac{1}{2}}(\Gamma)$ to be the dual space of $\dot{H}^{\frac{1}{2}}(\Gamma)$ with 
\[
\Vert g^h \Vert_{\dot{H}^{-\frac{1}{2}}(\Gamma)} := \sup_{\Vert f^h \Vert_{\dot{H}^{\frac{1}{2}}(\Gamma)} \leq 1} \, \left\lvert  \int_\Gamma g^h(y) f^h(y) \, d\mathcal{H}^{n-1}(y) \right\rvert 
\]
for $g^h \in \dot{H}^{-\frac{1}{2}}(\Gamma)$.
Based on the fact that $\dot{H}^{\frac{1}{2}}(\mathbf{R}^{n-1})$ is isomorphic to $\dot{H}^{\frac{1}{2}}(\Gamma)$, we can show that their dual spaces are isomorphic with each other as well.

\begin{lemma} \label{IsoHm}
The mapping $T_h : \dot{H}^{-\frac{1}{2}}(\mathbf{R}^{n-1}) \to \dot{H}^{-\frac{1}{2}}(\Gamma)$ is an isomorphism.
\end{lemma}
\begin{proof}
Let $g \in \dot{H}^{-\frac{1}{2}}(\mathbf{R}^{n-1})$ and $g^h = T_h(g)$. For any $f^h \in \dot{H}^{\frac{1}{2}}(\Gamma)$, we have that
\[
\int_\Gamma g^h(y) f^h(y) \, d\mathcal{H}^{n-1}(y) = \int_{\mathbf{R}^{n-1}} g(y') f(y') \big( 1 + \left\lvert \nabla' h(y') \right\rvert^2 \big)^{\frac{1}{2}} \, dy'.
\]
Since $\big( 1 + \left\lvert \nabla' h(y') \right\rvert^2 \big)^{\frac{1}{2}} \in C^1(\mathbf{R}^{n-1})$ and $\big( 1 + \big\lvert \nabla' h(y') \big\rvert^2 \big)^{\frac{1}{2}} = 1$ in $B_{R_h}(0')^\mathrm{c}$, by Proposition \ref{MRHhalf} we deduce that
\begin{align*}
\left\lvert \int_\Gamma g^h(y) f^h(y) \, d\mathcal{H}^{n-1}(y) \right\rvert &\leq \Vert g \Vert_{\dot{H}^{-\frac{1}{2}}(\mathbf{R}^{n-1})} \Vert \big( 1 + \left\lvert \nabla' h(y') \right\rvert^2 \big)^{\frac{1}{2}} f \Vert_{\dot{H}^{\frac{1}{2}}(\mathbf{R}^{n-1})} \\
&\leq C(n) C_s(h) C_1(h) \Vert g \Vert_{\dot{H}^{-\frac{1}{2}}(\mathbf{R}^{n-1})} \Vert f \Vert_{\dot{H}^{\frac{1}{2}}(\mathbf{R}^{n-1})}.
\end{align*}
where 
\[
C_1(h) := 1 + R_h \Vert \nabla'^2 h \Vert_{L^\infty(\mathbf{R}^{n-1})}.
\]
By controlling $\Vert f \Vert_{\dot{H}^{\frac{1}{2}}(\mathbf{R}^{n-1})}$ by $\Vert f^h \Vert_{\dot{H}^{\frac{1}{2}}(\Gamma)}$ using estimate \eqref{gghE}, we obtain that
\begin{align} \label{ghgEm}
\Vert g^h \Vert_{\dot{H}^{-\frac{1}{2}}(\Gamma)} \leq C(n) C_s(h)^{\frac{n}{2}+1} C_1(h) \Vert g \Vert_{\dot{H}^{-\frac{1}{2}}(\mathbf{R}^{n-1})}.
\end{align}

Let $g^h \in \dot{H}^{-\frac{1}{2}}(\Gamma)$ and $g = T_h^{-1}(g^h)$. For any $f \in \dot{H}^{\frac{1}{2}}(\mathbf{R}^{n-1})$, we have that
\begin{align*}
\int_{\mathbf{R}^{n-1}} g(y') f(y') \, dy' &= \int_{\mathbf{R}^{n-1}} g(y') \big( 1 + \big\lvert \nabla' h(y') \big\rvert^2 \big)^{-\frac{1}{2}} f(y') \big( 1 + \big\lvert \nabla' h(y') \big\rvert^2 \big)^{\frac{1}{2}} \, dy' \\
&= \int_\Gamma T_h\left( \big( 1 + \big\lvert \nabla' h(y') \big\rvert^2 \big)^{-\frac{1}{2}} \right) g^h f^h \, d\mathcal{H}^{n-1}(y).
\end{align*}
Since we define $\dot{H}^{-\frac{1}{2}}(\Gamma)$ to be the dual space of $\dot{H}^{\frac{1}{2}}(\Gamma)$, the duality relation says that
\begin{align} \label{Dugf}
\left\lvert \int_{\mathbf{R}^{n-1}} g(y') f(y') \, dy' \right\rvert \leq \left\Vert T_h\left( \big( 1 + \big\lvert \nabla' h(y') \big\rvert^2 \big)^{-\frac{1}{2}} \right) g^h \right\Vert_{\dot{H}^{-\frac{1}{2}}(\Gamma)} \Vert f^h \Vert_{\dot{H}^{\frac{1}{2}}(\Gamma)}.
\end{align}
Note that for any $f_\ast^h \in \dot{H}^{\frac{1}{2}}(\Gamma)$, we have that
\begin{align*}
&\left\lvert \int_\Gamma T_h\left( \big( 1 + \big\lvert \nabla' h(y') \big\rvert^2 \big)^{-\frac{1}{2}} \right) g^h f_\ast^h \, d\mathcal{H}^{n-1}(y) \right\rvert \leq \Vert g^h \Vert_{\dot{H}^{-\frac{1}{2}}(\Gamma)} \left\Vert T_h\left( \big( 1 + \big\lvert \nabla' h(y') \big\rvert^2 \big)^{-\frac{1}{2}} \right) f_\ast^h \right\Vert_{\dot{H}^{\frac{1}{2}}(\Gamma)}.
\end{align*}
Let $f_\ast = T_h^{-1}(f_\ast^h)$. By estimate \eqref{ghgE}, we see that
\[
\left\Vert T_h\left( \big( 1 + \big\lvert \nabla' h(y') \big\rvert^2 \big)^{-\frac{1}{2}} \right) f_\ast^h \right\Vert_{\dot{H}^{\frac{1}{2}}(\Gamma)} \leq C_s(h) \left\Vert \big( 1 + \big\lvert \nabla' h(y') \big\rvert^2 \big)^{-\frac{1}{2}} f_\ast \right\Vert_{\dot{H}^{\frac{1}{2}}(\mathbf{R}^{n-1})}.
\]
Since $\big( 1 + \big\lvert \nabla' h(y') \big\rvert^2 \big)^{-\frac{1}{2}} \in C^1(\mathbf{R}^{n-1})$ and $\big( 1 + \big\lvert \nabla' h(y') \big\rvert^2 \big)^{-\frac{1}{2}}$ is also identically $1$ in $B_{R_h}(0')^\mathrm{c}$, by Proposition \ref{MRHhalf} again we deduce that
\[
\left\Vert \big( 1 + \big\lvert \nabla' h(y') \big\rvert^2 \big)^{-\frac{1}{2}} f_\ast \right\Vert_{\dot{H}^{\frac{1}{2}}(\mathbf{R}^{n-1})} \leq C(n) C_s(h) C_1(h) \Vert f_\ast \Vert_{\dot{H}^{\frac{1}{2}}(\mathbf{R}^{n-1})}.
\]
Hence, by estimate \eqref{gghE} we have that
\begin{align*}
&\left\Vert T_h\left( \big( 1 + \big\lvert \nabla' h(y') \big\rvert^2 \big)^{-\frac{1}{2}} \right) g^h \right\Vert_{\dot{H}^{-\frac{1}{2}}(\Gamma)} = \sup_{\Vert f_\ast^h \Vert_{\dot{H}^{\frac{1}{2}}(\Gamma)} \leq 1} \left\lvert  \int_\Gamma T_h\left( \big( 1 + \big\lvert  \nabla' h(y') \big\rvert^2 \big)^{-\frac{1}{2}} \right) g^h f_\ast^h \, d\mathcal{H}^{n-1}(y) \right\rvert  \\
&\ \ \leq \sup_{\Vert f_\ast^h \Vert_{\dot{H}^{\frac{1}{2}}(\Gamma)} \leq 1} C(n) C_s(h)^2 C_1(h) \Vert g^h \Vert_{\dot{H}^{-\frac{1}{2}}(\Gamma)} \Vert f_\ast \Vert_{\dot{H}^{\frac{1}{2}}(\mathbf{R}^{n-1})} \leq C(n) C_s(h)^{\frac{n}{2}+2} C_1(h) \Vert g^h \Vert_{\dot{H}^{-\frac{1}{2}}(\Gamma)}.
\end{align*}
Finally, by controlling the $\dot{H}^{\frac{1}{2}}(\Gamma)$ norm of $f^h$ by the $\dot{H}^{\frac{1}{2}}(\mathbf{R}^{n-1})$ norm of $f$ using estimate \eqref{ghgE} again, we obtain from estimate \eqref{Dugf} that
\[
\left\lvert  \int_{\mathbf{R}^{n-1}} g(y') f(y') \, dy' \right\rvert  \leq C(n) C_s(h)^{\frac{n}{2}+3} C_1(h) \Vert g^h \Vert_{\dot{H}^{-\frac{1}{2}}(\Gamma)} \Vert f \Vert_{\dot{H}^{\frac{1}{2}}(\mathbf{R}^{n-1})},
\]
i.e.,
\begin{align} \label{gghEm}
\Vert g \Vert_{\dot{H}^{-\frac{1}{2}}(\mathbf{R}^{n-1})} \leq C(n) C_s(h)^{\frac{n}{2}+3} C_1(h) \Vert g^h \Vert_{\dot{H}^{-\frac{1}{2}}(\Gamma)}.
\end{align}
\end{proof}

Since $\dot{H}^{-\frac{1}{2}}(\mathbf{R}^{n-1})$ is the dual space of $\dot{H}^{\frac{1}{2}}(\mathbf{R}^{n-1})$, the fact that $\dot{H}^{\frac{1}{2}}(\mathbf{R}^{n-1})$ is continuously embedded in $L^{\frac{2n-2}{n-2}}(\mathbf{R}^{n-1})$ would imply that $L^{\frac{2n-2}{n}}(\mathbf{R}^{n-1})$ is continuously embedded in $\dot{H}^{-\frac{1}{2}}(\mathbf{R}^{n-1})$, i.e., there exists a constant $C(n)>0$ such that the estimate 
\begin{align} \label{EmHdhR}
\Vert g \Vert_{\dot{H}^{-\frac{1}{2}}(\mathbf{R}^{n-1})} \leq C(n) \Vert g \Vert_{L^{\frac{2n-2}{n}}(\mathbf{R}^{n-1})}
\end{align}
holds for any $g \in \dot{H}^{-\frac{1}{2}}(\mathbf{R}^{n-1})$. By the isomorphism between $\dot{H}^{-\frac{1}{2}}(\mathbf{R}^{n-1})$ and $\dot{H}^{-\frac{1}{2}}(\Gamma)$, we deduce the following embedding result.

\begin{corollary} \label{EmHdhG}
$L^{\frac{2n-2}{n}}(\Gamma)$ is continuously embedded in $\dot{H}^{-\frac{1}{2}}(\Gamma)$.
\end{corollary}
\begin{proof}
We consider $g^h \in \dot{H}^{-\frac{1}{2}}(\Gamma)$ and $f^h \in \dot{H}^{\frac{1}{2}}(\Gamma)$ with $\Vert f^h \Vert_{\dot{H}^{\frac{1}{2}}(\Gamma)} \leq 1$. Let $g = T_h^{-1}(g^h)$ and $f = T_h^{-1}(f^h)$. By Proposition \ref{MRHhalf}, estimate (\ref{EmHdhR}) and estimate (\ref{gghE}), we can deduce that
\begin{align*}
\left\lvert  \int_\Gamma g^h(y) f^h(y) \, d\mathcal{H}^{n-1}(y) \right\rvert  &= \left\lvert  \int_{\mathbf{R}^{n-1}} g(y') f(y') \big( 1 + \big\lvert \nabla' h (y') \big\rvert^2 \big)^{\frac{1}{2}} \, dy' \right\rvert  \\
&\leq \Vert g \Vert_{\dot{H}^{-\frac{1}{2}}(\mathbf{R}^{n-1})} \left\Vert \big( 1 + \big\lvert  \nabla' h (y') \big\rvert^2 \big)^{\frac{1}{2}} f \right\Vert_{\dot{H}^{\frac{1}{2}}(\mathbf{R}^{n-1})} \\
&\leq C(n) C_s(h) C_1(h) \Vert g \Vert_{\dot{H}^{-\frac{1}{2}}(\mathbf{R}^{n-1})} \Vert f \Vert_{\dot{H}^{\frac{1}{2}}(\mathbf{R}^{n-1})} \\
&\leq C(n) C_s(h)^{\frac{n}{2}+1} C_1(h) \Vert g \Vert_{L^{\frac{2n-2}{n}}(\mathbf{R}^{n-1})}.
\end{align*}
Since $\Vert g \Vert_{L^{\frac{2n-2}{n}}(\mathbf{R}^{n-1})} \leq \Vert g^h \Vert_{L^{\frac{2n-2}{n}}(\Gamma)}$, we obtain that
\begin{align} \label{EHdhLp}
\Vert g^h \Vert_{\dot{H}^{-\frac{1}{2}}(\Gamma)} \leq C(n) C_s(h)^{\frac{n}{2}+1} C_1(h) \Vert g^h \Vert_{L^{\frac{2n-2}{n}}(\Gamma)}.
\end{align}
\end{proof}
%%%

%%%
\subsection{$\dot{H}^{-\frac{1}{2}}$ estimate for the normal trace $w \cdot \mathbf{n}$} % Subsection 3.3
\label{sub:HmENT}

Since we are considering the case where $n \geq 3$, similar to the characterization of $\dot{H}^{\frac{1}{2}}(\mathbf{R}^{n-1})$, the space $\dot{H}^1(\mathbf{R}^n)$ has the characterization that $u \in \dot{H}^1(\mathbf{R}^n)$ if and only if $u \in L^{\frac{2n}{n-2}}(\mathbf{R}^n)$ such that $\nabla u \in L^2(\mathbf{R}^n)$. The space $\dot{H}^1(\mathbf{R}^n)$ is complete with norm $\| \nabla u \|_{L^2(\mathbf{R}^n)}$. Moreover, it  contains $C_\mathrm{c}^\infty(\mathbf{R}^n)$ as a dense subspace; see e.g. \cite[Theorem 3.1]{BGV}, \cite{Gal}, \cite{GKL}.

\begin{proposition} \label{LftHdh}
There exists a bounded linear lifting operator $\ell_\mathbf{n} : \dot{H}^{\frac{1}{2}}(\mathbf{R}^{n-1}) \to \dot{H}^1(\mathbf{R}^n)$.
\end{proposition}
\begin{proof}
For $f \in \dot{H}^{\frac{1}{2}}(\mathbf{R}^{n-1})$, we set 
\begin{align} \label{Lftuf}
u_f(x',x_n) := \frac{1}{(2 \pi)^n} \int_{\mathbf{R}^{n-1}} \mathrm{e}^{i x' \cdot \xi'} \left( \mathrm{e}^{- \lvert x_n\rvert \lvert \xi'\rvert} \widehat{f}(\xi') \right) \, d\xi'
\end{align}
for $(x',x_n) \in \mathbf{R}^n$, i.e., $u_f$ is the inverse Fourier transform of $\mathrm{e}^{- \lvert x_n\rvert \lvert \xi'\rvert} \widehat{f}(\xi')$ with respect to $\xi'$.
By the Fourier-Plancherel formula, see e.g. \cite[Theorem 1.25]{BCD}, we have that
\[
\int_{-\infty}^\infty \left\lvert \widehat{u_f}(\xi',\xi_n)  \right\rvert^2 \, d\xi_n = 8 \pi^2 \left\lvert \widehat{f}(\xi') \right\rvert^2 \int_0^\infty \mathrm{e}^{-2 x_n \lvert \xi'\rvert} \, dx_n = 4 \pi^2 \lvert \xi'\rvert^{-1} \left\lvert \widehat{f}(\xi') \right\rvert^2
\]
for $\xi' \neq 0'$. Hence, we can deduce that
\begin{align*}
\int_{\mathbf{R}^n} \lvert \xi' \rvert^2 \left\lvert \widehat{u_f}(\xi',\xi_n) \right\rvert^2 \, d\xi 
= \int_{\mathbf{R}^{n-1}} \lvert \xi' \rvert^2 \int_{-\infty}^\infty \left\lvert \widehat{u_f}(\xi',\xi_n) \right\rvert^2 \, d\xi_n \, d\xi' = 4 \pi^2 \int_{\mathbf{R}^{n-1}} \lvert \xi' \rvert \left\lvert \widehat{f}(\xi') \right\rvert^2 \, d\xi'.
\end{align*}
On the other hand, by the Fourier-Plancherel formula again, we see that
\begin{align*}
\int_{-\infty}^\infty \lvert \xi_n \rvert^2 \big\lvert \widehat{u_f}(\xi',\xi_n) \big\rvert^2 \, d\xi_n &= 4 \pi^2 \big\lvert \widehat{f}(\xi') \big\rvert^2 \int_{-\infty}^\infty \big\lvert \partial_{x_n} \big( \mathrm{e}^{- \lvert x_n \rvert \lvert \xi' \rvert} \big) \big\rvert^2 \, dx_n \\
&= 8 \pi^2 \lvert \xi' \rvert^2 \big\lvert \widehat{f}(\xi') \big\rvert^2 \int_0^\infty \mathrm{e}^{-2 x_n \lvert \xi' \rvert} \, dx_n = 4 \pi^2 \lvert \xi' \rvert \big\lvert \widehat{f}(\xi') \big\rvert^2,
\end{align*}
which further implies that
\begin{align*}
\int_{\mathbf{R}^n} \lvert \xi_n \rvert^2 \left\lvert \widehat{u_f}(\xi',\xi_n) \right\rvert^2 \, d\xi 
= \int_{\mathbf{R}^{n-1}} \int_{-\infty}^\infty \lvert \xi_n \rvert^2 \left\lvert \widehat{u_f}(\xi',\xi_n) \right\rvert^2 \, d\xi_n \, d\xi' = 4 \pi^2 \int_{\mathbf{R}^{n-1}} \lvert \xi' \rvert \left\lvert \widehat{f}(\xi') \right\rvert^2 \, d\xi'.
\end{align*}
Therefore, we obtain that
\[
\Vert u_f \Vert_{\dot{H}^1(\mathbf{R}^n)}^2 = \int_{\mathbf{R}^n} \lvert\xi\rvert^2 \left\lvert \widehat{u_f}(\xi',\xi_n) \right\rvert^2 \, d\xi = 8 \pi^2 \Vert f \Vert_{\dot{H}^{\frac{1}{2}}(\mathbf{R}^{n-1})}^2.
\]
Letting $\ell_\mathbf{n}(f) = u_f$ for any $f \in \dot{H}^{\frac{1}{2}}(\mathbf{R}^{n-1})$ completes the proof of Proposition \ref{LftHdh}.
\end{proof}

We start with the half space problem. If $w \in L^2(\mathbf{R}_+^n)^n$ satisfies $\operatorname{div} w = 0$ in $\mathbf{R}_+^n$, then the normal trace $w \cdot \mathbf{n}$ can be taken in the $\dot{H}^{-\frac{1}{2}}$ sense.

\begin{proposition} \label{NTL2hs}
There exists a constant $C(n)>0$ such that the estimate 
\begin{align} \label{HsENT}
\Vert w \cdot \mathbf{n} \Vert_{\dot{H}^{-\frac{1}{2}}(\mathbf{R}^{n-1})} \leq C(n) \Vert w \Vert_{L^2(\mathbf{R}_+^n)}
\end{align}
holds for any $w \in L^2(\mathbf{R}_+^n)^n$ with $\operatorname{div} w = 0$ in $\mathbf{R}_+^n$.
\end{proposition}
\begin{proof}
Let $\theta_2 \in C_\mathrm{c}^\infty(\mathbf{R})$ be such that $0 \leq \theta_2 \leq 1$ in $\mathbf{R}$, $\theta_2(z) = 1$ for any $\lvert z \rvert<1$ and $\theta_2(z) = 0$ for any $\lvert z \rvert \geq 2$. We define $\theta_{2\rho} \in C_\mathrm{c}^\infty(\mathbf{R}^n)$ by setting
\[
\theta_{2\rho}(x) := \theta_2\left( \frac{\lvert x \rvert}{\rho} \right), \quad x \in \mathbf{R}^n.
\]
For $f \in \dot{H}^{\frac{1}{2}}(\mathbf{R}^{n-1})$, we consider
\[
f_I(x') := \theta_I(x') f(x'), \quad \theta_I(x') := 1 - \theta_{2\rho}(x',0)
\]
for $x' \in \mathbf{R}^{n-1}$. Since $\theta_I = 0$ in $B_\rho(0')$, we decompose
\begin{align*}
&\int_{\mathbf{R}^{n-1}} \int_{\mathbf{R}^{n-1}} \frac{\left\lvert f_I(x') - f_I(y')\right\rvert^2}{\lvert x' - y'\rvert^n} \, dx' \, dy' = \int_{\lvert y'\rvert\geq \frac{\rho}{2}} \int_{\lvert x'\rvert\geq \frac{\rho}{2}} \frac{\left\lvert f_I(x') - f_I(y')\right\rvert^2}{\lvert x' - y'\rvert^n} \, dx' \, dy' \\
&+ \int_{\lvert y'\rvert\geq \rho} \int_{\lvert x'\rvert< \frac{\rho}{2}} \frac{\left\lvert f_I(y')\right\rvert^2}{\lvert x' - y'\rvert^n} \, dx' \, dy' + \int_{\lvert y'\rvert< \frac{\rho}{2}} \int_{\lvert x'\rvert\geq \rho} \frac{\left\lvert f_I(x')\right\rvert^2}{\lvert x' - y'\rvert^n} \, dx' \, dy' = I_1 + I_2 + I_3.
\end{align*}
Since $\lvert y'\rvert\geq \rho$ and $\lvert x'\rvert< \frac{\rho}{2}$ would imply that $\lvert x' - y'\rvert\geq \lvert y'\rvert- \frac{\rho}{2}$, by H$\ddot{\text{o}}$lder's inequality we deduce that
\begin{align*}
I_2 \leq C(n) \rho^{n-1} \Vert f_I \Vert_{L^{\frac{2n-2}{n-2}}\big( B_\rho(0')^\mathrm{c} \big)}^2 \left( \int_{\lvert y'\rvert\geq \rho} \frac{1}{(\lvert y'\rvert- \frac{\rho}{2})^{n^2-n}} \, dx' \right)^{\frac{1}{n-1}} \leq C(n) \Vert f \Vert_{L^{\frac{2n-2}{n-2}}\big( B_\rho(0')^\mathrm{c} \big)}^2.
\end{align*}
By symmetry, $I_3$ follows the same estimate as $I_2$. For $I_1$, we follow the proof of Proposition \ref{MRHhalf} to estimate it as
\begin{align*}
I_1 &\leq 2 \int_{\lvert y'\rvert\geq \frac{\rho}{2}} \int_{\lvert x'\rvert \geq \frac{\rho}{2}} \frac{\left\lvert f(x') - f(y')\right\rvert^2}{\lvert x' - y'\rvert^n} \, dx' \, dy' \\
&\ \ + 2 \int_{\lvert y'\rvert\geq \frac{\rho}{2}} \int_{\lvert x'\rvert \geq \frac{\rho}{2}} \left\lvert f(y')\right\rvert^2 \frac{\left\lvert \theta_I(x') - \theta_I(y')\right\rvert^2}{\lvert x' - y'\rvert^n} \, dx' \, dy' = I_{1,1} + I_{1,2}.
\end{align*}
Since $\theta_I = 1$ in $B_{2\rho}(0')^\mathrm{c}$, we further decompose $I_{1,2}$ as
\begin{align*}
I_{1,2} &= \int_{\frac{\rho}{2} \leq \lvert y'\rvert<2\rho} \int_{\lvert x'\rvert\geq 3\rho} \lvert f(y')\rvert^2 \frac{\left\lvert \theta_I(x') - \theta_I(y') \right\rvert^2}{\lvert x'-y'\rvert^n} \, dx' \, dy' \\
&\ \ + \int_{\lvert y'\rvert \geq 3\rho} \int_{\frac{\rho}{2} \leq\lvert x'\rvert < 2\rho} \left\lvert f(y')\right\rvert^2 \frac{\left\lvert \theta_I(x') - \theta_I(y') \right\rvert^2}{\lvert x'-y'\rvert^n} \, dx' \, dy' \\
&\ \ + \int_{\frac{\rho}{2} \leq \lvert y'\rvert < 2\rho} \int_{\frac{\rho}{2} \leq \lvert x'\rvert < 2\rho} \left\lvert f(y')\right\rvert^2 \frac{\left\lvert \theta_I(x') - \theta_I(y') \right\rvert^2}{\lvert x'-y'\rvert^n} \, dx' \, dy' = J_1 + J_2 + J_3.
\end{align*}
By the proof of Proposition \ref{MRHhalf}, we see that
\[
J_1 + J_3 \leq C(n) \Vert f \Vert_{L^{\frac{2n-2}{n-2}}\big( B_{\rho/2}(0')^\mathrm{c} \big)}^2, \quad J_2 \leq C(n) \Vert f \Vert_{L^{\frac{2n-2}{n-2}}\big( B_{3\rho}(0')^\mathrm{c} \big)}^2.
\]
Therefore, we deduce that
\begin{align*}
&\Vert f_I \Vert_{\dot{H}^{\frac{1}{2}}(\mathbf{R}^{n-1})}^2 \leq 2 \int_{\lvert y' \rvert\geq \frac{\rho}{2}} \int_{\lvert x' \rvert \geq \frac{\rho}{2}} \frac{\left\lvert f(x') - f(y')\right\rvert^2}{\lvert x' - y'\rvert^n} \, dx' \, dy' \\
&\ \ + C(n) \left( \Vert f \Vert_{L^{\frac{2n-2}{n-2}}\big( B_{\rho/2}(0')^\mathrm{c} \big)}^2 + \Vert f \Vert_{L^{\frac{2n-2}{n-2}}\big( B_\rho(0')^\mathrm{c} \big)}^2 + \Vert f \Vert_{L^{\frac{2n-2}{n-2}}\big( B_{3\rho}(0')^\mathrm{c} \big)}^2 \right).
\end{align*}
Since the $L^{\frac{2n-2}{n-2}}$ norm of $f$ is controlled by the $\dot{H}^{\frac{1}{2}}$ norm of $f$, we deduce that $\Vert f_I \Vert_{\dot{H}^{\frac{1}{2}}(\mathbf{R}^{n-1})}$ converges to zero as $\rho$ tends to infinity, i.e., $\theta_{2\rho}(\cdot',0) f$ converges to $f$ in $\dot{H}^{\frac{1}{2}}$ norm as $\rho \to \infty$.

Let us note that the multiplication by a smooth function with compact support is bounded in $\dot{H}^1(\mathbf{R}^n)$. 
Indeed, by H$\ddot{\text{o}}$lder's inequality and the continuous embedding of $\dot{H}^1(\mathbf{R}^n)$ in $L^{\frac{2n}{n-2}}(\mathbf{R}^n)$, we see that the estimate
\begin{equation} \label{MRHd1}
\begin{split}
\Vert \phi u \Vert_{\dot{H}^1(\mathbf{R}^n)} &= \Vert \nabla ( \phi u ) \Vert_{L^2(\mathbf{R}^n)} \leq \Vert u \nabla \phi \Vert_{L^2(\mathbf{R}^n)} + \Vert \phi \nabla u \Vert_{L^2(\mathbf{R}^n)} \\
&\leq \Vert \nabla \phi \Vert_{L^n(\mathbf{R}^n)} \Vert u \Vert_{L^{\frac{2n}{n-2}}(\mathbf{R}^n)} + \Vert \phi \Vert_{L^\infty(\mathbf{R}^n)} \Vert \nabla u \Vert_{L^2(\mathbf{R}^n)} \\
&\leq \big( \Vert \nabla \phi \Vert_{L^n(\mathbf{R}^n)} + \Vert \phi \Vert_{L^\infty(\mathbf{R}^n)} \big) \Vert u \Vert_{\dot{H}^1(\mathbf{R}^n)}
\end{split}
\end{equation}
holds for any $\phi \in C_\mathrm{c}^\infty(\mathbf{R}^n)$. Let $f \in \dot{H}^{\frac{1}{2}}(\mathbf{R}^{n-1})$ and $u_f$ be defined as in expression (\ref{Lftuf}). We consider $u_{f,2\rho} := \theta_{2\rho} u_f$. Since $\Vert \nabla \theta_{2\rho} \Vert_{L^n(\mathbf{R}^n)} \leq C(n) \Vert \theta_2' \Vert_{L^\infty(\mathbf{R})}$, by estimate (\ref{MRHd1}) we have that
\[
\Vert u_{f,2\rho} \Vert_{\dot{H}^1(\mathbf{R}^n)} \leq C(n) \Vert u \Vert_{\dot{H}^1(\mathbf{R}^n)}.
\]
Since $\operatorname{supp} u_{f,2\rho} \subseteq \overline{B_{2\rho}(0)}$, it actually holds that $u_{f,2\rho} \in H^1(\mathbf{R}^n)$ satisfies
\[
\Vert u_{f,2\rho} \Vert_{H^1(\mathbf{R}^n)} \leq C(n) \rho \Vert u_{f,2\rho} \Vert_{\dot{H}^1(\mathbf{R}^n)},
\]
see e.g. \cite[Proposition 1.55]{BCD}. Let $B_{2\rho}^+ := B_{2\rho}(0) \cap \mathbf{R}_+^n$. Since $u_{f,2\rho} \in H^1(B_{2\rho}^+)$, by the Gauss-Green formula, see e.g. \cite[Lemma 1.2.3]{HSo}, it holds that
\begin{align*}
\left\lvert \int_{\mathbf{R}^{n-1}} \theta_{2\rho}(x',0) f(x') w_n(x',0) \, dx' \right\rvert &\leq \int_{B_{2\rho}^+} \left\lvert \nabla u_{f,2\rho} \cdot w \right\rvert \, dx \leq \Vert \nabla u_{f,2\rho} \Vert_{L^2\big( B_{2\rho}^+ \big)} \Vert w \Vert_{L^2\big( B_{2\rho}^+ \big)} \\
&\leq \Vert u_{f,2\rho} \Vert_{\dot{H}^1(\mathbf{R}^n)} \Vert w \Vert_{L^2(\mathbf{R}_+^n)} \leq C(n) \Vert u_f \Vert_{\dot{H}^1(\mathbf{R}^n)} \Vert w \Vert_{L^2(\mathbf{R}_+^n)}
\end{align*}
for any $\rho>0$ and $w \in L^2(\mathbf{R}_+^n)^n$ with $\operatorname{div} w = 0$ in $\mathbf{R}_+^n$. Since we have already shown above that $\theta_{2\rho}(\cdot',0) f$ converges to $f$ in $\dot{H}^{\frac{1}{2}}$ norm as $\rho \to \infty$, by Proposition \ref{LftHdh} we conclude that for any $f \in \dot{H}^{\frac{1}{2}}(\mathbf{R}^{n-1})$ and $w \in L^2(\mathbf{R}_+^n)$ with $\operatorname{div} w = 0$ in $\mathbf{R}_+^n$, 
\[
\left\lvert  \int_{\mathbf{R}^{n-1}} f(x') w_n(x',0) \, dx' \right\rvert  \leq C(n) \Vert f \Vert_{\dot{H}^{\frac{1}{2}}(\mathbf{R}^{n-1})} \Vert w \Vert_{L^2(\mathbf{R}_+^n)},
\]
i.e., 
\[
\Vert w \cdot \mathbf{n} \Vert_{\dot{H}^{-\frac{1}{2}}(\mathbf{R}^{n-1})} \leq C(n) \Vert w \Vert_{L^2(\mathbf{R}_+^n)}.
\]
\end{proof}
\begin{remark}
Instead, if the domain $\Omega \subset \mathbf{R}^n$ that we are considering is bounded $C^2$, then for $w \in L^2(\Omega)^n$ satisfying $\operatorname{div} w = 0$ in $\Omega$, the normal trace $w \cdot \mathbf{n}$ can be taken in the $H^{-\frac{1}{2}}$ sense, i.e., there exists a constant $C$, independent of $w$, such that 
\[
\Vert w \cdot \mathbf{n} \Vert_{H^{-\frac{1}{2}}(\partial \Omega)} \leq C \Vert w \Vert_{L^2(\Omega)}
\]
for any $w \in L^2(\Omega)^n$ with $\operatorname{div} w = 0$ in $\Omega$; see e.g. \cite{LiMa}, \cite{Tem}.
\end{remark}

Now we are ready to consider the perturbed half space problem. For $w \in vBMOL^2\big( \mathbf{R}_h^n \big)$ with $\operatorname{div} w = 0$ in $\mathbf{R}_h^n$, we show that the normal trace $w \cdot \mathbf{n}$ can be taken in the $L^\infty \cap \dot{H}^{-\frac{1}{2}}$ sense.

\begin{proof}[Proof of Lemma \ref{ET}]
Let $w \in vBMOL^2\big( \mathbf{R}_h^n \big)$ with $\operatorname{div} w = 0$ in $\mathbf{R}_h^n$.
Let $\varphi_\ast \in C_\mathrm{c}^\infty(\mathbf{R}^n)$ satisfies $\varphi_\ast = 1$ in $B_{1+\tau_1(h)}(0)$ and $\operatorname{supp} \varphi_\ast \subseteq \overline{B_{2+\tau_1(h)}(0)}$ with $\tau_1(h) := R_h + 4nKR_h^2$. Then we set $w_1 := \varphi_\ast w$ and $w_2 := w - w_1$.
By Proposition \ref{2M}, we see that $w_1, w_2 \in vBMOL^2\big( \mathbf{R}_h^n \big)$ satisfying 
\[
\Vert w_1 \Vert_{vBMOL^2\big( \mathbf{R}_h^n \big)} + \Vert w_2 \Vert_{vBMOL^2\big( \mathbf{R}_h^n \big)} \leq C \Vert \varphi_\ast \Vert_{C^1\big( \mathbf{R}_h^n \big)} \Vert w \Vert_{vBMOL^2\big( \mathbf{R}_h^n \big)}
\]
with some constant $C=C(\alpha,\beta,K,R_\ast)>0$.
For $x \in \Gamma$ such that $\lvert x' \rvert = R_h$, we have that $h(x') = 0$ and $\nabla' h (x') = 0'$. 
Thus, for any $y \in B'_\Gamma(R_h)$, by picking an arbitrary $x' \in \mathbf{R}^{n-1}$ such that $\lvert x' \rvert = R_h$ and considering the mean value theorem, we have that
\[
\left\lvert h(y')\right\rvert =\left\lvert h(y') - h(x')\right\rvert \leq \Vert \nabla'^2 h \Vert_{L^\infty(\Gamma)} \lvert y'-x' \rvert^2 \leq 4nKR_h^2.
\]
As a result, we deduce that 
\[
B_{\rho+\tau_1(h)}(0) \cap \Gamma = B'_\Gamma\big (\rho+\tau_1(h) \big)
\]
for any $\rho>0$.
Since $w_1 = 0$ in $\mathbf{R}_h^n \cap B_{2+\tau_1(h)}(0)^\mathrm{c}$, by Corollary \ref{EmHdhG} and the $L^\infty$ estimate for $w_1 \cdot \mathbf{n}$ \cite[Theorem 22]{GG22a}, we can deduce that
\begin{align*}
\Vert w_1 \cdot \mathbf{n} \Vert_{\dot{H}^{-\frac{1}{2}}(\Gamma)} &\leq \Vert w_1 \cdot \mathbf{n} \Vert_{L^{\frac{2n-2}{n}}(\Gamma)} \leq \Lambda_{2+\tau_1(h)}^{\frac{n}{2n-2}} \Vert w_1 \cdot \mathbf{n} \Vert_{L^\infty(\Gamma)} \\
&\leq C(\alpha,\beta,K,R_\ast) C_s(h)^{\frac{n}{2n-2}} \big( 2+\tau_1(h) \big)^{\frac{n}{2}} \Vert w \Vert_{vBMOL^2\big( \mathbf{R}_h^n \big)}.
\end{align*}

On the other hand, we define that
\begin{eqnarray*}
w_{2,H}(x) :=
\left\{
\begin{array}{lcl}
w_2(x) \quad \text{if} \quad x \in \mathbf{R}_+^n \cap \mathbf{R}_h^n, \\
0 \quad \quad \quad \text{if} \quad x \in \mathbf{R}_+^n \setminus \mathbf{R}_h^n.
\end{array}
\right.
\end{eqnarray*}
Since $w_2 = 0$ in $\mathbf{R}_h^n \cap B_{1+\tau_1(h)}(0)$, we have that $w_{2,H} = 0$ in $\mathbf{R}_+^n \cap B_{1+\tau_1(h)}(0)$.
By Proposition \ref{NTL2hs}, we have that
\[
\Vert w_{2,H} \cdot \mathbf{n}_{\partial \mathbf{R}_+^n} \Vert_{\dot{H}^{-\frac{1}{2}}(\mathbf{R}^{n-1})} \leq C(n) \Vert w_{2,H} \Vert_{L^2(\mathbf{R}_+^n)} \leq C(n) \Vert w_2 \Vert_{L^2\big( \mathbf{R}_h^n \big)}
\]
where $\mathbf{n}_{\partial \mathbf{R}_+^n}$ denotes the outward normal on the boundary $\partial \mathbf{R}_+^n$ of the half space $\mathbf{R}_+^n$.
Since $T_h\big( w_{2,H} \cdot \mathbf{n}_{\partial \mathbf{R}_+^n} \big) = w_2 \cdot \mathbf{n}_\Gamma$ with $\mathbf{n}_\Gamma$ denoting the outward normal on the boundary $\Gamma$ of the perturbed half space $\mathbf{R}_h^n$, by estimate \eqref{ghgEm} we obtain that 
\begin{align*}
\Vert w_2 \cdot \mathbf{n}_\Gamma \Vert_{\dot{H}^{-\frac{1}{2}}(\Gamma)} &\leq C(n) C_s(h)^{\frac{n}{2}+1} C_1(h) \Vert w_{2,H} \cdot \mathbf{n}_{\partial \mathbf{R}_+^n} \Vert_{\dot{H}^{-\frac{1}{2}}(\mathbf{R}^{n-1})} \\
&\leq C(n) C_s(h)^{\frac{n}{2}+1} C_1(h) \Vert w_2 \Vert_{L^2\big( \mathbf{R}_h^n \big)}.
\end{align*}
\end{proof}
%%%%%%

%%%%%%
\section{Estimates for some boundary integrals} % Section 4
\label{sec:EBI}

Let $E$ denotes the fundamental solution of $-\Delta$ in $\mathbf{R}^n$, i.e.,
\[
E(x) :=
\begin{cases}
-\log \lvert x \rvert / 2 \pi & \quad (n = 2), \\
\lvert x \rvert^{2-n} / \big( n(n-2) b_1(n) \big) & \quad (n \geq 3),
\end{cases}
\]
where $b_1(n)$ denotes the volume of the unit ball $B_1(0)$ in $\mathbf{R}^n$.
In this section, we assume that $\mathbf{R}_h^n$ is a perturbed $C^2$ half space of type $(K)$ with boundary $\Gamma = \partial \mathbf{R}_h^n$.
The purpose of this section is to establish several estimates for the trace operator of 
\[
\big( Qg \big)(x) := \int_\Gamma \frac{\partial E}{\partial \mathbf{n}_x}(x-y) g(y) \, d\mathcal{H}^{n-1}, \quad x \in \Gamma^\Omega_{\rho_0}
\]
for $g \in L^\infty(\Gamma) \cap \dot{H}^{-\frac{1}{2}}(\Gamma)$ where $\partial / \partial \mathbf{n}_x$ denotes the exterior normal derivative with respect to $x$-variable.

Let us recall that for a perturbed $C^2$ half space $\mathbf{R}_h^n$ with $h \in C_\mathrm{c}^2(\mathbf{R}^{n-1})$ that is not identically zero, $R_h>0$ represents the smallest positive real number such that $\operatorname{supp} h \subseteq \overline{B_{R_h}(0')}$.
%%%

%%%
\subsection{Estimate for the normal derivative in $y$ of $E$} % Subsection 4.1
\label{sub:ENDE}

To be specific, we give an estimate for the boundary integral of $\frac{\partial E}{\partial \mathbf{n}_y}(x- \cdot)$ for $x \in \mathbf{R}_h^n$ that is close to the boundary $\Gamma$. We have a compatible result with \cite[Lemma 6]{GG22b}, which deals with the case where the domain is bounded.
For $\rho \in (0,\rho_0]$, we define that $\Gamma^\Omega_\rho := \Gamma^{\mathbf{R}^n}_\rho \cap \Omega$.

\begin{lemma} \label{EP} % Lemma 6
Let $\mathbf{R}_h^n$ be a perturbed $C^2$ half space of type $(K)$ with boundary $\Gamma = \partial \mathbf{R}_h^n$, $n \geq 2$ and $\rho \in (0,\rho_0]$.
Then, it holds that
\begin{enumerate}
\item[(i)]
\[
\int_\Gamma \frac{\partial E}{\partial\mathbf{n}_y} (x-y) \, d\mathcal{H}^{n-1}(y) = -\frac{1}{2} \quad \text{for any} \quad x \in \mathbf{R}_h^n,
\]

\item[(i\hspace{-1pt}i)]
\[
\sup_{x \in \Gamma^\Omega_\rho} \, \int_\Gamma \left\lvert  \frac{\partial E}{\partial\mathbf{n}_y} (x-y) \right\rvert  \, d\mathcal{H}^{n-1}(y) < C(n) C_{\ref{EP}}(K,h,\rho)
\]
where
\[
C_{\ref{EP}}(K,h,\rho) := \big( R_h^{n-1} + \rho K + \rho +1 \big) \Vert h \Vert_{C^1(\mathbf{R}^{n-1})} + \Vert \nabla'^2 h \Vert_{L^\infty(\mathbf{R}^{n-1})} + \rho + 1.
\]
\end{enumerate}
\end{lemma}
\begin{proof}
\begin{enumerate}
\item[(i)] 
This follows from the Gauss divergence theorem. For a bounded piecewise $C^1$ domain $D \subset \mathbf{R}^n$, we have that
\[
\int_{\partial D} \frac{\partial E}{\partial\mathbf{n}_y} (x-y) \, d\mathcal{H}^{n-1}(y) = \int_D \Delta_y E(x-y) \, dy
\]
for any $x \in D$.
Since $\Delta_y E(x-y) = -\delta(x-y)$, we obtain that
\[
\int_{\partial D} \frac{\partial E}{\partial\mathbf{n}_y} (x-y) \, d\mathcal{H}^{n-1}(y) = -1
\]
for $x \in D$. Let $x \in \mathbf{R}_h^n$. 
For $R>0$, we define the domain $D_R$ by
\[
D_R := \left\{ (y', y_n) \, \middle\vert \, h(y') < y_n < \Vert h \Vert_{L^\infty(\mathbf{R}^{n-1})} + 2 \lvert x_n \rvert, \, \lvert y' \rvert < R \right\}.
\]
We consider $R > R_h+ \lvert x' \rvert$. By applying the Gauss divergence theorem in $D_R$, we deduce that
\begin{align*}
- 1 &= \int_{y_n = \Vert h \Vert_{L^\infty(\mathbf{R}^{n-1})} + 2 \lvert x_n \rvert, \, \lvert y' \rvert < R} \frac{\partial E}{\partial\mathbf{n}_y} (x-y) \, d\mathcal{H}^{n-1}(y) + \int_{y \in \Gamma, \, \lvert y' \rvert < R} \frac{\partial E}{\partial\mathbf{n}_y} (x-y) \, d\mathcal{H}^{n-1}(y) \\
&\ \ + \int_{0 < y_n < \Vert h \Vert_{L^\infty(\mathbf{R}^{n-1})} + 2 \lvert x_n \rvert, \, \lvert y' \rvert = R} \frac{\partial E}{\partial\mathbf{n}_y} (x-y) \, d\mathcal{H}^{n-1}(y).
\end{align*}
The last term tends to zero naturally as $R \to \infty$. For the first term, since $\mathbf{n}_y$ is pointing straightly upward but $x$ is located below $\left\{ (y',y_n) \, \middle\vert \, y_n = \Vert h \Vert_{L^\infty(\mathbf{R}^{n-1})} + 2 \lvert x_n \rvert \right\}$, the kernel $\big( \partial E / \partial \mathbf{n}_y \big)(x-y)$ in this case is exactly the half of the Poisson kernel $P_{\delta_{h,x_n}}(x'-y')$ with $\delta_{h,x_n} := \Vert h \Vert_{L^\infty(\mathbf{R}^{n-1})} + 2 \lvert x_n \rvert - x_n$. Hence, the first integral on the right hand side tends to $-\frac{1}{2}$ as $R \to \infty$. We therefore obtain (i).

\item[(i\hspace{-1pt}i)]
Let us observe that
\[
-\mathbf{n}_y = -\mathbf{n} \big( y', h(y') \big) = \big( -\nabla' h(y'), 1 \big)/ \omega(y') 
\]
where $\omega(y')=\big( 1+ \left\lvert\nabla' h(y')\right\rvert^2 \big)^{1/2}$ and $\nabla'$ is the gradient in $y'$ variables.
This implies that 
\[
- C(n) \frac{\partial E}{\partial\mathbf{n}_y}(x-y) = \frac{\sigma(y')}{\omega(y') \left( \lvert x'-y'\rvert^2 + \big( x_n- h(y') \big)^2 \right)^{n/2}}
\]
for $y \in \Gamma$ with
\[
\sigma(y') := -\nabla' h(y') \cdot (x'-y') + \big( x_n - h(y') \big) \; \; \text{where} \; \; x_n > h(x'),\ x',y' \in \mathbf{R}^{n-1}.
\]
We set that
\[
K(x', y', x_n) := \frac{\sigma(y')}{\left( \lvert x'-y' \rvert^2 + \big( x_n - h(y') \big)^2 \right)^{n/2}}.
\]
By the Taylor expansion, for $\lvert x'-y' \rvert < 1$ we have that
\[
h(x') = h(y') + \nabla' h(y') \cdot (x'-y') + r(x',y')
\]
with
\[
r(x',y') = (x'-y')^{\mathrm{T}} \cdot \int^1_0 (1-\theta) \Big( \nabla'^2 h \Big) \big( \theta x'+(1-\theta)y' \big) \, d\theta \cdot (x'-y').
\]
We obtain that
\[
\sigma(y') = x_n - h(x') + r(x',y')
\]
with an estimate
\begin{equation} \label{R1}
\left\lvert  r(x',y') \right\rvert  \leq \Vert \nabla'^2 h \Vert_{L^\infty\big( B_1(x') \big)} \lvert x'-y' \rvert^2.
\end{equation}
We decompose $K$ into the sum of a leading term and a remainder term
\[
K(x', y', x_n) = K_0(x', y', x_n) + R(x', y', x_n)
\]
with
\begin{align*}
K_0(x', y', x_n) &:= \frac{x_n - h(x')}{\left( \lvert x'-y' \rvert^2 + \big( x_n - h(y') \big)^2 \right)^{n/2}}, \\
R(x', y', x_n) &:= \frac{r(x', y')}{\left(\lvert x'-y' \rvert^2 + \big( x_n - h(y') \big)^2 \right)^{n/2}}.
\end{align*}
The term $R$ is estimated as
\[
\left\lvert  R(x', y', x_n) \right\rvert  \leq \Vert \nabla'^2 h\Vert_{L^\infty\big( B_1(x') \big)} \lvert x'-y' \rvert^{2-n}
\]
for $\lvert x'-y' \rvert<1$ by estimate \eqref{R1}. 
Hence, 
\[
\int_{\underset{\lvert x'-y' \rvert<1}{y \in \Gamma,}} \left\lvert  \frac{R(x',y',x_n)}{\omega(y')} \right\rvert  \, d\mathcal{H}^{n-1}(y) \leq C(n) \Vert \nabla'^2 h \Vert_{L^\infty(\mathbf{R}^{n-1})}.
\]
Since
\[
\left\lvert \sigma(y')\right\rvert \leq \left\lvert \nabla' h (y')\right\rvert \cdot \lvert x'-y' \rvert + \lvert x_n \rvert + \left\lvert h(y')\right\rvert
\]
for any $y' \in \mathbf{R}^{n-1}$, we have that
\begin{align}
&\begin{aligned} \label{EKW}
\int_{\underset{\lvert y'-x' \rvert \geq 1}{y \in \Gamma,}} \left\lvert \frac{K(x',y',x_n)}{\omega(y')} \right\rvert \, d\mathcal{H}^{n-1}(y) &\leq \int_{\lvert y'-x' \rvert \geq 1} \left\lvert \nabla' h(y')\right\rvert \, dy' + \int_{\lvert y'-x' \rvert \geq 1} \left\lvert h(y')\right\rvert \, dy' \\
&\ \ + \int_{\lvert y'-x'\rvert \geq 1} \frac{\lvert x_n \rvert}{\lvert y'-x' \rvert^n} \, dy'.
\end{aligned}
\end{align}
Since $\operatorname{supp} h \subseteq \overline{B_{R_h}(0')}$, the first two terms of estimate \eqref{EKW} can be estimated by the constant $C(n) R_h^{n-1} \Vert h \Vert_{C^1(\mathbf{R}^{n-1})}$. On the other hand, since the estimate
\[
\left\lvert x_n - h(x') \right\rvert \leq (nK+1) \rho \Vert h \Vert_{C^1(\mathbf{R}^{n-1})} + \rho
\]
holds for any $x \in \Gamma^\Omega_\rho$, The third term of estimate \eqref{EKW} can be controlled by the constant $C(n) \big( (\rho K + \rho + 1) \Vert h \Vert_{C^1(\mathbf{R}^{n-1})} + \rho \big)$.
By (i), we observe that
\begin{align*}
&\frac{C(n)}{2} = \int_{\underset{\lvert y'-x' \rvert \geq 1}{y \in \Gamma,}} \frac{K(x', y', x_n)}{\omega(y')} \, d\mathcal{H}^{n-1}(y) \\
&\ \ + \int_{\underset{\lvert y'-x' \rvert < 1}{y \in \Gamma,}} \frac{K_0(x', y', x_n)}{\omega(y')} \, d\mathcal{H}^{n-1}(y) + \int_{\underset{\lvert y'-x' \rvert< 1}{y \in \Gamma,}} \frac{R(x', y', x_n)}{\omega(y')} \, d\mathcal{H}^{n-1}(y).
\end{align*}
The term $K_0$ is very singular but it is positive for $x \in \Gamma^\Omega_\rho$. Hence, we have that
\[
\int_{\underset{\lvert y'-x' \rvert<1}{y \in \Gamma,}} \frac{K_0(x', y', x_n)}{\omega(y')} \, d\mathcal{H}^{n-1}(y) \leq C(n) C_{\ref{EP}}(K,h,\rho)
\]
where the constant $C_{\ref{EP}}(K,h,\rho)$ has the explicit expression
\[
C_{\ref{EP}}(K,h,\rho) := \big( R_h^{n-1} + \rho K + \rho +1 \big) \Vert h \Vert_{C^1(\mathbf{R}^{n-1})} + \Vert \nabla'^2 h \Vert_{L^\infty(\mathbf{R}^{n-1})} + \rho +1.
\]
Therefore, we finally obtain the estimate 
\[
\int_\Gamma \left\lvert  \frac{\partial E}{\partial\mathbf{n}_y}(x-y) \right\rvert  \, d\mathcal{H}^{n-1}(y)
\leq C(n) C_{\ref{EP}}(K,h,\rho)
\]
which holds for any $x \in \Gamma^\Omega_\rho$.
This completes the proof of Lemma \ref{EP}.
\end{enumerate}
\end{proof}

Before we end this subsection, we would like to give an estimate on the difference between gradients of the signed distance function near the boundary with explicit dependency on the boundary function $h$. This estimate plays an important role in later estimations of various boundary integrals.

\begin{proposition} \label{DbSD}
Let $\mathbf{R}_h^n$ be a perturbed $C^2$ half space with boundary $\Gamma = \partial \mathbf{R}_h^n$ and $n \geq 2$. Then, for any $x \in \Gamma^{\mathbf{R}^n}_{\rho_0}$ and $y \in \Gamma$, it holds that
\[
\left\lvert \nabla d (x) - \nabla d (y) \right\rvert \leq C(n) C_s(h)^2 \Vert \nabla'^2 h \Vert_{L^\infty(\mathbf{R}^{n-1})} \lvert x-y \rvert.
\]
\end{proposition}
\begin{proof}
Let $y \in \Gamma$, $x \in \Gamma^{\mathbf{R}^n}_{\rho_0}$ and $\pi x$ be the unique projection of $x$ on $\Gamma$.
Note that
\begin{align*}
\nabla d(x) - \nabla d(y) &= \nabla d (\pi x) - \nabla d (y) = \left( \frac{\nabla' h (y')}{\omega(y')} - \frac{\nabla' h (\pi x')}{\omega(\pi x')}, \frac{1}{\omega(\pi x')} - \frac{1}{\omega(y')} \right)
\end{align*}
where $\omega(\cdot') =\big( 1 + \big\lvert \nabla' h (\cdot') \big\rvert^2 \big)^{1/2}$ and $\pi x'$ denotes the first $n-1$ component of $\pi x$.
By rewriting
\begin{align*}
\nabla' h (y') \omega(\pi x') - \nabla' h (\pi x') \omega(y') &= \big( \nabla' h (y') - \nabla' h (\pi x') \big) \omega(\pi x') + \nabla' h (\pi x') \big( \omega(\pi x') - \omega(y') \big)
\end{align*}
and applying the mean value theorem to both $\nabla' h (y') - \nabla' h (\pi x')$ and $\omega(\pi x') - \omega(y')$, we can deduce that
\[
\left\lvert \nabla d (\pi x) - \nabla d (y) \right\rvert \leq C(n) \Vert \nabla'^2 h \Vert_{L^\infty(\mathbf{R}^{n-1})} \big( 1 + \Vert \nabla' h \Vert_{L^\infty(\mathbf{R}^{n-1})}^2 \big) \lvert\pi x' - y'\rvert.
\]
Since we must have $\lvert x-y\rvert> \lvert x - \pi x\rvert$ for any $y \in \Gamma$ such that $y \neq \pi x$, by the triangle inequality $\lvert x - \pi x\rvert\geq \lvert y - \pi x\rvert - \lvert x-y\rvert$, we can deduce that the inequality $\lvert y - \pi x\rvert \leq 2\lvert x-y\rvert$ holds for any $y \in \Gamma$. Since obviously $\lvert \pi x' - y'\rvert\leq \lvert \pi x - y\rvert$, we obtain Proposition \ref{DbSD}.
\end{proof}
%%%

%%%
\subsection{Criterion for a class of functions to be in $H^{\frac{1}{2}}(\mathbf{R}^{n-1})$} % Subsection 4.2
\label{sub:CHHh}

Let $n \geq 2$. We say that $f \in H^{\frac{1}{2}}(\mathbf{R}^{n-1})$ if $f \in L^2(\mathbf{R}^{n-1})$ and
\[
\Vert f \Vert_{H^{\frac{1}{2}}(\mathbf{R}^{n-1})}^2 := \Vert f \Vert_{L^2(\mathbf{R}^{n-1})}^2 + \int_{\mathbf{R}^{n-1}} \int_{\mathbf{R}^{n-1}} \frac{\left\lvert f(x') - f(y')\right\rvert^2}{\lvert x'-y' \rvert^n} \, dx' \, dy' < \infty.
\]
Our criterion is similar and compatible with \cite[Lemma 3.2]{GGH}.

\begin{proposition} \label{BSNE}
Let $n \geq 2$ and $\rho>0$. Suppose that $f \in C^1(\mathbf{R}^{n-1})$ satisfies
\[
\operatorname{supp} f \subseteq B_{2 \rho}(0')^{\mathrm{c}}, \quad \left\lvert f(x')\right\rvert \cdot \lvert x' \rvert^{n-1} \leq c_1, \quad \left\lvert \nabla' f (x')\right\rvert \cdot \lvert x' \rvert^n \leq c_2
\]
with some constants $c_1$ and $c_2$ independent of $x' \in \mathbf{R}^{n-1}$. Then the estimate
\[
\Vert f \Vert_{H^{\frac{1}{2}}(\mathbf{R}^{n-1})}^2 \leq C(n) \left( \frac{c_1^2}{\rho^{n-1}} + \frac{c_1^2 + c_2^2}{\rho^n} \right)
\]
holds with some constant $C(n)>0$ depending on $n$ only.
\end{proposition}
\begin{proof}
By a direct calculation, we see that
\[
 \Vert f \Vert_{L^2(\mathbf{R}^{n-1})}^2 \leq c_1^2 \int_{B_{2 \rho}(0')^{\mathrm{c}}} \frac{1}{\lvert y' \rvert^{2n-2}} \, dy' \leq \frac{C(n) c_1^2}{\rho^{n-1}}.
\]
For $y' \in B_\rho(0')$ and $x' \in B_{2 \rho}(0')^\mathrm{c}$, we have that $\lvert x'-y' \rvert\geq \rho$. Hence, we can deduce that
\begin{align*}
\int_{B_\rho(0')} \int_{B_\rho(0')^\mathrm{c}} \frac{\left\lvert f(x') - f(y')\right\rvert^2}{\lvert x'-y' \rvert^n} \, dx' \, dy' &= \int_{B_\rho(0')} \int_{B_{2 \rho}(0')^\mathrm{c}} \frac{\left\lvert f(x')\right\rvert^2}{\lvert x'-y' \rvert^n} \, dx' \, dy' \\
&\leq \frac{C(n)}{\rho} \Vert f \Vert_{L^2\big( B_{2 \rho}(0')^\mathrm{c} \big)}^2 \leq \frac{C(n) c_1^2}{\rho^n}.
\end{align*}
By symmetry, we also have that
\[
\int_{B_\rho(0')} \int_{B_\rho(0')^\mathrm{c}} \frac{\left\lvert f(x') - f(y')\right\rvert^2}{\lvert x'-y' \rvert^n} \, dy' \, dx' \leq \frac{C(n) c_1^2}{\rho^n}.
\]
Hence, it is sufficient to estimate 
\[
I = \int_{B_\rho(0')^\mathrm{c}} \int_{B_\rho(0')^\mathrm{c}} \frac{\left\lvert f(x') - f(y')\right\rvert^2}{\lvert x'-y' \rvert^n} \, dx' \, dy'.
\]

We then follow the similar idea that proves \cite[Lemma 3.2]{GGH}. 
Assume that $\lvert x' \rvert \leq \lvert y' \rvert$ and connect $x'$ and $y'$ by a geodesic curve in $B_{\lvert x' \rvert}(0')^{\mathrm{c}}$.
Since the curve length is less than $(\pi/2)\lvert x'-y' \rvert$, by a fundamental theorem of calculus, we observe that
\begin{align*}
\left\lvert f(x')-f(y') \right\rvert
&\leq (\pi/2) \lvert x'-y' \rvert \cdot \sup\left\{ \left\lvert \nabla' f(z')\right\rvert  \,\middle\vert\,  z' \in B_{\lvert x' \rvert}(0')^{\mathrm{c}} \right\} \\
&\leq (\pi/2) c_2 \lvert x'-y' \rvert\cdot \lvert x' \rvert^{-n}.
\end{align*}
Since the integrand of $I$ is symmetric with respect to $x'$ and $y'$, we now estimate
\[
\frac{I}{2} = \int\!\!\!\!\int_{D_1} + \int\!\!\!\!\int_{D_2} \frac{\left\lvert f(x')-f(y') \right\rvert^2}{\lvert x'-y' \rvert^n} \, dx' \, dy' = I_1 + I_2
\]
with
\begin{align*}
D_1 &= \left\{ (x',y') \,\middle\vert\,  \rho \leq \lvert x' \rvert \leq \lvert y' \rvert,\ \lvert x'-y' \rvert \leq \lvert x' \rvert \right\}, \\
D_2 &= \left\{ (x',y') \,\middle\vert\,  \rho \leq\lvert x' \rvert \leq \lvert y' \rvert,\ \lvert x'-y' \rvert \geq \lvert x' \rvert \right\}.
\end{align*}
To estimate $I_1$, we observe that
\begin{align*}
\frac{\left\lvert  f(x')-f(y') \right\rvert ^2}{\lvert x'-y' \rvert^n} &\leq (\pi/2)^2 c_2^2 \lvert x' \rvert^{-2n} \lvert x'-y' \rvert^{-(n-2)} \\
&\leq (\pi/2)^2 c_2^2 \lvert x' \rvert^{-2n+1+\delta} \lvert x'-y' \rvert^{-(n-2)-1-\delta}
\end{align*}
for any $0< \delta < 1$ since $\lvert x'-y' \rvert\leq\lvert x' \rvert$. Thus,
\begin{align*}
I_1 &\leq (\pi/2)^2 c_2^2 \int_{B_\rho(0')^{\mathrm{c}}} \int_{B_\rho(x')} \lvert y'-x' \rvert^{-(n-2)} \, dy' \lvert x' \rvert^{-2n} \, dx' \\
&\ \ + (\pi/2)^2 c_2^2 \int_{B_\rho(0')^{\mathrm{c}}} \int_{B_\rho(x')^{\mathrm{c}}} \lvert y'-x' \rvert^{-(n-2)-1-\delta} \, dy' \lvert x' \rvert^{-2n+1+\delta} \, dx' < \frac{c_2^2 C(n)}{\rho^n} \left( 1 + \frac{1}{\delta (n-\delta)} \right).
\end{align*}
To estimate $I_2$, we observe that
\[
\frac{\left\lvert  f(x')-f(y') \right\rvert ^2}{\lvert x'-y' \rvert^n} \leq 2 \frac{\left\lvert f(x')\right\rvert ^2+\left\lvert f(y')\right\rvert ^2}{\lvert x'-y' \rvert^n} \leq 4 c_1^2 \lvert x'-y' \rvert^{-n} \lvert x' \rvert^{-(2n-2)}
\]
since $\lvert x' \rvert\leq\lvert y' \rvert$. Since $\lvert x'-y' \rvert \geq \lvert x' \rvert$ in this case, we have that
\[
\lvert x'-y' \rvert^{-n} \lvert x' \rvert^{-(2n-2)} \leq \lvert x'-y' \rvert^{-(n-2)} \lvert x' \rvert^{-2n}
\]
and 
\[
\lvert x'-y' \rvert^{-n} \lvert x' \rvert^{-(2n-2)} \leq \lvert x'-y' \rvert^{-(n-\delta)} \lvert x' \rvert^{-(2n-2)-\delta}
\]
for any $0<\delta<1$.
Hence,
\begin{align*}
I_2 &\leq 4 c_1^2 \int_{B_\rho(0')^{\mathrm{c}}} \int_{B_\rho(x')} \lvert y'-x' \rvert^{-(n-2)} \, dy' \lvert x' \rvert^{-2n} \, dx' \\
&\ \ + 4 c_1^2 \int_{B_\rho(0')^{\mathrm{c}}} \int_{B_\rho(x')^{\mathrm{c}}} \lvert y'-x' \rvert^{-(n-\delta)} \, dy' \lvert x' \rvert^{-(2n-2)-\delta} \, dx' < \frac{c_1^2 C(n)}{\rho^n} \cdot \frac{n - (n-2)\delta - \delta^2}{(1-\delta)(n+\delta-1)}.
\end{align*}
\end{proof}
%%%

%%%
\subsection{$L^\infty$ estimate for the trace operator of $Qg$} % Subsection 4.3
\label{sub:EPgS}

For $\rho \in (0,\infty)$, we let $1'_{B_\rho(0')}$ to be the characteristic function associated with the open ball $B_\rho(0')$ in $\mathbf{R}^{n-1}$, i.e., we define that
\[
1'_{B_\rho(0')}(x') :=
\begin{cases}
1 & \quad \text{if} \quad x' \in B_\rho(0'), \\
0 & \quad \text{if} \quad x' \in \mathbf{R}^{n-1} \setminus B_\rho(0').
\end{cases}
\]
For $g \in L^\infty(\Gamma)$, we decompose $g$ into the sum of the curved part $g_1$ and the straight part $g_2$ where
\begin{align*}
g_1\big( x',h(x') \big) &:= 1'_{B_{2 R_h}(0')}(x') g\big( x', h(x') \big), \\
g_2\big( x',h(x') \big) &:= g\big( x',h(x') \big) - g_1\big( x',h(x') \big)
\end{align*}
for any $x' \in \mathbf{R}^{n-1}$. Note that $g_1, g_2 \in L^\infty(\Gamma)$. With respect to $g_2$, we further define $g_2^H \in L^\infty(\partial \mathbf{R}_+^n)$ by setting
\begin{align} \label{HPg2}
g_2^H(x', 0) :=
\begin{cases}
0 & \quad \text{if} \quad \lvert x' \rvert < 2R_h, \\
g_2(x',0) & \quad \text{if} \quad \lvert x' \rvert \geq 2R_h
\end{cases}
\end{align}
for $x' \in \mathbf{R}^{n-1}$.

\begin{theorem} \label{TBIE}
Let $\mathbf{R}_h^n$ be a perturbed $C^2$ half space with boundary $\Gamma = \partial \mathbf{R}_h^n$ and $n \geq 3$. Moreover, let us assume that $h \in C_\mathrm{c}^2(\mathbf{R}^{n-1})$ satisfies the smallness condition
\[
R_h^{\frac{2n-1}{2n}} < \frac{1}{2}.
\]
Then, the boundary trace of $Qg$ is of the form
\[
\gamma \big( Qg \big) \big( x',h(x') \big) = \frac{1}{2} g\big( x',h(x') \big) - \big( Sg \big) \big( x',h(x') \big)
\]
for $g \in L^\infty(\Gamma) \cap \dot{H}^{-\frac{1}{2}}(\Gamma)$, where $S$ is a bounded linear operator from $L^\infty(\Gamma) \cap \dot{H}^{-\frac{1}{2}}(\Gamma)$ to $L^\infty(\Gamma)$ satisfying
\[
\Vert S \Vert_{L^\infty(\Gamma) \cap \dot{H}^{-\frac{1}{2}}(\Gamma) \to L^\infty(\Gamma)} \leq C^\ast_0(n) C_s(h)^{n+6} \big( C_{\ast,1}(h) + C_{\ast,2}(h) \big)
\]
with some specific constant $C^\ast_0(n)$ depending only on $n$ and
\begin{align*}
C_s(h) &:= 1 + \Vert h \Vert_{C^1(\mathbf{R}^{n-1})}, \quad C_1(h) := 1 + R_h \Vert \nabla'^2 h \Vert_{L^\infty(\mathbf{R}^{n-1})}, \\
C_{\ast,1}(h) &:= C_1(h)^3 (1 + R_h^{\frac{1}{4}}) \Big( R_h^{\frac{1}{2}} \Vert \nabla'^2 h \Vert_{L^\infty(\mathbf{R}^{n-1})} + R_h^{\frac{5}{2}} \Vert \nabla'^2 h \Vert_{L^\infty(\mathbf{R}^{n-1})}^3 \Big), \\
C_{\ast,2}(h) &:= \big( R_h + R_h^{\frac{1}{2n}} \big) \Vert \nabla'^2 h \Vert_{L^\infty(\mathbf{R}^{n-1})} + (R_h^{n-1} + 1) \Vert h \Vert_{C^1(\mathbf{R}^{n-1})}.
\end{align*}
\end{theorem}
\begin{proof}
For $x \in \Gamma^\Omega_{\rho_0}$, we decompose $g$ into the straight part $g_2$ and the curved part $g_1$, i.e.,
\begin{align*}
\big( Q g \big) (x) &= \int_{B'_\Gamma(2R_h)^\mathrm{c}} \frac{\partial E}{\partial \mathbf{n}_x}(x-y) g_2(y) \, d\mathcal{H}^{n-1}(y) \\
&\ \ + \int_{B'_\Gamma(2R_h)} \frac{\partial E}{\partial \mathbf{n}_x}(x-y) g_1(y) \, d\mathcal{H}^{n-1}(y) = I_1(x) + I_2(x).
\end{align*}
Moreover, we further decompose $I_2(x)$ as
\begin{align*}
I_2(x) &= \int_{B'_\Gamma(2R_h)} \big\{ \big( \nabla d (x) - \nabla d (y) \big) \cdot \nabla E (x-y) \big\} g_1(y) \, d\mathcal{H}^{n-1}(y) \\
&\ \ + \int_{B'_\Gamma(2R_h)} \frac{\partial E}{\partial \mathbf{n}_y}(x-y) g_1(y) \, d\mathcal{H}^{n-1}(y) = I_{2,1}(x) + I_{2,2}(x).
\end{align*}

Suppose that $x \in \Gamma^\Omega_{\rho_0}$ with $\lvert x' \rvert \geq 2R_h$. 
Since $\lvert x' \rvert \geq 2R_h$ is the straight part of $\Gamma$, we have that
\[
I_1(x) = -\int_{\lvert y' \rvert \geq 2R_h} P_{x_n}(x'-y') g_2(y') \, dy' = -\int_{\mathbf{R}^{n-1}} P_{x_n}(x'-y') g_2^H(y') \, dy',
\]
where $P_{x_n}$ denotes the Poisson kernel.
Let $x$ tends $x_0$ on the boundary, in this case we have that $I_1(x)$ tends to $\frac{1}{2} g_2^H(x_0)$, which is indeed $\frac{1}{2} g(x_0)$. 
We then estimate $I_{2,1}(x_0)$ for $x_0 \in \Gamma$ with $\lvert x_0' \rvert\geq 2R_h$. By Proposition \ref{DbSD}, $I_{2,1}(x_0)$ can be estimated as
\begin{align} \label{EI21}
\left\lvert I_{2,1}(x_0) \right\rvert \leq C(n) C_s(h)^3 \Vert \nabla'^2 h \Vert_{L^\infty(\mathbf{R}^{n-1})} \Vert g \Vert_{L^\infty(\Gamma)} \int_{\lvert y' \rvert<2R_h} \frac{1}{\lvert x_0' - y' \rvert^{n-2}} \, dy'.
\end{align}
If $\lvert x_0'\rvert \geq 3R_h$, then we have that $\lvert x_0' - y'\rvert \geq R_h$. In this case,
\[
\int_{\lvert y' \rvert<2R_h} \frac{1}{\lvert x_0' - y' \rvert^{n-2}} \, dy' \leq R_h^{-(n-2)} \left\lvert B_{2R_h}(0') \right\rvert \leq C(n) R_h.
\]
If $\lvert x_0' \rvert < 3R_h$, then in this case we have the estimate
\begin{align} \label{EI21x0s}
\int_{\lvert y' \rvert<2R_h} \frac{1}{\lvert x_0' - y' \rvert^{n-2}} \, dy' \leq \int_{\lvert x_0' - y' \rvert<5R_h} \frac{1}{\lvert x_0' - y' \rvert^{n-2}} \, dy' \leq C(n) R_h.
\end{align}
Hence, for $x_0 \in \Gamma$ with $\lvert x_0' \rvert \geq 2R_h$, we obtain that
\[
\left\lvert I_{2,1}(x_0) \right\rvert \leq C(n) C_s(h)^3 R_h \Vert \nabla'^2 h \Vert_{L^\infty(\mathbf{R}^{n-1})} \Vert g \Vert_{L^\infty(\Gamma)}.
\]

Next, we estimate $I_{2,2}(x_0)$ for $x_0 \in \Gamma$ with $\lvert x_0' \rvert \geq 2R_h$. Since $\operatorname{supp} h \subseteq \overline{B_{R_h}(0')}$, for $y \in \Gamma$ with $R_h< \lvert y' \rvert<2R_h$ and $x_0 \in \Gamma$ with $\lvert x_0' \rvert \geq 2R_h$, we actually have that
\[
\frac{\partial E}{\partial \mathbf{n}_y} (x_0-y) = 0.
\]
Thus, for $x_0 \in \Gamma$ with $\lvert x_0' \rvert \geq 2R_h$,
\[
I_{2,2}(x_0) = \int_{B'_\Gamma(R_h)} \frac{\partial E}{\partial \mathbf{n}_y}(x_0-y) g_1(y) \, d\mathcal{H}^{n-1}(y).
\]
By estimate (\ref{R1}), we have that
\[
\left\lvert I_{2,2}(x_0) \right\rvert \leq C_s(h) \Vert \nabla'^2 h \Vert_{L^\infty(\mathbf{R}^{n-1})} \Vert g \Vert_{L^\infty(\Gamma)} \int_{\lvert y'\rvert<R_h} \frac{1}{\lvert x_0' - y'\rvert^{n-2}} \, dy'.
\]
Since $\lvert x_0'\rvert\geq 2R_h$, it holds that $\lvert x_0' - y'\rvert\geq R_h$ for $\lvert y'\rvert<R_h$. Hence, we obtain the estimate for $\left\lvert I_{2,2}(x_0) \right\rvert$ for the case where $\lvert x_0'\rvert\geq 2R_h$, i.e.,
\begin{align*}
\left\lvert I_{2,2}(x_0)\right\rvert &\leq C_s(h) R_h^{-(n-2)} \Vert \nabla'^2 h \Vert_{L^\infty(\mathbf{R}^{n-1})} \Vert g \Vert_{L^\infty(\Gamma)} \int_{\lvert y' \rvert<R_h} 1 \, dy' \\
&\leq C(n) C_s(h) R_h \Vert \nabla'^2 h \Vert_{L^\infty(\mathbf{R}^{n-1})} \Vert g \Vert_{L^\infty(\Gamma)}.
\end{align*}
Therefore, for $x_0 \in \Gamma$ with $\lvert x_0' \rvert \geq 2R_h$, by setting
\[
\big( S g \big)(x_0) = - \int_{B'_\Gamma(2R_h)} \frac{\partial E}{\partial \mathbf{n}_{x_0}}(x_0-y) g(y) \, d\mathcal{H}^{n-1}(y),
\]
we get that
\[
\gamma \big( Qg) \big) (x_0) = \frac{1}{2} g(x_0) - \big( Sg \big) (x_0)
\]
with
\[
\Vert S \Vert_{\mathrm{op}} \leq C(n) C_s(h)^3 R_h \Vert \nabla'^2 h \Vert_{L^\infty(\mathbf{R}^{n-1})}.
\]

Suppose now that $x \in \Gamma^\Omega_{\rho_0}$ with $\lvert x' \rvert < 2R_h$. 
There exists a bounded $C^2$ domain $\Omega_c \subset \mathbf{R}_h^n$ such that $\partial \Omega_c \cap \Gamma = B'_\Gamma(2R_h)$.
Let us recall a standard result concerning the double layer potential, see e.g. \cite[Lemma 6.17]{HL}. Let $f \in L^\infty(\partial \Omega_c)$, then the boundary trace of the double layer potential
\[
(P f) (z) = \int_{\partial \Omega_c} \frac{\partial E}{\partial \mathbf{n}_y} (z-y) f(y) \, d\mathcal{H}^{n-1}(y), \quad z \in \Omega_c
\]
is of the form
\[
\gamma \big( P f \big) (w) = \frac{1}{2} f(w) + \int_{\partial \Omega_c} \frac{\partial E}{\partial \mathbf{n}_y} (w-y) f(y) \, d\mathcal{H}^{n-1}(y)
\]
for $w \in \partial \Omega_c$. We define $g_c \in L^\infty(\partial \Omega_c)$ by letting
\begin{eqnarray*} \label{BDBE1}
g_c(w) =
\left\{
\begin{array}{lcl}
g_1(w) \quad \text{for} \quad w \in \partial \Omega_c \cap \Gamma, \\
0 \quad \quad \quad \text{for} \quad w \in \partial \Omega_c \setminus \Gamma.
\end{array}
\right.
\end{eqnarray*}
Thus, for any $z \in \Omega_c$ we have that
\[
I_2(z) = I_{2,1}(z) + \big(P g_c \big) (z).
\]
Let $x$ tends to $x_0$ on the boundary, we deduce that
\[
I_2(x_0) = I_{2,1}(x_0) + \big( \gamma(P g_c) \big)(x_0) = I_{2,1}(x_0) + \frac{1}{2} g(x_0) + I_{2,2}(x_0).
\]

For $x_0 \in \Gamma$ with $\lvert x_0' \rvert<2R_h$, by applying Proposition \ref{DbSD} again, we see that $\left\lvert I_{2,1}(x_0) \right\rvert$ can also be controlled by estimate (\ref{EI21}) and (\ref{EI21x0s}), i.e., in this case we also have that 
\[
\left\lvert I_{2,1}(x_0) \right\rvert \leq C(n) C_s(h)^3 R_h \Vert \nabla'^2 h \Vert_{L^\infty(\mathbf{R}^{n-1})} \Vert g \Vert_{L^\infty(\Gamma)}.
\]
We now estimate $I_{2,2}(x_0)$ for $x_0 \in \Gamma$ with $\lvert x_0' \rvert< 2R_h$.
By estimate (\ref{R1}) again, we have that
\begin{align*}
\int_{\left\{y \in B'_{\Gamma}(2R_h) \, \middle\vert\, \lvert x_0'-y' \rvert<1 \right\}} \left\lvert \frac{\partial E}{\partial \mathbf{n}_y}(x_0-y) \right\rvert \, d\mathcal{H}^{n-1}(y) \leq C_s(h) \int_{\lvert y' \rvert<2R_h, \lvert x_0'-y' \rvert<1} \frac{\Vert \nabla'^2 h \Vert_{L^\infty(\mathbf{R}^{n-1})}}{\lvert x_0'-y' \rvert^{n-2}} \, dy'.
\end{align*}
Since in this case $\lvert x_0' \rvert<2R_h$, $\lvert y' \rvert<2R_h$ would imply that $\lvert x_0'-y' \rvert<4R_h$, we have that
\[
\int_{\lvert y' \rvert<2R_h} \frac{1}{\lvert x_0'-y' \rvert^{n-2}} \, dy' \leq \int_{\lvert x_0'-y' \rvert<4R_h} \frac{1}{\lvert x_0'-y' \rvert^{n-2}} \, dy' \leq C(n) R_h.
\]
On the other hand, for $y \in \Gamma$ such that $\lvert y'-x_0' \rvert \geq 1$, we can straightforwardly estimate $\left\lvert \sigma(y')\right\rvert$ in $\partial E/ \partial \mathbf{n}_y$ by $\left\lvert \nabla' h(y')\right\rvert \cdot \lvert x_0'-y' \rvert + \left\lvert h(x_0') \right\rvert + \left\lvert h(y')\right\rvert$. Hence, we have that
\begin{align*}
&\int_{\left\{y \in B'_{\Gamma}(2R_h) \,\middle\vert\, \lvert x_0'-y' \rvert\geq1 \right\}} \left\lvert  \frac{\partial E}{\partial \mathbf{n}_y}(x_0-y) \right\rvert  \, d\mathcal{H}^{n-1}(y) \\
&\ \ \leq C_s(h) \int_{\mathbf{R}^{n-1}} \left\lvert \nabla' h(y')\right\rvert + \left\lvert h(y')\right\rvert \, dy' + C_s(h) \int_{\lvert x_0'-y' \rvert\geq1} \frac{\Vert h \Vert_{L^\infty(\mathbf{R}^{n-1})}}{\lvert x_0'-y' \rvert^n} \, dy' \\
&\ \ \leq C(n) C_s(h) (R_h^{n-1}+1) \Vert h \Vert_{C^1(\mathbf{R}^{n-1})}.
\end{align*}
Combining these estimates together, we see that the estimate for $\left\lvert I_{2,1}(x_0)\right\rvert + \left\lvert I_{2,2}(x_0) \right\rvert$ reads as
\begin{align*}
\left\lvert I_{2,1}(x_0)\right\rvert +\left\lvert I_{2,2}(x_0) \right\rvert \leq C(n) C_s(h)^3 \big( R_h \Vert \nabla'^2 h \Vert_{L^\infty(\mathbf{R}^{n-1})} + (R_h^{n-1} + 1) \Vert h \Vert_{C^1(\mathbf{R}^{n-1})} \big) \Vert g \Vert_{L^\infty(\Gamma)}.
\end{align*}
In order to estimate $I_1(x_0)$ for $x_0 \in \Gamma$ with $\lvert x_0' \rvert<2R_h$, we further decompose
\begin{align*}
g_2(y',h(y')) &= 1'_{B_{4 r_h}(0')}(y') g_2(y',h(y')) + 1'_{B_{4 r_h}(0')^\mathrm{c}}(y') g_2(y',h(y')) \\
&= g_{2,1}(y',h(y')) + g_{2,2}(y',h(y'))
\end{align*}
for any $y' \in \mathbf{R}^{n-1}$ where $r_h := R_h^{\frac{1}{2n}}$ and
\begin{align*}
I_1(x_0) &= \int_{B'_\Gamma(r_h)^\mathrm{c}} \frac{\partial E}{\partial \mathbf{n}_{x_0}}(x_0 - y)  g_{2,2}(y) \, d\mathcal{H}^{n-1}(y) + \int_{B'_\Gamma(r_h)^\mathrm{c}} \frac{\partial E}{\partial \mathbf{n}_{x_0}}(x_0 - y)  g_{2,1}(y) \, d\mathcal{H}^{n-1}(y) \\
&\ \ + \int_{B'_\Gamma(r_h)} \frac{\partial E}{\partial \mathbf{n}_{x_0}}(x_0 - y) g_2(y) \, d\mathcal{H}^{n-1}(y) = I_{1,1}(x_0) + I_{1,2}(x_0) + I_{1,3}(x_0).
\end{align*}

We next seek to control $I_{1,1}(x_0)$ for $x_0 \in \Gamma$ with $\lvert x_0' \rvert<2R_h$. Let $\theta_2 \in C_\mathrm{c}^\infty(\mathbf{R}^{n-1})$ be a cut-off function such that $0 \leq \theta_2 \leq 1$ in $\mathbf{R}^{n-1}$, $\theta_2 = 1$ in $B_1(0')$ and $\operatorname{supp} \theta_2 \subseteq \overline{B_2(0')}$. We then set that 
\begin{align*}
\theta_{2r_h,x_0}(y') := \theta_2\left( \frac{y' - x_0'}{r_h} \right), \quad K_{x_0}(y') := \left\{ \nabla d (x_0) \cdot \nabla E \Big( x_0 - \big( y',h(y') \big) \Big) \right\} \big( 1 - \theta_{2r_h,x_0}(y') \big)
\end{align*}
for $y' \in \mathbf{R}^{n-1}$. Since we are assuming that $2R_h < r_h$, it holds that $B_{2r_h}(x_0') \subset B_{3r_h}(0')$.
Hence, $1 - \theta_{2r_h,x_0} = 1$ in $B_{2 r_h}(x_0')^\mathrm{c}$ would imply that
\[
I_{1,1}(x_0) = \int_{\mathbf{R}^{n-1}} K_{x_0}(y') g_{2,2}\big( y',h(y') \big) \big( 1 + \big\lvert \nabla' h (y') \big\rvert^2 \big)^{\frac{1}{2}} \, dy'.
\]
Note that for any $x_0,y \in \Gamma$ such that $x_0 \neq y$, we have that
\[
\nabla d (x_0) \cdot \nabla E (x_0 - y) = - C(n) \frac{- \nabla' h (x_0') \cdot (x_0' - y') + \big( h(x_0') - h(y') \big)}{\omega(x_0') \Big( \lvert x_0' - y' \rvert^2 + \big( h(x_0') - h(y') \big)^2 \Big)^{n/2}}
\]
where $\omega(x_0') = \big( 1 + \left\lvert \nabla' h (x_0') \right\rvert^2 \big)^{1/2}$. Through some simple calculations, we can deduce by the mean value theorem that the estimate
\[
\left\lvert K_{x_0}(y') \right\rvert \leq C(n) \frac{\Vert \nabla' h \Vert_{L^\infty(\mathbf{R}^{n-1})}}{\lvert x_0' - y' \rvert^{n-1}} \leq C(n) \frac{R_h \Vert \nabla'^2 h \Vert_{L^\infty(\mathbf{R}^{n-1})}}{\lvert x_0' - y' \rvert^{n-1}}
\]
holds for any $x_0,y \in \Gamma$ and the estimate
\[
\left\lvert \nabla'_{y'} \big( \nabla d (x_0) \cdot \nabla E (x_0 - y) \big) \right\rvert \leq C(n) \frac{R_h \Vert \nabla'^2 h \Vert_{L^\infty(\mathbf{R}^{n-1})} + R_h^3 \Vert \nabla'^2 h \Vert_{L^\infty(\mathbf{R}^{n-1})}^3}{\lvert x_0' - y' \rvert^n}
\]
holds for any $x_0,y \in \Gamma$ with $x_0 \neq y$.
In addition, for $y \in \Gamma$ such that $r_h <\lvert y' - x_0'\rvert <2r_h$, we have that
\[
\left\lvert \nabla'_{y'} \theta_{2r_h,x_0}(y') \right\rvert \leq \frac{\Vert \nabla' \theta_2 \Vert_{L^\infty(\mathbf{R}^{n-1})}}{r_h} \leq \frac{2 \Vert \nabla' \theta_2 \Vert_{L^\infty(\mathbf{R}^{n-1})}}{\lvert x_0' - y' \rvert}.
\]
Hence, we see that the estimate
\[
\left\lvert \nabla'_{y'} K_{x_0}(y') \right\rvert \leq C(n) \frac{R_h \Vert \nabla'^2 h \Vert_{L^\infty(\mathbf{R}^{n-1})} + R_h^3 \Vert \nabla'^2 h \Vert_{L^\infty(\mathbf{R}^{n-1})}^3}{\lvert x_0' - y' \rvert^n}
\]
holds for any $x_0,y \in \Gamma$. By the duality relation between $\dot{H}^{\frac{1}{2}}(\mathbf{R}^{n-1})$ and $\dot{H}^{-\frac{1}{2}}(\mathbf{R}^{n-1})$, we see that $I_{1,1}(x_0)$ follows the estimate
\[
\left\lvert I_{1,1}(x_0) \right\rvert \leq \Vert \big( 1 + \big\lvert \nabla' h (\cdot') \big\rvert^2 \big)^{\frac{1}{2}} K_{x_0}(\cdot') \Vert_{\dot{H}^{\frac{1}{2}}(\mathbf{R}^{n-1})} \Vert g_{2,2}\big( \cdot', h(\cdot') \big) \Vert_{\dot{H}^{-\frac{1}{2}}(\mathbf{R}^{n-1})}.
\]
By Proposition \ref{MRHhalf} and Proposition \ref{BSNE}, we have that
\begin{align*}
&\Vert \big( 1 + \big\lvert\nabla' h (\cdot') \big\rvert^2 \big)^{\frac{1}{2}} K_{x_0}(\cdot') \Vert_{\dot{H}^{\frac{1}{2}}(\mathbf{R}^{n-1})} \leq C(n) C_s(h) C_1(h) \Vert K_{x_0}(\cdot') \Vert_{\dot{H}^{\frac{1}{2}}(\mathbf{R}^{n-1})} \\
&\ \ \leq C(n) C_s(h) C_1(h) \left( R_h^{\frac{1}{2}} \Vert \nabla'^2 h \Vert_{L^\infty(\mathbf{R}^{n-1})} + R_h^{\frac{5}{2}} \Vert \nabla'^2 h \Vert_{L^\infty(\mathbf{R}^{n-1})}^3 \right).
\end{align*}
Since $L^{\frac{2n-2}{n}}(\mathbf{R}^{n-1})$ is continuously embedded in $\dot{H}^{-\frac{1}{2}}(\mathbf{R}^{n-1})$, by estimate (\ref{gghEm}) and estimate (\ref{EHdhLp}) we see that
\[
\big\Vert 1'_{B_{4r_h}(0')}(\cdot') g\big( \cdot',h(\cdot') \big) \big\Vert_{\dot{H}^{-\frac{1}{2}}(\mathbf{R}^{n-1})} \leq C(n) C_s(h)^{n+5} C_1(h)^2 R_h^{\frac{1}{4}} \Vert g \Vert_{L^\infty(\Gamma)}.
\]
Therefore, the estimate for $I_{1,1}(x_0)$ reads as
\[
\left\lvert I_{1,1}(x_0) \right\rvert \leq C(n) C_s(h)^{n+6} C_{\ast,1}(h) \Vert g \Vert_{L^\infty(\Gamma) \cap \dot{H}^{-\frac{1}{2}}(\Gamma)}
\]
where 
\[
C_{\ast,1}(h) := C_1(h)^3 (1 + R_h^{\frac{1}{4}}) \left( R_h^{\frac{1}{2}} \Vert \nabla'^2 h \Vert_{L^\infty(\mathbf{R}^{n-1})} + R_h^{\frac{5}{2}} \Vert \nabla'^2 h \Vert_{L^\infty(\mathbf{R}^{n-1})}^3 \right).
\]

Since for $I_{1,2}(x_0)$ and $I_{1,3}(x_0)$, the integration region is bounded, $I_{1,2}(x_0)$ and $I_{1,3}(x_0)$ can be estimated in exactly the same way as $I_{2,1}(x_0) + I_{2,2}(x_0)$ in the case where $\lvert x_0' \rvert<2R_h$. As a result, here we directly give the estimate for $I_{1,2}(x_0)$ and $I_{1,3}(x_0)$ without going through what have already been done again.
The estimate for $I_{1,2}(x_0)$ and $I_{1,3}(x_0)$ reads as
\[
\left\lvert I_{1,2}(x_0) \right\rvert +\left\lvert I_{1,3}(x_0) \right\rvert \leq C(n) C_s(h)^3 C_{\ast,2}(h) \Vert g \Vert_{L^\infty(\Gamma)}.
\]
where
\[
C_{\ast,2}(h) := \big( R_h + R_h^{\frac{1}{2n}} \big) \Vert \nabla'^2 h \Vert_{L^\infty(\mathbf{R}^{n-1})} + (R_h^{n-1} + 1) \Vert h \Vert_{C^1(\mathbf{R}^{n-1})}
\]
Therefore, for $x_0 \in \Gamma$ with $\lvert x_0' \rvert < 2R_h$, by setting
\[
\big( S g \big)(x_0) = - \int_{\Gamma} \frac{\partial E}{\partial \mathbf{n}_{x_0}}(x_0-y) g(y) \, d\mathcal{H}^{n-1}(y),
\]
we obtain that
\[
\gamma \big( Q g \big) (x_0) = \frac{1}{2} g(x_0) - \big( S g \big) (x_0)
\]
with
\[
\Vert S \Vert_{\mathrm{op}} \leq C(n) C_s(h)^{n+6} \big( C_{\ast,1}(h) + C_{\ast,2}(h) \big).
\]
This completes the proof of Theorem \ref{TBIE}.
\end{proof}
%%%

%%%
\subsection{$\dot{H}^{-\frac{1}{2}}$ estimate for the trace operator $S$} % Subsection 4.4
\label{sub:L2S}

In this subsection, we assume that $\mathbf{R}_h^n$ is a perturbed $C^2$ half space with boundary $\Gamma = \partial \mathbf{R}_h^n$ and $n \geq 3$. We shall derive the $\dot{H}^{-\frac{1}{2}}$ estimate for the trace operator $S$ from its $L^{\frac{2n-2}{n}}$ estimate. We begin with the $L^p$ estimate for $S$.

\begin{lemma} \label{BSLSH}
Let $g \in L^\infty(\Gamma) \cap \dot{H}^{-\frac{1}{2}}(\Gamma)$. Then, it holds that $S g \in L^p(\Gamma)$ for any $1 \leq p < \infty$. For $1 < p < \infty$, $Sg$ satisfies the estimate
\[
\Vert S g \Vert_{L^p(\Gamma)} \leq C^\ast_1(n,p) C_s(h)^{n+7} \big( C_{\ast,1}(h) + C_{\ast,2}(h) +1 \big) R_h^{\frac{n-1}{p}} \Vert g \Vert_{L^\infty(\Gamma) \cap \dot{H}^{-\frac{1}{2}}(\Gamma)}
\]
with some specific constant $C^\ast_1(n,p)>0$ that depends on $n$ and $p$ only. For $p=1$, $Sg$ satisfies the estimate
\[
\Vert S g \Vert_{L^1(\Gamma)} \leq C^\ast_2(n) C_s(h)^{n+7} C_{\ast,3}(h) \Vert g \Vert_{L^\infty(\Gamma) \cap \dot{H}^{-\frac{1}{2}}(\Gamma)}
\]
with some specific constant $C^\ast_2(n)>0$ that depends on $n$ only and
\[
C_{\ast,3}(h) := R_h^{n-1} \big( C_{\ast,1}(h) + C_{\ast,2}(h) \big) + R_h^n \Vert \nabla'^2 h \Vert_{L^\infty(\mathbf{R}^{n-1})}.
\]
\end{lemma}
\begin{proof}
We firstly consider $x \in \Gamma$ with $\lvert x'\rvert < 3 R_h$. Since we already have the $L^\infty$ estimate for $Sg$ on $\Gamma$ according to Theorem \ref{TBIE}, the estimate
\begin{align*}
\left( \int_{B'_\Gamma(3R_h)} \left\lvert S g (x) \right\rvert^p \, d\mathcal{H}^{n-1}(x) \right)^{\frac{1}{p}} \leq C^\ast_0(n) C_s(h)^{n+6} \big( C_{\ast,1}(h) + C_{\ast,2}(h) \big) \Lambda_{3R_h}^{\frac{1}{p}} \Vert g \Vert_{L^\infty(\Gamma) \cap \dot{H}^{-\frac{1}{2}}(\Gamma)}
\end{align*}
follows naturally, where the surface area $\Lambda_{3R_h}$ is estimated by
\[
\Lambda_{3R_h} = \int_{B'_\Gamma(3R_h)} 1 \, d\mathcal{H}^{n-1}(x) = \int_{\lvert x'\rvert<3R_h} \big( 1 + \big\lvert \nabla' h (x') \big\rvert^2 \big)^{\frac{1}{2}} \, dx' \leq C(n) C_s(h) R_h^{n-1}.
\]
Hence, for any $1 \leq p <\infty$, it holds that
\begin{align*}
\left( \int_{B'_\Gamma(3R_h)} \left\lvert S g (x) \right\rvert^p \, d\mathcal{H}^{n-1}(x) \right)^{\frac{1}{p}} \leq C(n,p) C_s(h)^{n+7} \big( C_{\ast,1}(h) + C_{\ast,2}(h) \big) R_h^{\frac{n-1}{p}} \Vert g \Vert_{L^\infty(\Gamma) \cap \dot{H}^{-\frac{1}{2}}(\Gamma)}.
\end{align*}

Let $1 < p < \infty$. We then consider $x \in \Gamma$ with $\lvert x'\rvert \geq 3R_h$. 
For $y \in \Gamma$ with $\lvert y'\rvert< 2R_h$, the triangle inequality implies that $\lvert x'-y'\rvert\geq \lvert x'\rvert - 2R_h$. In this case, we deduce that
\begin{align*}
\left\lvert S g(x) \right\rvert^p &\leq C(n,p) C_s(h)^p \left( \int_{\lvert y'\rvert<2R_h} \frac{1}{\lvert x'-y'\rvert^{n-1}} \, dy' \right)^p \Vert g \Vert_{L^\infty(\Gamma)}^p \\
&\leq C(n,p) C_s(h)^p \Vert g \Vert_{L^\infty(\Gamma)}^p \frac{\left\lvert B_{2R_h}(0')\right\rvert^p}{(\lvert x'\rvert - 2R_h)^{pn-p}}.
\end{align*}
Hence, we have that
\begin{align*}
\int_{B'_\Gamma(3R_h)^\mathrm{c}} \left\lvert S g(x)\right\rvert^p \, d\mathcal{H}^{n-1}(x) \leq C(n,p) C_s(h)^{p+1} R_h^{p(n-1)} \Vert g \Vert_{L^\infty(\Gamma)}^p \int_{\lvert x'\rvert \geq 3R_h} \frac{1}{(\lvert x'\rvert- 2R_h)^{pn-p}} \, dx',
\end{align*}
where the integral on the right hand side can be estimated as
\begin{align*}
&\int_{\lvert x'\rvert\geq 3R_h} \frac{1}{(\lvert x'\rvert- 2R_h)^{pn-p}} \, dx' \leq C(n) \int_{R_h}^\infty \frac{(r+2R_h)^{n-2}}{r^{pn-p}} \,  dr \\
&\ \ \leq C(n) \sum_{i=0}^{n-2} R_h^{n-2-i} \int_{R_h}^\infty \frac{r^i}{r^{pn-p}} \, dr \leq C(n,p) R_h^{(1-p)(n-1)}.
\end{align*}
Therefore, we obtain that
\[
\left( \int_{B'_\Gamma(3R_h)^\mathrm{c}} \left\lvert S g(x) \right\rvert^p \, d\mathcal{H}^{n-1}(x) \right)^{\frac{1}{p}} \leq C(n,p) C_s(h)^2 R_h^{\frac{n-1}{p}} \Vert g \Vert_{L^\infty(\Gamma)}.
\]

For $x \in \Gamma$ with $\lvert x'\rvert\geq 3R_h$, we indeed have that $\nabla d(x) = (0,...,0,1)$. Thus, $Sg(x)$ has the form
\[
Sg(x) = C(n) \int_{\lvert y'\rvert<2R_h} \frac{h(y') \big( 1 + \lvert\nabla' h (y')\rvert^2 \big)^{\frac{1}{2}}}{\big( \lvert x'-y'\rvert^2 + h(y')^2 \big)^{\frac{n}{2}}} g\big( y',h(y') \big) \, dy'.
\]
Since $\lvert x'-y'\rvert \geq \lvert x'\rvert - 2R_h$ for any $\lvert y'\rvert < 2R_h$, $\left\lvert Sg(x)\right\rvert$ can thus be estimated as
\[
\left\lvert Sg(x)\right\rvert\leq C(n) C_s(h) \Vert h \Vert_{L^\infty(\mathbf{R}^{n-1})} \Vert g \Vert_{L^\infty(\Gamma)} \frac{\left\lvert B_{2R_h}(0')\right\rvert}{(\lvert x'\rvert- 2R_h)^n} 
\]
Hence,
\begin{align*}
\int_{B'_\Gamma(3R_h)^\mathrm{c}} \left\lvert Sg(x)\right\rvert \, d\mathcal{H}^{n-1}(x) &\leq C(n) C_s(h)^2 R_h^{n-1} \Vert h \Vert_{L^\infty(\mathbf{R}^{n-1})} \Vert g \Vert_{L^\infty(\Gamma)} \int_{\lvert x'\rvert \geq 3R_h} \frac{1}{(\lvert x'\rvert - 2R_h)^n} \, dx' \\
&\leq C(n) C_s(h) R_h^n \Vert \nabla'^2 h \Vert_{L^\infty(\mathbf{R}^{n-1})} \Vert g \Vert_{L^\infty(\Gamma)}.
\end{align*}
\end{proof}

We are now ready to state the $\dot{H}^{-\frac{1}{2}}$ estimate for the trace operator $S$.

\begin{corollary} \label{ELHmh}
For $g \in L^\infty(\Gamma) \cap \dot{H}^{-\frac{1}{2}}(\Gamma)$, we have that $S g \in \dot{H}^{-\frac{1}{2}}(\Gamma)$ satisfying
\[
\Vert Sg \Vert_{\dot{H}^{-\frac{1}{2}}(\Gamma)} \leq C^\ast_3(n) C_s(h)^{\frac{3n}{2} + 8} C_1(h) \big( C_{\ast,1}(h) + C_{\ast,2}(h) +1 \big) R_h^{\frac{n}{2}} \Vert g \Vert_{L^\infty(\Gamma) \cap \dot{H}^{-\frac{1}{2}}(\Gamma)}
\]
with some specific constant $C^\ast_3(n)>0$ that depends on $n$ only.
\end{corollary}
\begin{proof}
Since $L^{\frac{2n-2}{n}}(\Gamma)$ is continuously embedded in $\dot{H}^{-\frac{1}{2}}(\Gamma)$, by considering estimate (\ref{EHdhLp}) and Lemma \ref{BSLSH} with $p = \frac{2n-2}{n}$, we obtain Corollary \ref{ELHmh}.
\end{proof}

We define constants 
\[
C_\ast(h) := C_s(h)^{\frac{3n}{2} + 8} C_1(h) \big( C_{\ast,1}(h) + C_{\ast,2}(h) + R_h^{\frac{n}{2}} \big)
\]
and $C^\ast(n) := C^\ast_0(n) + C^\ast_3(n)$. We would like to emphasize that $C^\ast(n)$ is a specific constant that depends on dimension $n$ only and $C_\ast(h)$ is a constant that depends on the boundary function $h$ only.  
Theorem \ref{TBIE} and Corollary \ref{ELHmh} guarantee that the trace operator $S : L^\infty(\Gamma) \cap \dot{H}^{-\frac{1}{2}}(\Gamma) \to L^\infty(\Gamma) \cap \dot{H}^{-\frac{1}{2}}(\Gamma)$ is bounded linear and
\[
\Vert S \Vert_{L^\infty(\Gamma) \cap \dot{H}^{-\frac{1}{2}}(\Gamma) \to L^\infty(\Gamma) \cap \dot{H}^{-\frac{1}{2}}(\Gamma)} \leq C^\ast(n) C_\ast(h).
\]
Moreover, we can require $C_\ast(h)$ to be arbitrarily small by taking $R_h$ to be sufficiently small.
%%%%%%

%%%%%%
\section{Neumann problem with bounded data in a perturbed $C^2$ half space with small perturbation} 
% Section 5
\label{sec:NPB}

We consider the Neumann problem for the Laplace equation in a perturbed $C^2$ half space in $\mathbf{R}^n$ with $L^\infty$-initial data for $n \geq 3$. We shall begin with the half space problem. It is well-known that a solution of the Neumann problem
\begin{align}
&\begin{aligned} \label{NPHS}
	\Delta u &= 0 \quad \, \text{in}\quad \mathbf{R}_+^n \\
	\frac{\partial u}{\partial \mathbf{n}} &= g \quad\text{on}\quad \partial \mathbf{R}_+^n
\end{aligned}
\end{align}
is formally given by
\begin{align} \label{SNPHS}
u(x) = \int_{\mathbf{R}^{n-1}} N(x,y) g(y) \, d\mathcal{H}^{n-1},
\end{align}
where $N$ denotes the Neumann-Green function
\[
N(x,y) = E(x-y) + E(x'-y', x_n+y_n).
\]
Its exterior normal derivative $\partial N/ \partial \mathbf{n}_x$ for $y_n=0$ is nothing but the Poisson kernel with the parameter $x_n$. By symmetry we observe that
\[
- \frac{\partial}{\partial x_n} \int_{\mathbf{R}^{n-1}} E(x'-y', x_n) g(y') \, dy' \to \frac{1}{2} g(x')
\]
as $x_n>0$ tends to zero. Thus $u$ gives a solution to $(\ref{NPHS})$ formally. The function
\[
E \ast (\delta_{\partial \mathbf{R}_+^n} \otimes g) := \int_{\mathbf{R}^{n-1}} E(x'-y',x_n) g(y') \, dy'
\]
is called the single layer potential of $g$. 

For $g \in L^\infty(\mathbf{R}^{n-1})$, we let $\widetilde{g}(x',x_n) := g(x',0)$ for any $x \in \mathbf{R}^n$. Natrually, $\widetilde{g} \in L^\infty(\mathbf{R}^n)$. Let $1_{\mathbf{R}_+^n}$ be the characteristic function associated with the half space $\mathbf{R}_+^n$. In this case, we have that
\[
\nabla E*(\delta_{\partial \mathbf{R}_+^n} \otimes g) = \nabla\partial_{x_n} E*1_{\mathbf{R}^n_+} \widetilde{g}.
\]
Hence, by the $L^\infty$-$BMO$ estimate for the singular integral operator \cite[Theorem 4.2.7]{GraM}, we have the estimate
\begin{align} \label{BIBEH}
\big[ \nabla \big( E \ast (\delta_{\partial \mathbf{R}_+^n} \otimes g) \big) \big]_{BMO(\mathbf{R}^n)} \leq C(n) \Vert g \Vert_{L^\infty(\mathbf{R}^{n-1})}.
\end{align}
Moreover, since $- \partial_{x_n} (E \ast (\delta_{\partial \mathbf{R}_+^n} \otimes g))$ is the half of the Poisson integral, i.e.,
\[
- \partial_{x_n} \big( E \ast (\delta_{\partial \mathbf{R}_+^n} \otimes g) \big) = \frac{1}{2} \int_{\mathbf{R}^{n-1}} P_{x_n}(x'-y')g(y') \, dy',
\]
the estimate
\begin{align} \label{LIEPnS}
\big\Vert \partial_{x_n} \big( E*(\delta_{\partial \mathbf{R}_+^n} \otimes g) \big) \big\Vert_{L^\infty(\mathbf{R}_+^n)} \leq \frac{1}{2}\Vert g\Vert_{L^\infty(\mathbf{R}^{n-1})}
\end{align}
holds for any $g \in L^\infty(\mathbf{R}^{n-1})$, see \cite[Lemma 7]{GG22b}.
We are able to to establish similar estimates for the case where the domain is a perturbed $C^2$ half space.

\begin{lemma} \label{NM}
Let $\mathbf{R}_h^n$ be a perturbed $C^2$ half space of type $(K)$ with boundary $\Gamma=\partial \mathbf{R}_h^n$ and $n \geq 3$. Then,
\begin{enumerate}
\item[(i)] ($BMO$ estimate)
For all $g\in L^\infty(\Gamma)$, the estimate
\begin{equation} \label{BMOEH}
\left[ \nabla E * \big( \delta_\Gamma \otimes g \big) \right]_{BMO(\mathbf{R}^n)} \leq C(n) C_{\ref{NM},i}(h,\rho_0) \Vert g\Vert_{L^\infty(\Gamma)}
\end{equation}
holds with
\[
C_{\ref{NM},i}(h,\rho_0) := C_s(h)^2 (R_h + \rho_0 +1)^n \big( \Vert \nabla'^2 h \Vert_{L^\infty(\mathbf{R}^{n-1})} + \rho_0^{-1} \big).
\]

\item[(i\hspace{-1pt}i)] ($L^\infty$ estimate for normal component)
For all $g\in L^\infty(\Gamma) \cap \dot{H}^{-\frac{1}{2}}(\Gamma)$, the estimate
\begin{equation} \label{NCEH}
\left\Vert \nabla d \cdot \nabla E * \big( \delta_\Gamma \otimes g \big) \right\Vert_{L^\infty\big( \Gamma^\Omega_{\rho_0} \big)} \leq C(n) C_{\ref{NM},ii}(K,h,\rho_0) \Vert g\Vert_{L^\infty(\Gamma) \cap \dot{H}^{-\frac{1}{2}}(\Gamma)}
\end{equation}
holds with
\begin{align*}
&C_{\ref{NM},ii}(K,h,\rho_0) := C_s(h) \big( R_h^{n-1} + \rho_0 K + \rho_0 + 2 \big) + \rho_0 \\
&\ \ + C_s(h)^2 (2+6R_h) \Vert \nabla'^2 h \Vert_{L^\infty(\mathbf{R}^{n-1})} + C_s(h)^{n+10} C_1(h)^3 (1 + 3R_h)^{\frac{n}{2}}.
\end{align*}
\end{enumerate}
\end{lemma}

For $g \in L^\infty(\Gamma)$, the notation $E*(\delta_\Gamma\otimes g)$ in Lemma \ref{NM} means that
\[
E*(\delta_\Gamma\otimes g)(x) := \int_\Gamma E(x-y)g(y) \, d\mathcal{H}^{n-1}(y), \quad x \in \mathbf{R}^n.
\]
%%%

%%%
\subsection{$BMO$ estimate} \label{sub:BMOES} 

We follow the idea of the proof for \cite[Lemma 5 (i)]{GG22b}, which establishes the same $BMO$ estimate in the case where the domain is a bounded $C^2$ doamin.

\begin{proof}[Proof of Lemma \ref{NM} (i)]
For $g \in L^\infty(\Gamma)$, we follow the setting in subsection \ref{sub:EPgS} to decompose $g$ into the curved part $g_1$ and the straight part $g_2$ and let $g_2^H \in L^\infty(\mathbf{R}^{n-1})$ be defined as expression \eqref{HPg2}.
Since by definition we have that
\[
E \ast (\delta_\Gamma \otimes g_2) = E \ast (\delta_{\partial \mathbf{R}_+^n} \otimes g_2^H),
\]
the estimate 
\[
\big[ \nabla \big( E \ast (\delta_{\partial \mathbf{R}_+^n} \otimes g_2^H) \big) \big]_{BMO(\mathbf{R}^n)} \leq C(n) \Vert g_2^H \Vert_{L^\infty(\mathbf{R}^{n-1})} \leq C(n) \Vert g \Vert_{L^\infty(\Gamma)}
\]
follows from estimate \eqref{BIBEH}. As we are now considering the case where the boundary $\Gamma$ is uniformly $C^2$, the signed distance function $d$ is $C^2$ in $\Gamma^{\mathbf{R}^n}_{\rho_0}$, see e.g. \cite[Section 14.6]{GT}. 
We consider $\theta \in C_\mathrm{c}^\infty(\mathbf{R})$ such that $0 \leq \theta \leq 1$, $\theta(\sigma)=1$ for $\lvert\sigma\rvert \leq 1$ and $\theta(\sigma)=0$ for $\lvert\sigma\rvert \geq 2$.
Note that $\theta_d := \theta(4d/\rho_0)$ is $C^2$ in $\mathbf{R}^n$.
We extend $g_1 \in L^\infty(\Gamma)$ to $g_1^\mathrm{e} \in L^\infty\big( \Gamma^{\mathbf{R}^n}_{\rho_0/2} \big)$ by setting
\[
g_1^\mathrm{e}(x) := g_1(\pi x)
\]
for any $x \in \Gamma^{\mathbf{R}^n}_{\rho_0/2}$ with $\pi x$ denoting the unique projection of $x$ on $\Gamma$. For $x \in \Gamma^{\mathbf{R}^n}_{\rho_0/2}$, by considering the normal coordinate $x = F_{\pi x}(\eta)$ in $U_{\rho_0/2}(\pi x)$, we have that
\[
(\nabla_x d)_{F_{\pi x}} \cdot (\nabla_x g_1^\mathrm{e})_{F_{\pi x}} = \partial_{\eta_n} (g_1^\mathrm{e})_{F_{\pi x}} = 0
\]
as $(g_1^\mathrm{e})_{F_{\pi x}}(\eta', \tau_1) = (g_1^\mathrm{e})_{F_{\pi x}}(\eta', \tau_2)$ for any $\lvert\eta'\rvert< \rho_0/2$ and $\tau_1,\tau_2 \in (- \rho_0/2, \rho_0/2)$.
Here the notation $(f)_{F_{\pi x}}$ represents the composition of $f$ and $F_{\pi x}$, i.e., $(f)_{F_{\pi x}} := f \circ F_{\pi x}$.
Hence, we see that $\nabla d \cdot \nabla g_1^\mathrm{e} = 0$ in $\Gamma^{\mathbf{R}^n}_{\rho_0/2}$.

Let us consider $g_{1,\mathrm{c}}^\mathrm{e} := \theta_d g_1^\mathrm{e}$. A key observation is that
\begin{align*}	
\delta_\Gamma \otimes g_1 &= (\nabla 1_{\mathbf{R}_h^n} \cdot \nabla d) g_{1,\mathrm{c}}^\mathrm{e} = \operatorname{div} (g_{1,\mathrm{c}}^\mathrm{e} 1_{\mathbf{R}_h^n} \nabla d) - 1_{\mathbf{R}_h^n} \operatorname{div} (g_{1,\mathrm{c}}^\mathrm{e} \nabla d), \\
\operatorname{div} (g_{1,\mathrm{c}}^\mathrm{e} \nabla d) &= g_{1,\mathrm{c}}^\mathrm{e} \Delta d + \nabla d \cdot \nabla g_{1,\mathrm{c}}^\mathrm{e} = g_{1,\mathrm{c}}^\mathrm{e} \Delta d + \frac{4\theta'(4d/\rho_0)}{\rho_0} g_1^\mathrm{e}.
\end{align*}
Thus,
\[
\nabla E*(\delta_\Gamma \otimes g_1) = \nabla\operatorname{div} \big( E*( g_{1,\mathrm{c}}^\mathrm{e} 1_{\mathbf{R}_h^n} \nabla d) \big) -\nabla E* \big( 1_{\mathbf{R}_h^n} g_1^\mathrm{e} f_{\theta,\rho_0/4} \big) = I_1 + I_2
\]
where $f_{\theta,\rho_0/4} := \theta_d \Delta d + \frac{4\theta'(4d/\rho_0)}{\rho_0}$. 
By the $L^\infty$-$BMO$ estimate for the singular integral operator \cite[Theorem 4.2.7]{GraM}, the first term is estimated as
\[
[I_1]_{BMO(\mathbf{R}^n)} \leq C(n) \Vert g_{1,\mathrm{c}}^\mathrm{e} \nabla d \Vert_{L^\infty\big( \mathbf{R}_h^n \big)} \leq C(n) \Vert g\Vert_{L^\infty(\Gamma)}.
\]
Since 
\[
\operatorname{supp} (g_1^{\mathrm{e}} f_{\theta, \rho_0/4}) \subseteq U_{\mathrm{c}, \rho_0/2} := \left\{ x \in \Gamma^{\mathbf{R}^n}_{\rho_0/2} \,\middle\vert\, \left\lvert(\pi x)' \right\rvert \leq 2 R_h \right\},
\]
for $x \in \mathbf{R}^n$ with $d(x,U_{\mathrm{c}, \rho_0/2}) = \inf_{y \in U_{\mathrm{c}, \rho_0/2}} \lvert x-y\rvert < 1$ we have that
\begin{align*}
\left\lvert  I_2(x) \right\rvert  &\leq C(n) \int_{U_{\mathrm{c}, \rho_0/2}} \frac{1}{\lvert x-y \rvert^{n-1}} \, dy \Vert f_{\theta,\rho_0/4} \Vert_{L^\infty\big( U_{\mathrm{c}, \rho_0/2} \big)} \Vert g_1^\mathrm{e} \Vert_{L^\infty\big( U_{\mathrm{c}, \rho_0/2} \big)} \\
&\leq C(n) C_{\ref{NM},i}(h,\rho_0) \Vert g\Vert_{L^\infty(\Gamma)}.
\end{align*}
where
\[
C_{\ref{NM},i}(h,\rho_0) := (R_h + \rho_0 +1)^n \big( \Vert \nabla'^2 h \Vert_{L^\infty(\Gamma)} + \Vert \nabla' h \Vert_{L^\infty(\Gamma)}^2 \Vert \nabla'^2 h \Vert_{L^\infty(\Gamma)} + \rho_0^{-1}\big).
\]
For $x \in \mathbf{R}^n$ with $d(x, U_{\mathrm{c}, \rho_0/2}) = \inf_{y \in U_{\mathrm{c}, \rho_0/2}} \lvert x-y\rvert \geq 1$, same estimate above for $\left\lvert I_2(x) \right\rvert$ holds trivially as $\lvert x-y\rvert^{-(n-1)} \leq 1$ for any $y \in U_{\mathrm{c}, \rho_0/2}$.
The proof of the first part of Lemma \ref{NM} is now complete.
\end{proof}
%%%

%%%
\subsection{$L^\infty$ estimate for normal component} \label{sub:LIEnc} 

The $L^\infty$ estimate to the normal component of $\nabla E \ast (\delta_\Gamma \otimes g)$ within a small neighborhood of $\Gamma$ for $g \in L^\infty(\Gamma) \cap \dot{H}^{-\frac{1}{2}}(\Gamma)$ can be derived by almost the same argument as establishing the boundedness of the trace operator $S$ from $L^\infty(\Gamma) \cap \dot{H}^{-\frac{1}{2}}(\Gamma)$ to $L^\infty(\Gamma)$ in Theorem \ref{TBIE}.

\begin{proof}[Proof of Lemma \ref{NM} (i\hspace{-0.1pt}i)]
Let $g \in L^\infty(\Gamma) \cap \dot{H}^{-\frac{1}{2}}(\Gamma)$ and $x \in \Gamma^\Omega_{\rho_0}$. Suppose firstly that $\lvert x'\rvert \geq 3R_h$. 
By following the setting in subsection \ref{sub:EPgS} to decompose $g$ into the curved part $g_1$ and the straight part $g_2$, we have that
\begin{align*}
\nabla d(x) \cdot \big( \nabla E \ast (\delta_\Gamma \otimes g) \big) (x) &= \big( \partial_{x_n} E \ast (\delta_\Gamma \otimes g) \big) (x) \\
&= \big( \partial_{x_n} E \ast (\delta_{\partial \mathbf{R}_+^n} \otimes g_2^H) \big) (x) + \int_{B'_\Gamma(2R_h)} \big( \partial_{n} E \big) (x-y) g_1(y) \, d\mathcal{H}^{n-1}(y),
\end{align*}
where $g_2^H = T_h^{-1}(g_2)$. By estimate \eqref{LIEPnS}, we see that
\[
\big\lvert \big( \partial_{x_n} E \ast (\delta_{\partial \mathbf{R}_+^n} \otimes g_2^H) \big) (x) \big\rvert \leq \frac{1}{2} \Vert g_2^H \Vert_{L^\infty(\partial \mathbf{R}_+^n)} \leq \frac{1}{2} \Vert g \Vert_{L^\infty(\Gamma)}.
\]
Since $\lvert x'\rvert \geq 3R_h$, for any $y \in B'_\Gamma(2R_h)$ we have that $\lvert x-y \rvert\geq \lvert x'-y' \rvert \geq \lvert x'\rvert-\lvert y'\rvert \geq R_h$. Hence, 
\begin{align*}
\left\lvert  \int_{B'_\Gamma(2R_h)} \big( \partial_n E \big) (x-y) g_1(y) \, d\mathcal{H}^{n-1}(y) \right\rvert &\leq C(n) R_h^{-(n-1)} \Vert g \Vert_{L^\infty(\Gamma)} \int_{\lvert y'\rvert<2R_h} \big( 1 + \lvert \nabla' h (y')\rvert^2 \big)^{\frac{1}{2}} \, dy' \\
&\leq C(n) C_s(h) \Vert g \Vert_{L^\infty(\Gamma)}.
\end{align*}
Thus, for $x \in \Gamma^\Omega_{\rho_0}$ with $\lvert x'\rvert \geq 3R_h$, we show that
\[
\left\lvert \nabla d(x) \cdot \big( \nabla E \ast (\delta_\Gamma \otimes g) \big) (x) \right\rvert \leq C(n) C_s(h) \Vert g \Vert_{L^\infty(\Gamma)}.
\]

Next, we consider the case where $\lvert x'\rvert< 3R_h$.
In this case, we consider a modified decomposition of $g$. 
We let $g_1^\ast := 1'_{B_{1+3R_h}(0')} \cdot g$ and $g_2^\ast := g - g_1^\ast$ where $1'_{B_{1+3R_h}(0')}$ is the characteristic function associated with the open ball $B_{1+3R_h}(0')$ in $\mathbf{R}^{n-1}$.
We firstly deal with the modified curved part $g_1^\ast$. 
Note that
\begin{align*}
\nabla d (x) \cdot \big( \nabla E*(\delta_\Gamma \otimes g_1^\ast) \big)(x) &= \int_\Gamma \big( \nabla d(x) - \nabla d(y) \big) \cdot \nabla E (x-y) g_1^\ast(y) \, d\mathcal{H}^{n-1}(y) \\
&\ \ + \int_\Gamma \frac{\partial E}{\partial\mathbf{n}_y}(x-y) g_1^\ast(y) \, d\mathcal{H}^{n-1}(y) = I_1(x) + I_2(x).
\end{align*}
By Proposition \ref{DbSD}, we have that
\begin{align*}
\left\lvert I_1(x) \right\rvert 
&\leq C(n) \Vert \nabla'^2 h \Vert_{L^\infty(\mathbf{R}^{n-1})} C_s(h)^2 \Vert g \Vert_{L^\infty(\Gamma)} \int_{\lvert y'\rvert < 1+3R_h} \frac{1}{\lvert x'-y'\rvert^{n-2}} \, dy' \\
&\leq C(n) (1+6R_h) \Vert \nabla'^2 h \Vert_{L^\infty(\mathbf{R}^{n-1})} C_s(h)^2 \Vert g \Vert_{L^\infty(\Gamma)}.
\end{align*}
On the other hand, by Lemma \ref{EP} we have that
\[
\left\lvert I_2(x) \right\rvert \leq C(n) C_{\ref{EP}}(K,h,\rho_0) \Vert g \Vert_{L^\infty(\Gamma)}.
\]

Then, we deal with the modified straight part $g_2^\ast$. Let $\theta_\ast \in C_\mathrm{c}^\infty(\mathbf{R}^{n-1})$ be a cut-off function such that $0 \leq \theta_\ast \leq 1$, $\theta_\ast = 1$ in $B_{\frac{1}{2} + 3R_h}(0')$ and $\operatorname{supp} \theta_\ast \subseteq \overline{B_{1+3R_h}(0')}$.
Let $x \in \Gamma^\Omega_{\rho_0}$ with $\lvert x'\rvert<3R_h$. We define that
\[
\theta_{\ast,x}(y') := \theta_\ast(y' - x'), \quad K_x(y') := \left\{ \nabla d(x) \cdot \nabla E \Big( x - \big( y',h(y') \big) \Big) \right\} \big( 1 - \theta_{\ast,x}(y') \big).
\]
Note that
\[
\nabla d(x) \cdot \big( \nabla E \ast (\delta_\Gamma \otimes g_2^\ast) \big)(x) = \int_{\mathbf{R}^{n-1}} K_x(y') g_2^\ast \big( y',h(y') \big) \big( 1 + \big\lvert \nabla' h (y') \big\rvert^2 \big)^{\frac{1}{2}} \, dy'
\]
and for any $x \in \Gamma^\Omega_{\rho_0}$ and $y \in \Gamma$ with $x \neq y$, it holds that
\[
\nabla d (x) \cdot \nabla E (x - y) = - C(n) \frac{- \nabla' h (x') \cdot (x' - y') + \big( x_n - h(y') \big)}{\omega(x') \Big( \lvert x' - y' \rvert^2 + \big( x_n - h(y') \big)^2 \Big)^{n/2}}
\]
where $\omega(x') = \big( 1 + \big\lvert \nabla' h (x') \big\rvert^2 \big)^{\frac{1}{2}}$. 
By following similar calculations in the proof of Theorem \ref{TBIE}, we can deduce that $K_x(\cdot') \in C^1(\mathbf{R}^{n-1})$ satisfies $\operatorname{supp} K_x(\cdot') \subset B_{\frac{1}{2} + 3R_h}(x')^{\mathrm{c}}$,
\[
\left\lvert K_x(y')\right\rvert \cdot \lvert x'-y'\rvert^{n-1} \leq C(n) C_s(h), \quad \left\lvert\nabla'_{y'} K_x (y')\right\rvert \cdot \lvert x'-y'\rvert^n \leq C(n) C_s(h)^2.
\]
Hence, by the duality relation between $\dot{H}^{\frac{1}{2}}(\mathbf{R}^{n-1})$ and $\dot{H}^{-\frac{1}{2}}(\mathbf{R}^{n-1})$, Proposition \ref{MRHhalf}, Proposition \ref{BSNE}, estimate (\ref{gghEm}) and estimate (\ref{EHdhLp}), we can deduce that
\begin{align*} 
\big\lvert \nabla d(x) \cdot \big( \nabla E \ast (\delta_\Gamma \otimes g_2^\ast) \big) (x) \big\rvert &\leq \Vert \big( 1 + \big\lvert \nabla' h (\cdot') \big\rvert^2 \big)^{\frac{1}{2}} K_x(\cdot') \Vert_{\dot{H}^{\frac{1}{2}}(\mathbf{R}^{n-1})} \Vert g_2^\ast\big( \cdot', h(\cdot') \big) \Vert_{\dot{H}^{-\frac{1}{2}}(\mathbf{R}^{n-1})} \\
&\leq C(n) C_s(h)^{n+10} C_1(h)^3 (1 + 3R_h)^{\frac{n}{2}} \Vert g \Vert_{L^\infty(\Gamma) \cap \dot{H}^{-\frac{1}{2}}(\Gamma)}.
\end{align*}
\end{proof}

We would like to emphasize that it is insufficient to obtain an $L^\infty$ estimate for the normal component of $\nabla E \ast (\delta_\Gamma \otimes g)$ in a small neighborhood of $\Gamma$ if we only assume that $g \in L^\infty(\Gamma)$. 
Let $B$ be a ball centered at $0$ with radius $r_B$ such that $B'_\Gamma(2R_h) \subset B/2$.
By almost the same argument as in the proof of Lemma \ref{NM} (ii), we can see that if $x \in \Gamma_{\rho_0}^\Omega$ with $\lvert x'\rvert \geq r_B/2$, then we have the estimate
\[
\left\lvert \nabla d(x) \cdot \big( \nabla E \ast (\delta_\Gamma \otimes g) \big) (x) \right\rvert \leq C(K,R_\ast,R_h) \lVert g \rVert_{L^\infty(\Gamma)}.
\] 
In addition, if $x \in \Gamma_{\rho_0}^\Omega$ with $\lvert x' \rvert < r_B/2$, we decompose $g = g_c + g_s$ where
\begin{align*}
g_c\big( x',h(x') \big) := 1'_{B_{r_B}(0')}(x') g\big( x',h(x') \big), \quad g_s\big( x', h(x') \big) := g\big( x', h(x') \big) - g_c\big( x', h(x') \big)
\end{align*}
for any $x' \in \mathbf{R}^{n-1}$.
Since $g_c$ is compactly supported, we also have that
\[
\left\lvert \nabla d(x) \cdot \big( \nabla E \ast (\delta_\Gamma \otimes g_c) \big) (x) \right\rvert \leq C(K,R_\ast,R_h) \lVert g \rVert_{L^\infty(\Gamma)}.
\]
The main barrier comes from the contribution of $g_s$ in the case that $\lvert x'\rvert < r_B/2$.

\begin{proposition} \label{RfHdmh}
For $1 \leq j \leq n-1$, there does not exist a constant $C>0$ such that
\[
\left\lvert (\partial_j E) \ast \big( \delta_{\mathbf{R}^{n-1}} \otimes g \big) (x',0) \right\rvert \leq C \lVert g \rVert_{L^\infty(\mathbf{R}^{n-1})}
\]
for any $\lvert x' \rvert < r_B/2$ and $g \in L^\infty(\mathbf{R}^{n-1})$ with $\operatorname{supp} g \subseteq B_{r_B}(0')^\mathrm{c}$.
\end{proposition}
\begin{proof}
Let $1 \leq j \leq n-1$. Note that 
\[
(\partial_j E) \ast \big( \delta_{\mathbf{R}^{n-1}} \otimes g \big) (x',0) = C(n) \int_{\mathbf{R}^{n-1}} \frac{x_j - y_j}{\lvert x'-y' \rvert^n} g(y') \, dy' = C(n) R_j(g)
\]
where $R_j(g)$ represents the $j$-th Riesz transform of $g$.
Let $\theta_2 \in C_\mathrm{c}^\infty(\mathbf{R}^{n-1})$ be a cut-off function that satisfies $0 \leq \theta_2 \leq 1$, $\theta_2 = 1$ in $B_1(0')$ and $\operatorname{supp} \theta_2 \subseteq \overline{B_2(0')}$.
We set $\psi_{r_B/4}(z') := \theta_2\big( \frac{4 z'}{r_B} \big)$, $\phi_{r_B}(z') := 1 - \theta_2(z'/r_B)$, $\psi_{r_B/16}(z') := \theta_2\big( \frac{16 z'}{r_B} \big)$ and
\[
R_j^\ast(z') := \big( 1 - \psi_{r_B/16}(z') \big) \cdot \frac{z_j}{\lvert z' \rvert^n}
\]
for any $z' \in \mathbf{R}^{n-1}$ and $1 \leq j \leq n-1$.

We assume the contrary of Proposition \ref{RfHdmh}. Suppose that there exist a constant $C'>0$ such that 
\[
\left\lvert R_j (g) (x') \right\rvert  \leq C' \lVert g \rVert_{L^\infty(\mathbf{R}^{n-1})}
\]
for any $\lvert x' \rvert< r_B/2$ and $g \in L^\infty(\mathbf{R}^{n-1})$ with $\operatorname{supp} g \subseteq B_{r_B}(0')^\mathrm{c}$. As a consequence, the estimate
\begin{align} \label{rResiz}
\lVert P(g) \rVert_{L^\infty(\mathbf{R}^{n-1})} \leq C' \lVert g \rVert_{L^\infty(\mathbf{R}^{n-1})}
\end{align}
holds for any $g \in L^\infty(\mathbf{R}^{n-1})$ where
\[
P(g)(x') := \psi_{r_B/4}(x') \Big( R_j^\ast \ast \big( \phi_{r_B} g \big) \Big)(x'), \quad x' \in \mathbf{R}^{n-1}.
\]
With respect to $f \in L^1(\mathbf{R}^{n-1})$, we can define the adjoint operator of $P$ by
\[
P^\ast(f)(x') := \phi_{r_B}(x') \Big( R_j^\ast \ast \big( \psi_{r_B/4} f \big) \Big)(x'), \quad x' \in \mathbf{R}^{n-1}.
\]
Estimate (\ref{rResiz}) implies that $P^\ast$ is a bounded linear operator which maps $L^1(\mathbf{R}^{n-1})$ to $L^1(\mathbf{R}^{n-1})$, i.e., it holds that
\begin{align} \label{aResiz}
\lVert P^\ast(f) \rVert_{L^1(\mathbf{R}^{n-1})} \leq C' \lVert f \rVert_{L^1(\mathbf{R}^{n-1})}
\end{align}
for any $f \in L^1(\mathbf{R}^{n-1})$.
For $t>0$, we consider the Gaussian function
\[
f_t(z') := \frac{1}{(4 \pi t)^{\frac{n-1}{2}}} \mathrm{e}^{- \frac{\lvert z'\rvert^2}{4t}}; \quad z' \in \mathbf{R}^{n-1}.
\]
Since $R_j^\ast(\cdot') \psi_{r_B/4}(x' - \cdot') \in C_\mathrm{c}^\infty(\mathbf{R}^{n-1})$ and 
\[
\underset{t \to 0}{\operatorname{lim}} \, f_t(x' - y') = \delta(x' - y')
\]
in the sense of distributions, we see that
\[
\underset{t \to 0}{\operatorname{lim}} \, P^\ast(f_t)(x') = \phi_{r_B}(x') R_j^\ast(x').
\]
Since $\lVert f_t \rVert_{L^1(\mathbf{R}^{n-1})} = 1$ for any $t>0$, estimate (\ref{aResiz}) implies that for any $t>0$, it holds that
\[
\lVert P^\ast(f_t) \rVert_{L^1(\mathbf{R}^{n-1})} \leq C'.
\]
Hence, by the Bolzano-Weierstrass theorem, there exists a sequence $\{ t_m \}_{m \in \mathbf{N}}$ which converges to zero so that the sequence
\[
\left\{ \lVert P^\ast(f_{t_m}) \rVert_{L^1(\mathbf{R}^{n-1})} \right\}_{m \in \mathbf{N}}
\]
is convergent. By Fatou's lemma, we can then conclude that
\[
\lVert \phi_{r_B} R_j^\ast \rVert_{L^1(\mathbf{R}^{n-1})} \leq \underset{t_m \to 0}{\operatorname{lim}} \, \lVert P^\ast(f_{t_m}) \rVert_{L^1(\mathbf{R}^{n-1})} \leq C'.
\]
However, for $\lvert x' \rvert \geq 2r_B$ we have that $\phi_{r_B}(x') R_j^\ast(x') = x_j / \lvert x' \rvert^n$, which is clearly not $L^1$ integrable in the region $\left\{x' \in \mathbf{R}^{n-1} \,\middle\vert\, \lvert x' \rvert \geq 2r_B \right\}$.
We reach a contradiction.
\end{proof}
%%%

%%%
\subsection{$L^2$ estimate for the gradient of the single layer potential} \label{sub:L2ESP}

We begin with the half space case.

\begin{proposition} \label{L2EHS}
Let $n \geq 3$. For any $g \in \dot{H}^{-\frac{1}{2}}(\mathbf{R}^{n-1})$, the estimate
\[
\big\Vert \nabla E \ast \big( \delta_{\partial \mathbf{R}_+^n} \otimes g \big) \big\Vert_{L^2(\mathbf{R}_+^n)} = C(n) \Vert g \Vert_{\dot{H}^{-\frac{1}{2}}(\mathbf{R}^{n-1})}
\]
holds with some constant $C(n)>0$ that depends on dimension $n$ only.
\end{proposition}
\begin{proof}
We consider the partial Fourier transform of $E$ with respect to $x'$, i.e., we let
\[
\widehat{E}'(\xi',x_n) := \int_{\mathbf{R}^{n-1}} \mathrm{e}^{- i \xi' \cdot x'} E(x',x_n) \, dx'.
\]
Since $E(x',x_n)$ is radial symmetric in $\mathbf{R}^{n-1}$ for any fixed $x_n>0$, $\widehat{E}'(\xi',x_n)$ can be calculated by the Hankel transform of order $\frac{n-3}{2}$ of the function $r^{\frac{n-3}{2}} E(r,x_n)$ where 
\[
E(r,x_n) := \frac{C(n)}{(r^2 + x_n^2)^{\frac{n-2}{2}}},
\]
i.e., we have that
\[
\widehat{E}'(\xi',x_n) =\lvert \xi'\rvert^{\frac{3-n}{2}} \int_0^\infty r^{\frac{n-1}{2}} E(r,x_n) J_{\frac{n-3}{2}}(r\lvert \xi'\rvert) \, dr = C(n) \frac{\mathrm{e}^{- x_n \lvert \xi'\rvert}}{\lvert \xi'\rvert}
\]
where $J_{\frac{n-3}{2}}$ is the Bessel function of the first kind of order $\frac{n-3}{2}$, see e.g. \cite[Formula 6.565.2]{GR}.
Then by the Fourier-Plancherel formula, we obtain that
\begin{align*}
&\big\lVert \nabla' E \ast \big( \delta_{\partial \mathbf{R}_+^n} \otimes g \big) \big\rVert_{L^2(\mathbf{R}_+^n)}^2
= \int_0^\infty \int_{\mathbf{R}^{n-1}} \big\lvert \nabla' E \ast \big( \delta_{\partial \mathbf{R}_+^n} \otimes g \big) (x',x_n) \big\rvert^2 \, dx' \, dx_n \\
&\ \ = C(n) \int_{\mathbf{R}^{n-1}} \big\lvert \widehat{g}'(\xi') \big\rvert^2 \int_0^\infty \mathrm{e}^{-2 x_n \lvert\xi'\rvert} \, dx_n \, d\xi' = C(n) \Vert g \Vert_{\dot{H}^{-\frac{1}{2}}(\mathbf{R}^{n-1})}^2
\end{align*}
and
\begin{align*}
&\big\lVert \partial_{x_n} E \ast \big( \delta_{\partial \mathbf{R}_+^n} \otimes g \big) \big\rVert_{L^2(\mathbf{R}_+^n)}^2 = \int_0^\infty \int_{\mathbf{R}^{n-1}} \big\lvert \partial_{x_n} E \ast \big( \delta_{\partial \mathbf{R}_+^n} \otimes g \big) (x',x_n) \big\rvert^2 \, dx' \, dx_n \\
&\ \ = C(n) \int_{\mathbf{R}^{n-1}} \big\lvert\widehat{g}'(\xi') \big\rvert^2 \int_0^\infty \mathrm{e}^{-2 x_n \lvert\xi'\rvert} \, dx_n \, d\xi' = C(n) \lVert g \rVert_{\dot{H}^{-\frac{1}{2}}(\mathbf{R}^{n-1})}^2.
\end{align*}
\end{proof}

We then generalize this result to arbitrary perturbed $C^2$ half space $\mathbf{R}_h^n$.

\begin{lemma} \label{L2EPH}
Let $\mathbf{R}_h^n$ be a perturbed $C^2$ half space with boundary $\Gamma = \partial \Omega$ and $n \geq 3$. For any $g \in L^\infty(\Gamma) \cap \dot{H}^{-\frac{1}{2}}(\Gamma)$,
the estimate
\[
\big\Vert \nabla E \ast \big( \delta_\Gamma \otimes g \big) \big\Vert_{L^2(\mathbf{R}^n)} \leq C(n) C_{\ref{L2EPH}}(h,\rho_0) \Vert g \Vert_{L^\infty(\Gamma) \cap \dot{H}^{-\frac{1}{2}}(\Gamma)} 
\]
holds with 
\begin{align*}
C_{\ref{L2EPH}}(h,\rho_0) &:= \rho_0^{\frac{1}{2}} R_h^{\frac{n-1}{2}} + \rho_0^{\frac{n+2}{2n}} C_s(h)^2 R_h^{\frac{n^2+n-2}{2n}} \big( \Vert \nabla'^2 h \Vert_{L^\infty(\mathbf{R}^{n-1})} + \rho_0^{-1} \big) \\
&\ \ + C_s(h)^{n+5} C_1(h)^2 \big( 1+R_h^{\frac{n}{2}} \big).
\end{align*}
\end{lemma}
\begin{proof}
Let $g \in L^\infty(\Gamma) \cap \dot{H}^{-\frac{1}{2}}(\Gamma)$.
Following the setting in subsection \ref{sub:EPgS}, we decompose $g$ into the curved part $g_1 = 1'_{B_{2R_h}(0')} g$ and the straight part $g_2 = g - g_1$.
Since $g \in L^\infty(\Gamma)$ and $L^{\frac{2n-2}{n}}(\Gamma)$ is continuously embedded in $\dot{H}^{-\frac{1}{2}}(\Gamma)$, by estimate (\ref{EHdhLp}) we see that $g_1 \in \dot{H}^{-\frac{1}{2}}(\Gamma)$ satisfies
\begin{align*}
\Vert g_1 \Vert_{\dot{H}^{-\frac{1}{2}}(\Gamma)} &\leq C(n) C_s(h)^{\frac{n}{2}+1} C_1(h) \Vert g_1 \Vert_{L^\frac{2n-2}{n}(\Gamma)} \\
&\leq C(n) C_s(h)^{\frac{n}{2}+2} C_1(h) R_h^{\frac{n}{2}} \Vert g \Vert_{L^\infty(\Gamma)}.
\end{align*}
Hence, it holds that both $g_1, g_2 \in L^\infty(\Gamma) \cap \dot{H}^{-\frac{1}{2}}(\Gamma)$. 
With respect to $g_2$, let $g_2^H$ be defined by expression \eqref{HPg2}.
Since $T_h(g_2^H) = g_2$ and the mapping $T_h : \dot{H}^{-\frac{1}{2}}(\mathbf{R}^{n-1}) \to \dot{H}^{-\frac{1}{2}}(\Gamma)$ is an isomorphism, by estimate \eqref{gghEm} that $g_2^H \in \dot{H}^{-\frac{1}{2}}(\mathbf{R}^{n-1})$ satisfies
\begin{equation} \label{HmEg2H}
\begin{split}
\Vert g_2^H \Vert_{\dot{H}^{-\frac{1}{2}}(\mathbf{R}^{n-1})} &\leq C(n) C_s(h)^{\frac{n}{2}+3} C_1(h) \Vert g_2 \Vert_{\dot{H}^{-\frac{1}{2}}(\Gamma)} \\
&\leq C(n) C_s(h)^{n+5} C_1(h)^2 \big( 1+R_h^{\frac{n}{2}} \big) \Vert g \Vert_{L^\infty(\Gamma) \cap \dot{H}^{-\frac{1}{2}}(\Gamma)}.
\end{split}
\end{equation}

Next, we follow the proof of Lemma \ref{NM} (i) to estimate the $L^2$ norm of $\nabla E \ast \big( \delta_\Gamma \otimes g_1 \big)$. 
We consider $\theta \in C_\mathrm{c}^\infty(\mathbf{R})$ such that $0 \leq \theta \leq 1$, $\theta(\sigma)=1$ for $\lvert\sigma\rvert \leq 1$ and $\theta(\sigma)=0$ for $\lvert\sigma\rvert \geq 2$.
We let $\theta_d := \theta(4d/\rho_0)$ where $d$ is the signed distance function defined by expression (\ref{SDF}). 
Note that this $\theta_d$ is $C^2$ in $\mathbf{R}^n$.
We extend $g_1 \in L^\infty(\Gamma)$ to $g_1^\mathrm{e} \in L^\infty\big( \Gamma^{\mathbf{R}^n}_{\rho_0/2} \big)$ by setting $g_1^\mathrm{e}(x) := g_1(\pi x)$ for any $x \in \Gamma^{\mathbf{R}^n}_{\rho_0/2}$ with $\pi x$ denoting the unique projection of $x$ on $\Gamma$. 
It holds that $\nabla d \cdot \nabla g_1^\mathrm{e}=0$ in $\Gamma^{\mathbf{R}^n}_{\rho_0/2}$.
We then set $g_{1,\mathrm{c}}^\mathrm{e} := \theta_d g_1^\mathrm{e}$. 
For any $x \in \mathbf{R}^n$, we have that
\begin{align*}
\nabla E \ast \big( \delta_\Gamma \otimes g_1\big)(x) &= \nabla\operatorname{div} \big(E*( g_{1,\mathrm{c}}^\mathrm{e} 1_{\mathbf{R}_h^n} \nabla d) \big)(x) -\nabla E* \big( 1_{\mathbf{R}_h^n} g_1^\mathrm{e} f_{\theta,\rho_0/4} \big) (x) \\
&= I_1(x) + I_2(x)
\end{align*}
where $f_{\theta,\rho_0/4} := \theta_d \Delta d + \frac{4\theta'(4d/\rho_0)}{\rho_0}$. Since $\nabla \operatorname{div} E$ is bounded in $L^p$ for any $1<p<\infty$, see e.g. \cite[Theorem 5.2.7 and Theorem 5.2.10]{Gra}, we can deduce that
\[
\Vert I_1 \Vert_{L^2(\mathbf{R}^n)} \leq C \Vert g_{1,\mathrm{c}}^\mathrm{e} 1_{\mathbf{R}_h^n} \nabla d \Vert_{L^2(\mathbf{R}^n)} \leq C(n) \rho_0^{\frac{1}{2}} R_h^{\frac{n-1}{2}} \Vert g \Vert_{L^\infty(\Gamma)}
\]
as $\operatorname{supp} g_{1,\mathrm{c}}^{\mathrm{e}} \subset \left\{ x \in \Gamma^{\mathbf{R}^n}_{\rho_0 /2} \,\middle\vert\, \left\lvert(\pi x)'\right\rvert < 2R_h \right\}$.
 Since $\nabla E(x)$ is an integration kernel that is dominated by $C(n) \lvert x \rvert^{1-n}$ for $x \in \mathbf{R}^n \setminus \{0\}$, by the famous Hardy-Littlewood-Sobolev inequality, see e.g. \cite[Theorem 1.7]{BCD}, we have that
\[
\Vert I_2 \Vert_{L^2(\mathbf{R}^n)} \leq C(n) \Vert 1_{\mathbf{R}_h^n} g_1^\mathrm{e} f_{\theta,\rho_0/4} \Vert_{L^r(\mathbf{R}^n)}
\]
where $r= \frac{2n}{n+2}$. Since $\operatorname{supp} g_1^{\mathrm{e}} f_{\theta,\rho_0 /4} \subset \left\{ x \in \Gamma^{\mathbf{R}^n}_{\rho_0 /2} \,\middle\vert\, \left\lvert(\pi x)'\right\rvert< 2R_h \right\}$, we deduce that
\begin{align*}
&\Vert 1_{\mathbf{R}_h^n} g_1^{\mathrm{e}} f_{\theta,\rho_0 /4} \Vert_{L^r(\mathbf{R}^n)} \\
&\ \ \leq C(n) \rho_0^{\frac{1}{r}} R_h^{\frac{n-1}{r}} \left( \Vert \nabla'^2 h \Vert_{L^\infty(\mathbf{R}^{n-1})} \big( 1 + \Vert \nabla' h \Vert_{L^\infty(\mathbf{R}^{n-1})}^2 \big) + \rho_0^{-1} \right) \Vert g \Vert_{L^\infty(\Gamma)}.
\end{align*}
Hence, we obtain the $L^2$ estimate for $g_1$, i.e., 
\[
\big\Vert \nabla E \ast \big( \delta_\Gamma \otimes g_1 \big) \big\Vert_{L^2(\mathbf{R}^n)} \leq C(n) C_{\ref{L2EPH},1}(h,\rho_0) \Vert g \Vert_{L^\infty(\Gamma)}.
\]
where
\[
C_{\ref{L2EPH},1}(h,\rho_0) := \rho_0^{\frac{1}{2}} R_h^{\frac{n-1}{2}} + \rho_0^{\frac{n+2}{2n}} C_s(h)^2 R_h^{\frac{n^2+n-2}{2n}} \big( \Vert \nabla'^2 h \Vert_{L^\infty(\mathbf{R}^{n-1})} + \rho_0^{-1} \big).
\]

The $L^2$ estimate of $\nabla E \ast \big( \delta_\Gamma \otimes g_2 \big)$ in $\mathbf{R}_+^n$ follows directly from Proposition \ref{L2EHS} and estimate (\ref{HmEg2H}). We have that
\begin{align*}
\big\Vert \nabla E \ast \big(\delta_{\partial \mathbf{R}_+^n} \otimes g_2^H \big) \big\Vert_{L^2(\mathbf{R}_+^n)} &= C(n) \Vert g_2^H \Vert_{\dot{H}^{-\frac{1}{2}}(\mathbf{R}^{n-1})} \\
&\leq C(n) C_s(h)^{n+5} C_1(h)^2 \big( 1+R_h^{\frac{n}{2}} \big) \Vert g \Vert_{L^\infty(\Gamma) \cap \dot{H}^{-\frac{1}{2}}(\Gamma)}.
\end{align*}
Note that for any $x \in \mathbf{R}^n$, it holds that
\[
\nabla E \ast \big( \delta_\Gamma \otimes g_2 \big) (x) = \nabla E \ast \big( \delta_{\partial \mathbf{R}_+^n} \otimes g_2^H \big) (x).
\]
Moreover, for any $x = (x',x_n) \in \mathbf{R}_+^n$ we have that
\[
\left\lvert \nabla E \ast \big( \delta_{\partial \mathbf{R}_+^n} \otimes g_2^H \big) (x',-x_n) \right\rvert  = \left\lvert \nabla E \ast \big( \delta_{\partial \mathbf{R}_+^n} \otimes g_2^H \big) (x',x_n) \right\rvert.
\]
Therefore, the $L^2$ estimate of $\nabla E \ast \big( \delta_\Gamma \otimes g_2 \big)$ in $\mathbf{R}^n$ reads as
\begin{align*}
\big\Vert \nabla E \ast \big( \delta_\Gamma \otimes g_2 \big) \big\Vert_{L^2(\mathbf{R}^n)} &= 2 \big\Vert \nabla E \ast \big(\delta_{\partial \mathbf{R}_+^n} \otimes g_2^H \big) \big\Vert_{L^2(\mathbf{R}_+^n)} \\
&\leq C(n) C_s(h)^{n+5} C_1(h)^2 \big( 1+R_h^{\frac{n}{2}} \big) \Vert g \Vert_{L^\infty(\Gamma) \cap \dot{H}^{-\frac{1}{2}}(\Gamma)}.
\end{align*}
\end{proof}
%%%

%%%
\subsection{Solution to the Neumann problem} \label{sub:SoNP}

Let $\mathbf{R}_h^n$ be a perturbed $C^2$ half space with boundary $\Gamma = \partial \mathbf{R}_h^n$ and $n \geq 3$. We further assume that $\mathbf{R}_h^n$ has small perturbation, i.e., we require that the boundary function $h \in C_c^2(\mathbf{R}^{n-1})$ satisfies
\[
C_\ast(h) < \frac{1}{2 C^\ast(n)}
\]
where $C^\ast(n)$ is a specific constant that depends on dimension $n$ only.
Under this setting, we are able to construct a solution to Neumann problem (\ref{1NP}).

\begin{proof}[Proof of Lemma \ref{EN}]
Let $g \in L^\infty(\Gamma) \cap \dot{H}^{-\frac{1}{2}}(\Gamma)$.
By Corollary \ref{ELHmh} and Theorem \ref{TBIE}, for any $i \in \mathbf{N}$ we have that
\[
\Vert (2S)^i g \Vert_{L^\infty(\Gamma) \cap \dot{H}^{-\frac{1}{2}}(\Gamma)} \leq 2^i C^\ast(n)^i C_\ast(h)^i \Vert g \Vert_{L^\infty(\Gamma) \cap \dot{H}^{-\frac{1}{2}}(\Gamma)}.
\]
Since we are now assuming that $2 C^\ast(n) C_\ast(h) < 1$, the operator $I - 2S$, which is bounded linear from $L^\infty(\Gamma) \cap \dot{H}^{-\frac{1}{2}}(\Gamma)$ to $L^\infty(\Gamma) \cap \dot{H}^{-\frac{1}{2}}(\Gamma)$, admits a well-defined inverse constructed by the Neumann series  
\[
(I - 2S)^{-1} := \sum_{i=0}^\infty (2S)^i
\]
in the sense that $(I-2S)^{-1} : L^\infty(\Gamma) \cap \dot{H}^{-\frac{1}{2}}(\Gamma) \to L^\infty(\Gamma) \cap \dot{H}^{-\frac{1}{2}}(\Gamma)$ is also bounded linear.
For $g \in L^\infty(\Gamma) \cap \dot{H}^{-\frac{1}{2}}(\Gamma)$, we have that
\begin{align*}
\Vert (I - 2S)^{-1} g \Vert_{L^\infty(\Gamma) \cap \dot{H}^{-\frac{1}{2}}(\Gamma)} \leq \left( \sum_{i=0}^\infty 2^i C^\ast(n)^i C_\ast(h)^i \right) \Vert g \Vert_{L^\infty(\Gamma) \cap \dot{H}^{-\frac{1}{2}}(\Gamma)} \leq \frac{\Vert g \Vert_{L^\infty(\Gamma) \cap \dot{H}^{-\frac{1}{2}}(\Gamma)}}{1 - 2 C^\ast(n) C_\ast(h)}.
\end{align*}
Hence, with respect to $g \in L^\infty(\Gamma) \cap \dot{H}^{-\frac{1}{2}}(\Gamma)$, we claim that the solution to Neumann problem (\ref{1NP}) can be constructed as
\[
u(x) = E \ast \Big( \delta_\Gamma \otimes \big( 2(I - 2S)^{-1} g \big) \Big)(x),  \quad x \in \mathbf{R}_h^n.
\]
A simple check ensures that $u$ satisfies Neumann problem (\ref{1NP}) formally. It is sufficient to establish the $vBMOL^2$ estimate for $\nabla u$.
The $vBMO^{\infty, \rho_0}$-norm for $\nabla u$ in $\mathbf{R}_h^n$ is guaranteed by Lemma \ref{NM}. By estimate (\ref{BMOEH}) and estimate (\ref{NCEH}), we have that
\begin{align*}
\Vert \nabla u \Vert_{vBMO^{\infty,\rho_0}\big( \mathbf{R}_h^n \big)} \leq \frac{C(n) C_{\ref{NM}}(K,h,\rho_0)}{1 - 2C^\ast(n) C_\ast(h)} \Vert g \Vert_{L^\infty(\Gamma) \cap \dot{H}^{-\frac{1}{2}}(\Gamma)} 
\end{align*}
where $C_{\ref{NM}}(K,h,\rho_0) := C_{\ref{NM},i}(h,\rho_0) + C_{\ref{NM},ii}(K,h,\rho_0)$. The $L^2$ estimate of $\nabla u$ follows directly from Lemma \ref{L2EPH}, we have that
\begin{align*}
\Vert \nabla u \Vert_{L^2\big( \mathbf{R}_h^n \big)} \leq C(n) C_{\ref{L2EPH}}(h,\rho_0) \Vert (I - 2S)^{-1} g \Vert_{L^\infty(\Gamma) \cap \dot{H}^{-\frac{1}{2}}(\Gamma)} \leq \frac{C(n) C_{\ref{L2EPH}}(h,\rho_0)}{1 - 2C^\ast(n) C_\ast(h)} \Vert g \Vert_{L^\infty(\Gamma) \cap \dot{H}^{-\frac{1}{2}}(\Gamma)}.
\end{align*}
This completes the proof of Lemma \ref{EN}.
\end{proof}
%%%%%%

%%%%%%
\section*{Acknowledgement}
The first author was partly supported by the Japan Society for the Promotion of Science through grants No. 19H00639 (Kiban A), No. 18H05323 (Kaitaku), No. 17H01091 (Kiban A) and by Arithmer Inc. and Daikin Industries Ltd. through a collaborative grant.
%
%\paragraph{Paragraph headings} Use paragraph headings as needed.
%\begin{acknowledgements}
%General acknowledgments should be placed at the end of the article.
%\end{acknowledgements}




%%%%%%
\end{document}